\documentclass[11pt,a4paper,leqno]{article}
\usepackage[latin1]{inputenc}
\usepackage{amsmath,amsthm}
\usepackage{amsfonts,amssymb,latexsym}
\usepackage{graphicx}
\usepackage[french]{babel}
\usepackage[all]{xy}

\DeclareMathOperator{\et}{et}

\DeclareMathOperator{\rg}{rg}

\DeclareMathOperator{\Gal}{Gal}

\DeclareMathOperator{\Lip}{Lip}

\DeclareMathOperator{\eff}{eff}

\DeclareMathOperator{\Val}{Val}

\DeclareMathOperator{\PGCD}{pgcd}

\DeclareMathOperator{\primitifs}{primitifs}

\DeclareMathOperator{\Pic}{Pic}

\DeclareMathOperator{\card}{card}

\DeclareMathOperator{\Vol}{Vol}
\newcommand{\CC}{\mathbf{C}}
\newcommand{\RR}{\mathbf{R}}
\newcommand{\ZZ}{\mathbf{Z}}
\newcommand{\NN}{\mathbf{N}}
\newcommand{\QQ}{\mathbf{Q}}
\newcommand{\xx}{\boldsymbol{x}}

\newcommand{\y}{\mathbf{y}}

\newcommand{\yy}{\boldsymbol{y}}
\newcommand{\zz}{\boldsymbol{z}}

\newcommand{\rr}{\boldsymbol{r}}
\newcommand{\ww}{\boldsymbol{w}}
\renewcommand{\tt}{\boldsymbol{t}}

\newcommand{\aalpha}{\boldsymbol{\alpha}}

\newcommand{\0}{\boldsymbol{0}}

\newcommand{\uu}{\boldsymbol{u}}
\newcommand{\vv}{\boldsymbol{v}}
\newcommand{\bb}{\boldsymbol{b}}
\newcommand{\ra}{\rightarrow}
\newcommand{\mt}{\mapsto}

\newcommand{\PP}{\mathbf{P}}

\renewcommand{\AA}{\mathbf{A}}
\newcommand{\OO}{\mathcal{O}}

\newcommand{\BB}{\mathcal{B}}

 \newtheorem{thm}{Th\'eor\`eme}[section]
 \newtheorem{prop}[thm]{Proposition}
 \newtheorem{lemma}[thm]{Lemme}
 \newtheorem{cor}[thm]{Corollaire}
 \newtheorem{Def}[thm]{D\'efinition}
  \newtheorem{rem}[thm]{Remarque}

  \usepackage[refpage]{nomencl}

\makenomenclature

\usepackage{makeidx}
\makeindex

\begin{document}

\title{\underline{POINTS DE HAUTEUR BORN\'EE} \\ \underline{ SUR LES HYPERSURFACES LISSES} \\
\underline{DE $ \PP^{n}\times \PP^{n}\times\PP^{n} $}}

\maketitle

\begin{abstract}
Nous d\'emontrons ici la conjecture de Batyrev et Manin pour le nombre de points de hauteur born\'ee de certaines hypersurfaces de l'espace triprojectif de tridegr\'e $ (1,1,1) $. La constante intervenant dans le r\'esultat final est bien celle conjectur\'ee par Peyre. La m\'ethode utilis\'ee est fortement inspir\'ee de celle d\'evelopp\'ee par D. Schindler pour traiter le cas des hypersurfaces des espaces biprojectifs, qui elle-m\^eme s'inspire de la m\'ethode du cercle de Hardy-Littlewood.  
\end{abstract} 

\tableofcontents

\section{Introduction}

On consid\`ere une hypersurface $ V $ de l'espace $ \PP^{n}_{\QQ}\times \PP^{n}_{\QQ}\times\PP^{n}_{\QQ}  $ d\'efinie par une \'equation $ F(\xx,\yy,\zz)=0 $ o\`u  \[ F(\xx,\yy,\zz)=\sum_{0\leqslant i,j,k \leqslant n}\alpha_{i,j,k}x_{i}y_{j}z_{k}, \] avec $ (\xx,\yy,\zz)=((x_{0}:...:x_{n}),(y_{0}:...:y_{n}),(z_{0}:...:z_{n})) \in \PP^{n}_{\QQ}\times \PP^{n}_{\QQ}\times \PP^{n}_{\QQ} $ et $ \alpha_{i,j,k}\in \QQ $. On dira que $ (x_{0},...,x_{n})\in \ZZ^{n+1} $ est \emph{primitif} si $ \PGCD(x_{0},...,x_{n})=1 $
Dans tout ce qui va suivre, on note pour tout $ (\xx,\yy,\zz) \in \PP^{n}_{\ZZ}\times \PP^{n}_{\ZZ}\times \PP^{n}_{\ZZ} $ : \[ H(\xx,\yy,\zz)=\max_{0\leqslant i \leqslant n}|x_{i}|^{n} \max_{ 0\leqslant j \leqslant n} |y_{j}|^{n}\max_{0\leqslant k\leqslant n}|z_{k}|^{n},  \] la hauteur associ\'ee au fibr\'e anticanonique , et \[ H'(\xx,\yy,\zz)=\max_{0\leqslant i \leqslant n}|x_{i}| \max_{ 0\leqslant j \leqslant n} |y_{j}|\max_{0\leqslant k\leqslant n}|z_{k}|,  \] o\`u $ (x_{0},...,x_{n}), \; (y_{0},...,y_{n}), \; (z_{0},...,z_{n}) \in \ZZ^{n+1} $ sont primitifs et tels que \linebreak  $ (\xx,\yy,\zz)=((x_{0}:...:x_{n}),(y_{0}:...:y_{n}), \;(z_{0}:...:z_{n})) $.
On souhaite d\'eterminer une formule asymptotique pour le nombre de points $ ([\xx],[\yy],[\zz]) $ d'un ouvert de Zariski de l'hypersurface $ V $ de hauteur $ H(\xx,\yy,\zz) $ born\'ee par $ B $ (on notera $ \mathcal{N}_{U}(B) $ ce nombre de points), ce qui revient \`a \'evaluer, quitte \`a remplacer $ B $ par $ B^{n} $ \begin{multline*} \tilde{N}_{U}(B) =\card\{(\xx,\yy,\zz)\in (\ZZ^{n+1}\times \ZZ^{n+1}\times \ZZ^{n+1})\cap U \; | \; \\ (x_{0},...,x_{n}), (y_{0},...,y_{n}), (z_{0},...,z_{n}) \; \primitifs, \; H'(\xx,\yy,\zz)\leqslant B\}, \end{multline*} o\`u $ U $ d\'esigne un ouvert de Zariski de $ \AA_{\CC}^{n+1}\times \AA_{\CC}^{n+1}\times \AA_{\CC}^{n+1} $. On a en effet \[ \mathcal{N}_{U}(B)=\frac{1}{8}\tilde{N}_{U}(B^{\frac{1}{n}}) \] ( le coefficient $ \frac{1}{8} $ est d\^u au fait que deux vecteurs primitifs $ \xx $ et $ -\xx $ repr\'esentent le m\^eme \'el\'ement de $ \PP^{n} $). Par une inversion de M\"{o}bius, on se ram\`ene au calcul de \begin{multline}\label{DEFNU} N_{U}(B)= \card\{(\xx,\yy,\zz)\in (\ZZ^{n+1}\times \ZZ^{n+1}\times \ZZ^{n+1})\cap U  \; | \; H'(\xx,\yy,\zz) \leqslant B\}.  \end{multline} Nous allons \'evaluer ce nombre $ N_{U}(B) $ en suivant la m\'ethode d\'ecrite par Schindler (cf. \cite{S1}, \cite{S2}). Nous \'etablirons en fait (voir proposition \ref{preconclusion}) que, pour un ouvert $ U $ bien choisi, ce nombre est \[ N_{U}(B)=\frac{1}{2}\sigma B^{n}\log(B)^{2}+O(B^{n}\log(B)), \] (o\`u $ \sigma $ est une constante que nous pr\'eciserons), ce qui nous permettra d'en d\'eduire que : \[ \mathcal{N}_{U}(B)=C(V)B^{n}\log(B)^{2}+O(B^{n}\log(B)), \] o\`u $ C(V) $ est la constante conjectur\'ee par Peyre (cf.\cite{Pe}). Remarquons qu'il est ici indispensable de se restreindre \`a un ouvert de Zariski $ U $. La vari\'et\'e $ V $ pr\'esente en effet des sous-vari\'et\'es accumulatrices. Consid\'erons par exemple un point $ (\xx,\yy)\in \ZZ^{n+1}\times \ZZ^{n+1} $ tel que : \[\forall k\in \{0,...,n\},\; \sum_{i,j=0}^{n}\alpha_{i,j,k}x_{i}y_{j}=0. \] Alors, pour tout $ \zz\in \ZZ^{n+1} $ tel que $ |\zz|\leqslant B/(|\xx||\yy|) $ on aurait $ F(\xx,\yy,\zz)=0 $, ce qui implique donc que \[ \card\{ (\xx,\yy,\zz) \in V\cap (\ZZ^{n+1}\times \ZZ^{n+1} \times \ZZ^{n+1}) \; | \;  H'(\xx,\yy,\zz)\leqslant B\} \gg B^{n+1}. \]

On notera, pour $ (\xx,\yy,\zz)\in \CC^{n+1}\times \CC^{n+1}\times \CC^{n+1} $ : \begin{equation} \forall k \in \{ 0,...,n\},\; \; B_{k}(\xx,\yy)=\sum_{i=0}^{n}\sum_{j=0}^{n}\alpha_{i,j,k}x_{i}y_{j},  \end{equation} \begin{equation} \forall j \in \{ 0,...,n\},\; \;B_{j}'(\xx,\zz)=\sum_{i=0}^{n}\sum_{k=0}^{n}\alpha_{i,j,k}x_{i}z_{k}, \end{equation}\begin{equation}\forall i \in \{ 0,...,n\},\; \;  B_{i}''(\yy,\zz)=\sum_{j=0}^{n}\sum_{k=0}^{n}\alpha_{i,j,k}y_{j}z_{k}. \end{equation} Par ailleurs, on d\'efinit \begin{equation}
V_{3}^{\ast}=\{ (\xx,\yy)\in \AA_{\CC}^{n+1}\times \AA_{\CC}^{n+1} \;  |\; \forall k \in \{0,...,n\}, \; B_{k}(\xx,\yy)=0 \},
\end{equation}
 \begin{equation}
V_{2}^{\ast}=\{ (\xx,\zz)\in \AA_{\CC}^{n+1}\times \AA_{\CC}^{n+1} \;  |\; \forall j \in \{0,...,n\},\; B_{j}'(\xx,\zz)=0  \},
\end{equation}
 \begin{equation}
V_{1}^{\ast}=\{ (\yy,\zz)\in \AA_{\CC}^{n+1}\times \AA_{\CC}^{n+1} \;  |\; \forall i \in \{0,...,n\}, \; B_{i}''(\xx,\yy)=0 \}.
\end{equation}

On veut supposer que l'hypersurface $ V $ est lisse. Il nous sera \'egalement utile de supposer qu'elle v\'erifie par ailleurs la propri\'et\'e suivante : \[ \dim V_{1}^{\ast}=\dim V_{2}^{\ast}= \dim V_{3}^{\ast}=n+1. \] Il convient donc de d\'emontrer qu'il existe des rationnels $ (\alpha_{i,j,k})_{i,j,k} $ tels que ces propri\'et\'es soient vraie. En fait, nous allons montrer que chacune de ces propri\'et\'e est vraie pour un ouvert dense de $(\alpha_{i,j,k})_{i,j,k} \in \AA_{\QQ}^{3(n+1)} $. \\

Montrons que $ V $ est lisse pour un ouvert dense de $ (\alpha_{i,j,k})_{i,j,k} \in \AA_{\QQ}^{3(n+1)} $. Remarquons que $ X=\PP^{n}\times \PP^{n}\times \PP^{n} $ peut \^etre vue comme une sous-vari\'et\'e lisse de $ \PP^{N} $ (o\`u $ N=(n+1)^{3}-1 $) via le plongement de Segre : \begin{align*}
s : &  \PP^{n}\times \PP^{n}\times \PP^{n}  \ra  \PP^{N} \\ & ([\xx],[\yy],[\zz])  \mt  [(x_{i}y_{j}z_{k})_{i,j,k}].
\end{align*}  
Alors, d'apr\`es le th\'eor\`eme de Bertini (cf. \cite{Ha}), pour une famille ouverte dense d'hyperplans projectifs $ H_{\aalpha}=\{(X_{i,j,k}) \;  | \; \sum_{i,j,k}\alpha_{i,j,k}X_{i,j,k}=0 \}\subset \PP^{N} $, on a que $ X\cap H_{\aalpha} $ est lisse, or on remarque que :  \[ X\cap H_{\aalpha}=\{([\xx],[\yy],[\zz])\in X \; | \; \sum_{i,j,k=0}^{n}\alpha_{i,jk}x_{i}y_{j}z_{k}=0 \}. \]
Par cons\'equent, $ V $ est lisse pour un ouvert dense de $ (\alpha_{i,j,k})_{i,j,k} \in \AA_{\QQ}^{3(n+1)} $. \\

Nous allons \`a pr\'esent montrer que pour un ouvert dense de $ (\alpha_{i,j,k})_{i,j,k} \in \AA_{\QQ}^{3(n+1)} $, on a $ \dim V_{3}^{\ast}=n+1 $. On plonge $ Y=\PP^{n}\times \PP^{n} $ dans $ \PP^{N'} $ (o\`u $ N'=(n+1)^{2}-1 $) via le plongement de Segre. Toujours par application du th\'eor\`eme de Bertini, on a qu'il existe un ouvert dense d'hyperplans $ H_{\aalpha_{0}}=\{[X_{i,j}] \; | \; \sum_{i,j=0}^{n}\alpha_{i,j,0}X_{i,j}=0 \}  $ ne contenant pas $ Y $ tels que $ H_{\aalpha_{0}}\cap Y $ soit lisse. On a alors $ \dim(Y\cap H_{\aalpha_{0}})=2n-1 $ pour chacun de ces hyperplans. On proc\`ede de m\^eme par la suite avec des hyperplans $ H_{\aalpha_{1}}, ..., H_{\aalpha_{n}} $  (avec $ H_{\aalpha_{k}}=\{[X_{i,j } ] \; | \; \sum_{i,j=0}^{n}\alpha_{i,j,k}X_{i,j}=0 \} $. On trouve alors que, pour un ouvert dense de $ (\alpha_{i,j,k})=(\aalpha_{0},...,\aalpha_{n})\in \AA_{\QQ}^{3(n+1)} $, on a que \[ Y\cap H_{\aalpha_{0}}\cap ...\cap H_{\aalpha_{n}}=\{ ([\xx],[\yy])\in \PP^{n}\times \PP^{n} \; | \; B_{k}(\xx,\yy)=0 \; \forall k \} \] est lisse et de dimension $ 2n-(n+1)=n-1 $. Par cons\'equent, $ \dim V_{3}^{\ast}=n+1 $ pour un ouvert dense de $  (\alpha_{i,j,k})_{i,j,k}\in  \AA_{\QQ}^{3(n+1)} $. On montre de fa\c{c}on analogue que $ \dim V_{2}^{\ast}=n+1 $ et $ \dim V_{1}^{\ast}=n+1 $ pour un ouvert dense de $ (\alpha_{i,j,k}) $. \\

On conclut qu'il existe un ouvert dense de rationnels $  (\alpha_{i,j,k})_{i,j,k} \in  \AA_{\QQ}^{3(n+1)} $ tels que l'hypersurface $ V $ qu'ils d\'efinissent soit lisse et telle que $ \dim V_{1}^{\ast}=\dim V_{2}^{\ast}= \dim V_{3}^{\ast}=n+1 $. Nous supposerons donc dor\'enavant que $ V $ est une telle hypersurface.  \\

La m\'ethode employ\'ee ici pour \'evaluer $ N_{U}(B) $ consiste dans un premier temps \`a donner une formule asymptotique pour le nombre $ N_{U}(P_{1},P_{2},P_{3}) $ de points $ (\xx,\yy,\zz)  $ de $ U\cap (\ZZ^{n+1}\times \ZZ^{n+1}\times \ZZ^{n+1}) $ tels que $ |\xx|\leqslant P_{1} $,  $ |\yy|\leqslant P_{2} $, $ |\zz|\leqslant P_{3} $ (ici $ |\xx| $ d\'esignera $ \max_{i}|x_{i}| $). On d\'emontre en fait que l'on a une formule du type \begin{equation}\label{formuleinitiale}
 N_{U}(P_{1},P_{2},P_{3})=\sigma P_{1}^{n}P_{2}^{n}P_{3}^{n} +O\left( P_{1}^{n}P_{2}^{n}P_{3}^{n}\min\{P_{1},P_{2},P_{3}\}^{-\delta}\right)
\end{equation} pour des constantes $ \sigma $ et $ \delta>0 $ que nous pr\'eciserons. Par la suite, on utilise des r\'esultats semblables \`a ceux de la section $ 9 $ de \cite{S2} pour en d\'eduire $ N_{U}(B) $.\\

Dans la section $ 2 $, en utilisant des arguments issus de la m\'ethode du cercle, on \'etablit la formule \eqref{formuleinitiale} pour $ P_{1},P_{2},P_{3} $ \og relativement proches \fg. Plus pr\'ecis\'ement on montre (cf. proposition \ref{prop1}) que si $ P_{2}=P_{1}^{b} $ et $ P_{3}=P_{1}^{b'} $ avec $ b,b'\geqslant 1 $ et si $ 1+b+b'< n+1 $, alors la formule \eqref{formuleinitiale} est v\'erifi\'ee. Par la suite dans la section $ 3 $, pour un $ \xx\in \ZZ^{n+1} $, on donne une formule asymptotique pour le nombre de points $ (\yy,\zz) $ de la fibre $ V_{\xx} $ de $ V $ tels que $ |\yy|\leqslant P_{2} $ et $ |\zz|\leqslant P_{3} $ en utilisant \`a nouveau la m\'ethode du cercle. Les r\'esultats obtenus combin\'es avec ceux de la section $ 2 $ nous permettrons dans la section $ 4 $ d'\'etablir la formule \eqref{formuleinitiale} pour $ b,b' $ arbitrairement grands mais v\'erifiant $ b'\leqslant b+1+\nu $ (avec $ \nu>0 $ arbitrairement petit) (voir proposition \ref{prop4} ). La section $ 5 $ est consacr\'ee au cas o\`u $ b'> b+1+\nu $. On r\'esout ce probl\`eme en invoquant des r\'esultats de g\'eom\'etrie des nombres, et plus pr\'esis\'ement de comptage de points d'un r\'eseau hyperplan dans un domaine born\'e. Tout ceci permet finalement de d\'emontrer la formule \eqref{formuleinitiale} pour tous $ P_{1},P_{2},P_{3} $ (cf. proposition \ref{NU}). Dans la section $ 6 $, on utilise les r\'esultats \'etablis par Blomer et Br\"{u}dern dans \cite{BB} pour conclure quant \`a la valeur de $ N_{U}(B) $ \`a partir des r\'esultats obtenus pour $ N_{U}(P_{1},P_{2},P_{3}) $. Enfin, la section $ 7 $ est consacr\'ee \`a l'\'etude des constantes intervenant dans la formule asymptotique obtenue pour $ \mathcal{N}_{U}(B) $. On v\'erifie en particulier que le r\'esultat est bien en accord avec les conjectures avanc\'ees par Peyre dans \cite{Pe}.

\section{Premi\`ere \'etape}

Dans cette premi\`ere partie, nous allons d\'emontrer, pour $ 1\leqslant P_{1}\leqslant P_{2}\leqslant P_{3} $, que le nombre \begin{multline*} N(P_{1},P_{2},P_{3})=\card\{ (\xx,\yy,\zz)\in (\ZZ^{n+1}\cap P_{1}\BB_{1})\times (\ZZ^{n+1}\cap P_{2}\BB_{2})\\ \times (\ZZ^{n+1}\cap P_{3}\BB_{3}) \; | \; \max_{k}|B_{k}(\xx,\yy)|\neq 0,  F(\xx,\yy,\zz)=0 \} \end{multline*}
(o\`u $ P_{i}\BB_{i}=P_{i}[-1,1]^{n+1}=[-P_{i},P_{i}]^{n+1} $) est du type \[  N(P_{1},P_{2},P_{3})=\sigma P_{1}^{n}P_{2}^{n}P_{3}^{n}+O(P_{1}^{n-\delta}P_{2}^{n}P_{3}^{n}). \] pour $ n $ assez grand. La m\'ethode utilis\'ee est inspir\'ee de l'article \cite{S1}. 
\subsection{Sommes d'exponentielles}\label{ssection1}

Soit $ \alpha \in [0,1] $, on pose \begin{equation}
S(\alpha)=\sum_{\xx\in P_{1}\BB_{1}\cap \ZZ^{n+1}}\sum_{\substack{\yy\in P_{2}\BB_{2}\cap \ZZ^{n+1}\\ \max_{k}|B_{k}(\xx,\yy)|\neq 0}}\sum_{\zz\in P_{3}\BB_{3}\cap \ZZ^{n+1}}e(\alpha F(\xx,\yy,\zz)) 
\end{equation}  o\`u $ e $ d\'esigne l'application $ x\mt \exp(2i\pi x) $.
On commence par remarquer que, pour $ \xx,\yy $ fix\'es : 
\begin{equation}\label{premieresomme}
\left|\sum_{\zz\in P_{3}\BB_{3}\cap \ZZ^{n+1}}e(\alpha F(\xx,\yy,\zz))\right| \ll \prod_{k=0}^{n}\min\left(P_{3}, ||\alpha B_{k}(\xx,\yy)||^{-1}\right)
\end{equation} 
o\`u, pour $ a\in \RR $, $ ||a|| $ d\'esigne la distance de $ a $ \`a $ \ZZ $, autrement dit $ ||a||=\inf_{m\in \ZZ}|a-m| $. Consid\'erons $ \rr=(r_{0},...,r_{n})\in \NN^{n+1} $. On note, pour $ \xx\in P_{1}\BB_{1}  $ fix\'e : \[  \mathcal{A}(\xx,\rr)=\lbrace \yy \in  P_{2}\BB_{2} \; | \; r_{k}P_{3}^{-1}\leqslant \lbrace \alpha B_{k}(\xx,\yy) \rbrace \leqslant (r_{k}+1)P_{3}^{-1} \rbrace, \] o\`u, pour $ m\in \RR $, $ \lbrace m\rbrace $ d\'esigne la partie fractionnaire de $ m $. On a alors que, pour $ \xx\in P_{1}\BB_{1} $ fix\'e, d'apr\`es \eqref{premieresomme} : 
\begin{multline}\label{deuxiemesomme}
\sum_{\substack{\yy\in P_{2}\BB_{2}\cap \ZZ^{n+1}\\ \max_{k}|B_{k}(\xx,\yy)|\neq 0}}\sum_{\zz\in P_{3}\BB_{3}\cap \ZZ^{n+1}}e(\alpha F(\xx,\yy,\zz)) \\ \ll \sum_{\substack{ |\rr| \leqslant P_{3}}}A(\xx,\rr)\prod_{k=0}^{n}\min\left(P_{3},\max\left(\frac{P_{3}}{r_{k}},\frac{P_{3}}{P_{3}-r_{k}-1}\right)\right).
\end{multline}   

D'autre part, si $ (\uu,\vv)\in \mathcal{A}(\xx,\rr)^{2} $, on a alors : \[ ||\alpha B_{k}(\xx,\uu-\vv)||=||\alpha B_{k}(\xx,\uu)-\alpha B_{k}(\xx,\vv)||<P_{3}^{-1}. \]
Par cons\'equent, si l'on note \[ N(\xx)=\lbrace \yy\in P_{2}\BB_{2}\; | \;  \forall k\in \{0,...,n\}, \;  ||\alpha B_{k}(\xx,\yy)||<P_{3}^{-1}\rbrace, \] on a alors $ A(\xx,\rr)\leqslant N(\xx)+1 $ pour tous $ \xx,\rr $, et la formule \eqref{deuxiemesomme} donne alors : \begin{align*} |S(\alpha)|  & \ll \sum_{\xx\in P_{1}\BB_{1} }N(\xx)\sum_{\substack{ |\rr| \leqslant P_{3}}}\prod_{k=0}^{n}\min\left(P_{3},\max\left(\frac{P_{3}}{r_{k}},\frac{P_{3}}{P_{3}-r_{k}-1}\right)\right) \\ & \ll \sum_{\xx\in P_{1}\BB_{1} }N(\xx)\prod_{k=0}^{n}\left( \sum_{r=0}^{\lfloor P_{3}\rfloor}\min\left(P_{3},\max\left(\frac{P_{3}}{r},\frac{P_{3}}{P_{3}-r-1}\right)\right)\right) \\ & \ll \sum_{\xx\in P_{1}\BB_{1} }N(\xx)(P_{3}\log(P_{3}))^{n+1} \\ & \ll  P_{3}^{n+1+\varepsilon}M_{3}(\alpha,P_{1},P_{2},P_{3}^{-1}), \end{align*}
avec $ \varepsilon>0 $ arbitrairement petit, et \begin{multline}\label{DefM3}
M_{3}(\alpha,P_{1},P_{2},P_{3}^{-1})=\card\left\{(\xx,\yy)\in \ZZ^{n+1}\times \ZZ^{n+1}\; | \; |\xx|\leqslant P_{1}, \right. \\ \left. \; |\yy|\leqslant P_{2},\; ||\alpha B_{k}(\xx,\yy)  ||\leqslant P_{3}^{-1} \right\}.
\end{multline} 
On a donc \'etabli le lemme suivant :
\begin{lemma}\label{lemme11}
Soit $ \alpha\in [0,1] $, $ P >0 $ et $ \varepsilon>0 $ quelconque. Pour un r\'eel $ \kappa>0 $ fix\'e, l'une au moins des assertions suivantes est v\'erifi\'ee : \begin{enumerate}
\item $ |S(\alpha)|<P_{1}^{n+1}P_{2}^{n+1}P_{3}^{n+1+\varepsilon}P^{-\kappa} $.
\item $ M_{3}(\alpha,P_{1},P_{2},P_{3}^{-1}) \gg P_{1}^{n+1}P_{2}^{n+1}P^{-\kappa}$.
\end{enumerate}
\end{lemma}
Nous allons \`a pr\'esent r\'eexprimer la deuxi\`eme assertion de ce lemme, en utilisant \cite[lemme 3.1]{S1} pour \'evaluer $ M_{3}(\alpha,P_{1},P_{2},P_{3}^{-1}) $.
\subsection{Une in\'egalit\'e de type Weyl}\label{ssection2}

Rappelons le r\'esultat suivant (cf. \cite[lemme 3.1]{S1}) :

\begin{lemma}\label{lemme31}
Soient $ n_{1},n_{2} $ des entiers et $ \lambda_{i,j} $ des r\'eels pour $ 1\leqslant i\leqslant n_{1} $ et  $ 1\leqslant j\leqslant n_{2} $. On consid\`ere les formes lin\'eaires \[ \forall \uu\in \RR^{n_{2}} \; \;  L_{i}(\uu)=\sum_{j=1}^{n_{2}}\lambda_{i,j}u_{j} \; \; i\in \{ 1,...,n_{1}\} \] ainsi que \[  \forall \uu\in \RR^{n_{1}} \; \;L_{j}^{t}(\uu)=\sum_{j=1}^{n_{1}}\lambda_{i,j}u_{i} \; \; j\in \{ 1,...,n_{2}\}. \] D'autre part, \'etant donn\'e un r\'eel $ a>1 $, on note \begin{multline*} U(Z)=\card\{(u_{1},...,u_{n_{2}},...,u_{n_{2}+n_{1}})\in \ZZ^{n_{1}+n_{2}}\; | \; \forall j\in \{1,...,n_{2}\},\;|u_{j}|<aZ, \\  \;  \forall i\in \{1,...,n_{1}\},\;|L_{i}(u_{1},...,u_{n_{2}})-u_{n_{2}+i}|<a^{-1}Z \}, \end{multline*}
et de m\^eme 
\begin{multline*} U^{t}(Z)=\card\{(u_{1},...,u_{n_{1}},...,u_{n_{1}+n_{2}})\in \ZZ^{n_{1}+n_{2}}\; |\;  \forall i\in \{1,...,n_{1}\},\; |u_{i}|<aZ, \; \\ \forall j\in \{1,...,n_{2}\}, \;  |L^{t}_{j}(u_{1},...,u_{n_{1}})-u_{n_{1}+j}|<a^{-1}Z\}. \end{multline*}
On a alors, si $ 0<Z_{1}\leqslant Z_{2}\leqslant 1 $ : \[ U(Z_{2})\ll \max\left( \left(\frac{Z_{2}}{Z_{1}}\right)^{n_{2}}U(Z_{1}), \frac{Z_{2}^{n_{2}}}{Z_{1}^{n_{1}}}a^{n_{2}-n_{1}}U^{t}(Z_{1}) \right).\] 
\end{lemma} 
\begin{rem} Par la d\'emonstration de \cite[lemme 3.1]{S1}, on a que cette majoration de $ U(Z_{2}) $ est ind\'ependante des $ \lambda_{i,j} $ choisis. \end{rem}
Nous allons appliquer ce lemme, pour un $ \xx\in [-P_{1},P_{1}]^{n+1} $ fix\'e, aux r\'eels $ (\lambda_{k,j})_{\substack{0\leqslant k\leqslant n \\ 0\leqslant j\leqslant n}}=(\alpha\sum_{i=0}^{n}\alpha_{i,j,k}x_{i})_{\substack{0\leqslant k\leqslant n \\ 0\leqslant j\leqslant n}} $, de sorte que \[ L_{k}(\yy)=\sum_{j=0}^{n}\lambda_{k,j}y_{j}=\alpha\sum_{0\leqslant i,j\leqslant n}\alpha_{i,j,k}x_{i}y_{j}=\alpha B_{k}(\xx,\yy) \] et \[ L_{j}^{t}(\zz)=\sum_{k=0}^{n}\lambda_{k,j}z_{k}=\alpha\sum_{0\leqslant i,k\leqslant n}\alpha_{i,j,k}x_{i}z_{k}=\alpha B_{j}'(\xx,\zz). \]
Dans ce qui va suivre, pour tous r\'eels strictement positifs $ H_{1},H_{2},H_{3} $, on note : \begin{multline} M_{1}(\alpha,H_{1},H_{2},H_{3})=\card\{ (\yy,\zz)\in \ZZ^{n+1}\times \ZZ^{n+1} \; |\; |\zz|\leqslant H_{1}, \\ \; |\yy|\leqslant H_{2}, \; ||\alpha B_{i}''(\yy,\zz)||<H_{3} \}, \end{multline} \begin{multline} M_{2}(\alpha,H_{1},H_{2},H_{3})=\card\{ (\xx,\zz)\in \ZZ^{n+1}\times \ZZ^{n+1} \; |\; |\xx|\leqslant H_{1}, \\ \; |\zz|\leqslant H_{2}, \; ||\alpha B_{j}'(\xx,\zz)||<H_{3} \}, \end{multline}\begin{multline} M_{3}(\alpha,H_{1},H_{2},H_{3})=\card\{ (\xx,\yy)\in \ZZ^{n+1}\times \ZZ^{n+1} \; |\; |\xx|\leqslant H_{1}, \\ \; |\yy|\leqslant H_{2}, \; ||\alpha B_{k}(\xx,\yy)||<H_{3} \}. \end{multline}

On fixe un r\'eel $ \theta_{2}\in [0,1] $. On choisit alors des param\`etres $ Z_{1},Z_{2} $ et $ a $ tels que : \[ P_{2}=aZ_{2}, \; P_{3}^{-1}=a^{-1}Z_{2}, \; P_{2}^{\theta_{2}}=aZ_{1}, \] ce qui implique : \[ a^{-1}Z_{1}=P_{2}^{-1+\theta_{2}}P_{3}^{-1}. \] On a alors, d'apr\`es le lemme \ref{lemme31} : \begin{align*} U(Z_{2}) & =\card\{ \yy \; | \; |\yy|\leqslant P_{2}, \; ||\alpha B_{k}(\xx,\yy)||<P_{3}^{-1}, \; \forall k\in \{0,...,n\}\} \\ &\ll \max\left(\left(\frac{Z_{2}}{Z_{1}}\right)^{n+1}U(Z_{1}), \left(\frac{Z_{2}}{Z_{1}}\right)^{n+1}U^{t}(Z_{1}) \right) \\ & = P_{2}^{(n+1)(1-\theta_{2})}
\max\left(U(Z_{1}), U^{t}(Z_{1}) \right) \end{align*}
(cette majoration \'etant ind\'ependante de $ \xx $) avec \[  U(Z_{1})=\card\{ \yy \; | \; |\yy|\leqslant P_{2}^{\theta_{2}}, \; ||\alpha B_{k}(\xx,\yy)||<P_{2}^{-1+\theta_{2}}P_{3}^{-1}, \; \forall k\in \{0,...,n\}\}, \]\[  U^{t}(Z_{1})=\card\{ \zz \; | \; |\zz|\leqslant P_{2}^{\theta_{2}}, \; ||\alpha B_{j}'(\xx,\zz)||<P_{2}^{-1+\theta_{2}}P_{3}^{-1}, \; \forall j\in \{0,...,n\}\}. \]
En sommant ces majoration sur tous les $ \xx\in [-P_{1},P_{1}]\cap\ZZ^{n+1} $, on trouve alors : \begin{multline}\label{eq1} M_{3}(\alpha,P_{1},P_{2},P_{3}^{-1}) \\ \ll P_{2}^{(n+1)(1-\theta_{2})}\left( M_{3}(\alpha,P_{1},P_{2}^{\theta_{2}},P_{2}^{-1+\theta_{2}}P_{3}^{-1})+M_{2}(\alpha,P_{1},P_{2}^{\theta_{2}},P_{2}^{-1+\theta_{2}}P_{3}^{-1})\right).\end{multline}
Par la suite, on applique \`a nouveau le lemme \ref{lemme31} en prenant cette fois-ci un $ \yy\in [-P_{2}^{\theta_{2}},P_{2}^{\theta_{2}}]^{n+1} $, et en choisissant $ (\lambda_{i,k})_{i,k}=\left(\alpha\sum_{j=0}^{n}\alpha_{i,j,k}y_{j} \right)_{i,k} $, ainsi que : \[ aZ_{2}=P_{1}, \; a^{-1}Z_{2}=P_{2}^{-1+\theta_{2}}P_{3}^{-1}, \; aZ_{1}=P_{1}^{\theta_{1}},\; a^{-1}Z_{1}=P_{1}^{-1+\theta_{1}}P_{2}^{-1+\theta_{2}}P_{3}^{-1}, \] o\`u $ \theta_{1} $ est un r\'eel tel que $ P_{1}^{\theta_{1}}=P_{2}^{\theta_{2}} $. On a alors que : \begin{align*} U(Z_{2}) & =\card\{ \xx \; | \; |\xx|\leqslant P_{1}, \; ||\alpha B_{k}(\xx,\yy)||<P_{2}^{-1+\theta_{2}}P_{3}^{-1}, \; \forall k\in \{0,...,n\}\} \\ &\ll \max\left(\left(\frac{Z_{2}}{Z_{1}}\right)^{n+1}U(Z_{1}), \left(\frac{Z_{2}}{Z_{1}}\right)^{n+1}U^{t}(Z_{1}) \right) \\ & = P_{1}^{(n+1)(1-\theta_{1})}
\max\left(U(Z_{1}), U^{t}(Z_{1}) \right), \end{align*} 
avec
\[  U(Z_{1})=\card\{ \xx \; | \; |\xx|\leqslant P_{1}^{\theta_{1}}, \; ||\alpha B_{k}(\xx,\yy)||<P_{1}^{-1+\theta_{1}}P_{2}^{-1+\theta_{2}}P_{3}^{-1}, \; \forall k\in \{0,...,n\}\}, \]\[  U^{t}(Z_{1})=\card\{ \zz \; | \; |\zz|\leqslant P_{1}^{\theta_{1}}, \; ||\alpha B_{i}''(\yy,\zz)||<P_{1}^{-1+\theta_{1}}P_{2}^{-1+\theta_{2}}P_{3}^{-1}, \; \forall i\in \{0,...,n\}\}. \]

Puis, en sommant sur les $ \yy\in  [-P_{2}^{\theta_{2}},P_{2}^{\theta_{2}}]^{n+1}\cap\ZZ^{n+1}  $, on trouve : \begin{multline}\label{eq2} M_{3}(\alpha,P_{1},P_{2}^{\theta_{2}},P_{2}^{-1+\theta_{2}}P_{3}^{-1})  \ll P_{1}^{(n+1)(1-\theta_{1})}( M_{3}(\alpha,P_{1}^{\theta_{1}},P_{2}^{\theta_{2}},P_{1}^{-1+\theta_{1}}P_{2}^{-1+\theta_{2}}P_{3}^{-1})\\+M_{1}(\alpha,P_{1}^{\theta_{1}},P_{2}^{\theta_{2}},P_{1}^{-1+\theta_{1}}P_{2}^{-1+\theta_{2}}P_{3}^{-1})).\end{multline}

En proc\'edant de la m\^eme mani\`ere pour $ M_{2}(\alpha,P_{1},P_{2}^{\theta_{2}},P_{2}^{-1+\theta_{2}}P_{3}^{-1}) $ (en appliquant cette fois-ci le lemme \ref{lemme31} \`a $ (\lambda_{i,j})=\left(\alpha\sum_{k=0}^{n}\alpha_{i,j,k}z_{k}\right) $ pour un $ z\in [-P_{2}^{\theta_{2}},P_{2}^{\theta_{2}}]^{n+1} $ fix\'e, et en prenant \`a nouveau $ aZ_{2}=P_{1}, \; a^{-1}Z_{2}=P_{2}^{-1+\theta_{2}}P_{3}^{-1}, \; aZ_{1}=P_{1}^{\theta_{1}},\; a^{-1}Z_{1}=P_{1}^{-1+\theta_{1}}P_{2}^{-1+\theta_{2}}P_{3}^{-1} $) on trouve : 

\begin{multline}\label{eq3} M_{2}(\alpha,P_{1},P_{2}^{\theta_{2}},P_{2}^{-1+\theta_{2}}P_{3}^{-1})  \ll P_{1}^{(n+1)(1-\theta_{1})}( M_{2}(\alpha,P_{1}^{\theta_{1}},P_{2}^{\theta_{2}},P_{1}^{-1+\theta_{1}}P_{2}^{-1+\theta_{2}}P_{3}^{-1})\\+M_{1}(\alpha,P_{1}^{\theta_{1}},P_{2}^{\theta_{2}},P_{1}^{-1+\theta_{1}}P_{2}^{-1+\theta_{2}}P_{3}^{-1})).\end{multline}

En regroupant les r\'esultats obtenus en \eqref{eq1}, \eqref{eq2}, \eqref{eq3}, on trouve : \begin{multline}M_{3}(\alpha,P_{1},P_{2},P_{3}^{-1}) \\ \ll P_{1}^{(n+1)(1-\theta_{1})}P_{2}^{(n+1)(1-\theta_{2})}\max_{i\in\{1,2,3\}}M_{i}(\alpha,P_{1}^{\theta_{1}},P_{2}^{\theta_{2}},P_{1}^{-1+\theta_{1}}P_{2}^{-1+\theta_{2}}P_{3}^{-1}). \end{multline}

On d\'eduit de ceci et du lemme \ref{lemme11} le r\'esultat suivant : 
\begin{lemma}
Soit $ \alpha\in [0,1] $, $ P >0 $ et $ \varepsilon>0 $ quelconque. On note $ \theta $ le r\'eel tel que $ P_{1}^{\theta_{1}}=P_{2}^{\theta_{2}}=P^{\theta} $. Pour un r\'eel $ \kappa>0 $ fix\'e, l'une au moins des assertions suivantes est vraie : \begin{enumerate}
\item $ |S(\alpha)|<P_{1}^{n+1}P_{2}^{n+1}P_{3}^{n+1+\varepsilon}P^{-\kappa} $.
\item Il existe un entier $ i\in \{1,2,3\} $ tel que \[ M_{i}(\alpha,P^{\theta},P^{\theta},P_{1}^{-1}P_{2}^{-1}P_{3}^{-1}P^{2\theta}) \gg P^{2(n+1)\theta-\kappa}.\]
\end{enumerate}
\end{lemma}

Consid\'erons \`a pr\'esent un couple $ (\xx,\yy)\in M_{3}(\alpha,P^{\theta},P^{\theta},P_{1}^{-1}P_{2}^{-1}P_{3}^{-1}P^{2\theta}) $ tel qu'il existe un entier $ k_{0}\in \{ 0,...,n\} $ tel que $ B_{k_{0}}(\xx,\yy)\neq 0 $. On note alors $ q=|B_{k_{0}}(\xx,\yy)| $. Par d\'efinition de $ (\xx,\yy) $, on a que $ q\ll P^{2\theta} $. On note alors $ a $ l'entier le plus proche de $ \alpha B_{k_{0}}(\xx,\yy)  $ et $ \delta=\alpha B_{k_{0}}(\xx,\yy)-a $. On a alors $ \alpha q=a+\delta $, et donc : \[ |q\alpha-a|=|\delta|=||\alpha B_{k_{0}}(\xx,\yy)|| <P_{1}^{-1}P_{2}^{-1}P_{3}^{-1}P^{2\theta}. \]
En proc\'edant de m\^eme avec des couples $ (\xx,\zz)\in M_{2}(\alpha,P^{\theta},P^{\theta},P_{1}^{-1}P_{2}^{-1}P_{3}^{-1}P^{2\theta}) $ ou $ (\yy,\zz)\in M_{1}(\alpha,P^{\theta},P^{\theta},P_{1}^{-1}P_{2}^{-1}P_{3}^{-1}P^{2\theta}) $, et en utilisant le lemme pr\'ec\'edent, on obtient le r\'esultat ci-dessous : 
\begin{lemma}\label{lemme12}
Pour un r\'eel $ \kappa>0 $ fix\'e, il existe une constante $ C>0 $ telle que l'une au moins des assertions suivantes est v\'erifi\'ee : \begin{enumerate}
\item \label{i} $ |S(\alpha)|<P_{1}^{n+1}P_{2}^{n+1}P_{3}^{n+1+\varepsilon}P^{-\kappa} $.
\item \label{ii}Il existe des entiers $ a,q $ tels que $ 1\leqslant q\ll P^{2\theta} $, $ \PGCD(a,q)=1 $ et \[2|q\alpha-a|\leqslant P_{1}^{-1}P_{2}^{-1}P_{3}^{-1}P^{2\theta}.\]
\item \label{iii}On a \begin{multline*} \card\{(\xx,\yy)\in \ZZ^{n+1}\times \ZZ^{n+1}\; | \;  |\xx|\leqslant P^{\theta},\; |\yy|\leqslant P^{\theta}, \; \\  B_{k}(\xx,\yy)=0 \;   \forall k \in \{0,...,n\}\}\geqslant CP^{2(n+1)\theta-\kappa}. \end{multline*}
\item \label{iv}On a \begin{multline*}\card\{(\xx,\zz)\in \ZZ^{n+1}\times \ZZ^{n+1}\; | \;  |\xx|\leqslant P^{\theta},\; |\zz|\leqslant P^{\theta}, \; \\  B_{j}'(\xx,\zz)=0 \;   \forall j \in \{0,...,n\}\}\geqslant CP^{2(n+1)\theta-\kappa}. \end{multline*}
\item \label{v}On a \begin{multline*}\card\{(\yy,\zz)\in \ZZ^{n+1}\times \ZZ^{n+1}\; | \;  |\yy|\leqslant P^{\theta},\; |\zz|\leqslant P^{\theta}, \; \\  B_{i}''(\yy,\zz)=0 \;   \forall i \in \{0,...,n\}\}\geqslant CP^{2(n+1)\theta-\kappa}. \end{multline*}
\end{enumerate} 
\end{lemma}

Remarquons que les conditions \ref{iii}, \ref{iv}, \ref{v} impliquent respectivement : \[ \dim V_{3}^{\ast} \geqslant 2(n+1)-\kappa/\theta, \; \dim V_{2}^{\ast} \geqslant 2(n+1)-\kappa/\theta, \;\dim V_{1}^{\ast} \geqslant 2(n+1)-\kappa/\theta, \] (voir par exemple la d\'emonstration du th\'eor\`eme $ 3.1 $ de \cite{Br}). \\

A partir d'ici, nous allons poser $ \kappa=K\theta $, et nous supposerons \[ K=2(n+1)-\max_{i\in \{1,2,3\}}\dim V_{i}^{\ast}=n+1 ,\] (rappelons que nous avons suppos\'e que les $ \alpha_{i,j,k} $ sont tels que $ \dim V_{1}^{\ast}=\dim V_{2}^{\ast}=\dim V_{3}^{\ast}=n+1 $). D'autre part, nous fixerons dor\'enavant $ P=P_{1}P_{2}P_{3} $, que l'on peut aussi \'ecrire $ P=P_{1}^{1+b+b'} $, en notant $ P_{2}=P_{1}^{b} $ et $ P_{3}=P_{1}^{b'} $, avec $ 1\leqslant b \leqslant b' $, puisque $ P_{1}\leqslant P_{2}\leqslant P_{3} $. Etant donn\'e que l'on a $ P^{\theta}=P_{2}^{\theta_{2}}=P_{1}^{\theta_{1}} $, on peut remarquer que l'on a alors : \[ \theta_{2}=\frac{\theta_{1}}{b},\; \theta= \frac{\theta_{1}}{1+b+b'}.  \]
Par ailleurs nous noterons, \`a partir d'ici, (pour un $ \theta $ fix\'e) $ \mathfrak{M}(\theta) $ l'ensemble des $ \alpha\in [0,1] $ satisfaisant la condition \eqref{ii} du lemme pr\'ec\'edent, c'est-\`a-dire : 
\begin{multline}
\mathfrak{M}(\theta)=\{ \alpha\in [0,1]\; | \; \exists \; q,a\in \ZZ \; |\; \PGCD(a,q)=1,\; \\ 1\leqslant q \ll P^{2\theta}, \; 2|q\alpha-a|\ll P^{-1+2\theta} \}.
\end{multline}  
Pour le choix de $ K $ effectu\'e ci-dessus, on voit que le lemme \ref{lemme12} implique : 
\begin{lemma}\label{lemme13}
Soit $ 0<\theta<(1+b+b')^{-1} $, et $ \varepsilon>0 $. L'une au moins des assertions suivantes est vraie : 
 \begin{enumerate}
\item  $ |S(\alpha)|<P_{1}^{n+1}P_{2}^{n+1}P_{3}^{n+1}P^{-K\theta+\varepsilon} $.
\item Le r\'eel $ \alpha $ appartient \`a l'ensemble $ \mathfrak{M}(\theta) $.
\end{enumerate} 
\end{lemma}

\subsection{La m\'ethode du cercle}

On suppose \`a pr\'esent que l'on a $ K=n+1>\max(4,b+b'+1) $. On choisit par ailleurs des r\'eels $ \delta>0 $ et $ \theta_{0}\in [0,1] $ (avec $ \delta $ arbitrairement petit) tels que : \begin{equation}\label{cond1}
 K-4=n-3>2\delta\theta_{0}^{-1},\end{equation}
\begin{equation}\label{cond2}
 K=n+1>(2\delta+1)(1+b+b'),
\end{equation}
\begin{equation}\label{cond3}
 1>10(1+b+b')\theta_{0}+(1+b+b')\delta=(1+b+b')(10\theta_{0}+\delta).
\end{equation}
On \'etablit alors le lemme suivant : 
\begin{lemma}\label{arcmin1}
On a une majoration \[\int_{\alpha\notin \mathfrak{M}(\theta_{0}) }|S(\alpha)|d\alpha \ll P_{1}^{n+1}P_{2}^{n+1}P_{3}^{n+1}P^{-1-\delta}. \]
\end{lemma} \begin{proof} On consid\`ere une suite d'\'el\'ements $ \theta_{i} $ tels que $ 0<\theta_{0}<\theta_{1}<...<\theta_{T-1}<\theta_{T} $, $ \theta_{T}\leqslant (1+b+b')^{-1} $ et $ \theta_{T}K=\theta_{T}(n+1)>2\delta+1 $  (un tel choix de $ \theta_{T} $ est possible d'apr\`es \eqref{cond2}). On choisit de plus les $ \theta_{i} $ tels que $ (\theta_{t+1}-\theta_{t})\leqslant \frac{1}{8}\delta $. Pour un $ \delta>0 $ fix\'e, on peut choisir une telle suite avec $ T\ll P^{\frac{\delta}{4}} $, ce que nous supposerons. On a alors, d'apr\`es le lemme \ref{lemme13} : \begin{align*}
\int_{\alpha\notin \mathfrak{M}(\theta_{T})}|S(\alpha)|d\alpha & \ll P_{1}^{n+1}P_{2}^{n+1}P_{3}^{n+1}P^{-K\theta_{T}+\varepsilon} \\ & \ll P_{1}^{n+1}P_{2}^{n+1}P_{3}^{n+1}P^{-1-\delta} \end{align*} 
(puisque $ \theta_{T}K>2\delta+1 $). On a par ailleurs, par d\'efinition de $ \mathfrak{M}(\theta) $: \[ \Vol(\mathfrak{M}(\theta))\ll \sum_{q\ll P^{2\theta}}\sum_{\substack{|a|<q \\ \PGCD(a,q)=1}}q^{-1}P_{1}^{-1}P_{2}^{-1}P_{3}^{-1}P^{2\theta}\ll P^{-1+4\theta}.\]
On a alors, toujours par le lemme \ref{lemme13}, pour tout $ t\in \{0,...,T-1\} $ : \begin{align*}
\int_{\alpha\notin \mathfrak{M}(\theta_{t+1})\setminus \mathfrak{M}(\theta_{t})}S(\alpha)d\alpha & \ll P_{1}^{n+1}P_{2}^{n+1}P_{3}^{n+1}P^{-K\theta_{t}+\varepsilon}\Vol\left(\mathfrak{M}(\theta_{t+1})\setminus \mathfrak{M}(\theta_{t})\right) \\ & \ll P_{1}^{n+1}P_{2}^{n+1}P_{3}^{n+1}P^{-K\theta_{t}+\varepsilon-1+4\theta_{t+1}}.
\end{align*} 
Or, puisque $ 2\delta\theta_{0}^{-1}<K-4 $ (cf. \ref{cond1}), en rappelant que $ (\theta_{t+1}-\theta_{t})\leqslant \frac{1}{8}\delta $, on a \begin{align*}
-K\theta_{t}+4\theta_{t+1} & <-2\delta\theta_{0}^{-1} +4(\theta_{t+1}-\theta_{t}) \\ & \leqslant -2\delta\theta_{0}^{-1}+\frac{\delta}{2}<-\frac{3}{2}\delta. 
\end{align*}
D'o\`u finalement : \begin{align*} \int_{\alpha\notin \mathfrak{M}(\theta_{0}) }|S(\alpha)|d\alpha & =\sum_{t=0}^{T-1}\int_{\alpha\notin \mathfrak{M}(\theta_{t+1})\setminus \mathfrak{M}(\theta_{t})}S(\alpha)d\alpha +\int_{\alpha\notin \mathfrak{M}(\theta_{T})}|S(\alpha)|d\alpha  \\ & \ll P_{1}^{n+1}P_{2}^{n+1}P_{3}^{n+1}\left(\sum_{t=0}^{T-1}P^{-K\theta_{t}+\varepsilon-1+4\theta_{t+1}} + P^{-1-\delta}\right)\\ & \ll P_{1}^{n+1}P_{2}^{n+1}P_{3}^{n+1}(TP^{\varepsilon-1-\frac{3}{2}\delta}+P^{-1-\delta}) \\ & \ll   P_{1}^{n+1}P_{2}^{n+1}P_{3}^{n+1}P^{-1-\delta}. \end{align*} \end{proof}

On d\'efinit \`a pr\'esent une nouvelle famille d'arc majeurs, pour $ q\geqslant 1 $, et $  a\in \ZZ $  : \[ \mathfrak{M}'_{a,q}(\theta)=\{\alpha\in [0,1]\; | \; |q\alpha-a|\leqslant qP^{-1+2\theta}\} \] et \[ \mathfrak{M}'(\theta)=\bigcup_{1\leqslant q \ll P^{2\theta}}\bigcup_{\substack{0\leqslant a< q \\ \PGCD(a,q)=1}}\mathfrak{M}'_{a,q}(\theta) \] (remarquons que ce nouvel ensemble $ \mathfrak{M}'(\theta) $ contient $ \mathfrak{M}(\theta) $). On peut voir 
facilement que les ensembles $ \mathfrak{M}'_{a,q}(\theta) $, (pour $ 1\leqslant q \ll P^{2\theta} $, et $ 0\leqslant a< q $, $   \PGCD(a,q)=1 $) sont disjoints lorsque $ \theta<\frac{1}{8} $. En effet si l'on suppose, par l'absurde qu'il existe $ \alpha\in \mathfrak{M}'_{a,q}(\theta)\cap\mathfrak{M}'_{a',q'}(\theta) $, avec $ (a,q)\neq(a',q') $. On a alors, puisque $ \PGCD(a,q)=\PGCD(a',q')=1 $ : \[ \frac{1}{qq'}\leqslant\left|\frac{a}{q}-\frac{a'}{q'}\right| \leqslant \left|\alpha-\frac{a}{q}\right|+\left|\alpha-\frac{a'}{q'}\right|\ll P^{-1+2\theta}. \] On aurait alors : \[ 1 \leqslant qq'P^{-1+2\theta} \ll  P^{-1+6\theta}, \] ce qui est absurde pour $ \theta<\frac{1}{6} $. En particulier, puisque $ \theta_{0}<\frac{1}{10} $ (d'apr\`es  \eqref{cond3}), les ensembles $ \mathfrak{M}'_{a,q}(\theta_{0}) $ sont disjoints. On a alors imm\'ediatement, par le lemme \ref{arcmin1} (puisque $ \mathfrak{M}(\theta_{0})\subset \mathfrak{M}'(\theta_{0}) $) :
\begin{lemma}\label{lemme32}
On a l'estimation : \begin{multline*} N(P_{1},P_{2},P_{3})=\sum_{1\leqslant q \ll P^{2\theta_{0}}}\sum_{\substack{0\leqslant a< q \\ \PGCD(a,q)=1}}\int_{\mathfrak{M}'_{a,q}(\theta_{0}) }S(\alpha)d\alpha \\ +O(P_{1}^{n+1}P_{2}^{n+1}P_{3}^{n+1}P^{-1-\delta}). \end{multline*}
\end{lemma}
Etant donn\'e $ \alpha \in \mathfrak{M}'_{a,q}(\theta_{0}) $, on pose $ \alpha=\frac{a}{q}+\beta $, avec $ |\beta|\leqslant P^{-1+2\theta_{0}} $. On introduit \`a pr\'esent les notations : 
\begin{equation}
S_{a,q}=\sum_{\bb^{(1)},\bb^{(2)},\bb^{(3)}\in (\ZZ/q\ZZ)^{n+1}}e\left(\frac{a}{q}F(\bb^{(1)},\bb^{(2)},\bb^{(3)})\right)\end{equation}
et \begin{equation}
I(\beta)=\int_{\BB_{1}\times\BB_{2}\times\BB_{3}}e(\beta F(\uu,\vv,\ww))d\uu d\vv d\ww. 
\end{equation}
On \'etablit le r\'esultat suivant : 
\begin{lemma}\label{lemme33}
Soit $ \alpha\in \mathfrak{M}'_{a,q}(\theta_{0}) $, avec $ q\ll P^{2\theta_{0}} $, et $ 0\leqslant a< q , \;  \PGCD(a,q)=1 $ et $ \beta=\alpha-\frac{a}{q} $ On a alors : \[ S(\alpha)=P_{1}^{n+1}P_{2}^{n+1}P_{3}^{n+1}q^{-3n-3}S_{a,q}I(P\beta) +O(P_{1}^{n+1}P_{2}^{n+1}P_{3}^{n+1}P^{4\theta_{0}}P_{1}^{-1}). \]
\end{lemma}
\begin{proof}Remarquons dans un premier temps que :  \begin{multline*} S(\alpha)=\sum_{\xx\in P_{1}\BB_{1}\cap \ZZ^{n+1}}\sum_{\substack{\yy\in P_{2}\BB_{2}\cap \ZZ^{n+1}}}\sum_{\zz\in P_{3}\BB_{3}\cap \ZZ^{n+1}}e(\alpha F(\xx,\yy,\zz)) \\ +\card\{ (\xx,\yy)\in (P_{1}\BB_{1}\times P_{2}\BB_{2})\cap \ZZ^{2n+2}\cap V_{3}^{\ast}\}P_{3}^{n+1}. \end{multline*} Or, puisque $ \dim(V_{3}^{\ast})=n+1 $ et que $ P_{1}\leqslant P_{2} $, on a \begin{multline*} S(\alpha)=\sum_{\xx\in P_{1}\BB_{1}\cap \ZZ^{n+1}}\sum_{\substack{\yy\in P_{2}\BB_{2}\cap \ZZ^{n+1}}}\sum_{\zz\in P_{3}\BB_{3}\cap \ZZ^{n+1}}e(\alpha F(\xx,\yy,\zz)) \\ +O(P_{2}^{n+1}P_{3}^{n+1}). \end{multline*}
On peut alors \'ecrire  \begin{multline}\label{intermediaire} S(\alpha)=\sum_{\bb^{(1)},\bb^{(2)},\bb^{(3)}\in (\ZZ/q\ZZ)^{n+1}}e\left(\frac{a}{q}F(\bb^{(1)},\bb^{(2)},\bb^{(3)})\right)S_{3}(\bb^{(1)},\bb^{(2)},\bb^{(3)}) \\ +O(P_{2}^{n+1}P_{3}^{n+1}), \end{multline}
avec \[ S_{3}(\bb^{(1)},\bb^{(2)},\bb^{(3)})=\sum_{\substack{\xx',\yy',\zz'\\ (q\xx'+\bb^{(1)},q\yy'+\bb^{(2)},q\zz'+\bb^{(3)}) \\ \in P_{1}\BB_{1}\times P_{2}\BB_{2}\times P_{3}\BB_{3} }}e(\beta F(q\xx'+\bb^{(1)},q\yy'+\bb^{(2)},q\zz'+\bb^{(3)})). \]
On consid\`ere \ a pr\'esent deux triplets $ (\xx',\yy',\zz'), (\xx'',\yy'',\zz'')\in P_{1}\BB_{1}\times P_{2}\BB_{2}\times P_{3}\BB_{3} $ tels que : \[ \max_{0\leqslant i\leqslant n}|x_{i}'-x_{i}''|\leqslant 2, \; \max_{0\leqslant j\leqslant n}|y_{j}'-y_{j}''|\leqslant 2, \; \max_{0\leqslant k\leqslant n}|z_{k}'-z_{k}''|\leqslant 2. \] Dans ce cas, on a \begin{multline*} \left| F(q\xx'+\bb^{(1)},q\yy'+\bb^{(2)},q\zz'+\bb^{(3)}) -F(q\xx''+\bb^{(1)},q\yy''+\bb^{(2)},q\zz''+\bb^{(3)})\right| \\ \ll q(P_{1}P_{2}+P_{1}P_{3}+P_{2}P_{3}) \ll qP_{2}P_{3}. \end{multline*} Remarquons que $ q<\min\{P_{1},P_{2},P_{3}\}=P_{1} $, \'etant donn\'e que $ q\ll P^{2\theta_{0}} $ et $ \theta_{0}<\frac{1}{8(b+b'+1)} $ d'apr\`es \eqref{cond3}. On peut alors remplacer la somme $ S_{3} $ par une int\'egrale, et on obtient : 
\begin{multline*}
S_{3}(\bb^{(1)},\bb^{(2)},\bb^{(3)})=\int_{q\tilde{\uu}\in P_{1}\BB_{1}}\int_{q\tilde{\vv}\in P_{2}\BB_{2}}\int_{q\tilde{\ww}\in P_{3}\BB_{3}}e(\beta F(q\tilde{\uu},q\tilde{\vv},q\tilde{\ww}))d\tilde{\uu}d\tilde{\vv}d\tilde{\ww} \\  +O\left(\underbrace{|\beta|}_{\leqslant P^{-1+2\theta_{0}}}\underbrace{q}_{\leqslant P^{2\theta_{0}}} P_{2}P_{3}\left(\frac{P_{1}}{q}\right)^{n+1}\left(\frac{P_{2}}{q}\right)^{n+1}\left(\frac{P_{3}}{q}\right)^{n+1}\right) \\ + O\left(\left(\frac{P_{1}}{q}\right)^{n}\left(\frac{P_{2}}{q}\right)^{n+1}\left(\frac{P_{3}}{q}\right)^{n+1}\right) . \end{multline*}
En effectuant un changement de variables \[ \uu=qP_{1}^{-1}\tilde{\uu}, \; \vv=qP_{2}^{-1}\tilde{\vv}, \; \ww=qP_{3}^{-1}\tilde{\ww}, \] dans l'int\'egrale, puis en rempla\c{c}ant par l'expression de $ S_{3} $ obtenue dans \eqref{intermediaire}, on trouve le r\'esultat souhait\'e.
\end{proof}
En regroupant les lemmes \ref{lemme32} et \ref{lemme33}, on obtient : 
\begin{multline*} N(P_{1},P_{2},P_{3})= P_{1}^{n+1}P_{2}^{n+1}P_{3}^{n+1}  \sum_{1\leqslant q \ll P^{2\theta_{0}}}q^{-3n-3}\\ \sum_{\substack{0\leqslant a< q \\ \PGCD(a,q)=1}}S_{a,q}\int_{|\beta|\leqslant P^{-1+2\theta_{0}}}I(P\beta)d\beta  \\ + O\left( P_{1}^{n+1}P_{2}^{n+1}P_{3}^{n+1}P^{4\theta_{0}}P_{1}^{-1}\Vol\left(\mathfrak{M}'_{a,q}(\theta_{0})\right)\right).\end{multline*}
Or, on a que \[ \Vol\left(\mathfrak{M}'_{a,q}(\theta_{0})\right)\ll \sum_{q\ll P^{2\theta_{0}}}qP^{-1+2\theta_{0}}\ll P^{-1+6\theta_{0}}. \]
Par cons\'equent, si 'on pose : \begin{equation}
\mathfrak{S}(P^{2\theta_{0}})=\sum_{1\leqslant q\leqslant P^{2\theta_{0}}}q^{-3n-3}\sum_{\substack{0\leqslant a< q \\ \PGCD(a,q)=1}}S_{a,q},
\end{equation}
 et \begin{equation}
J(P^{2\theta_{0}})=\int_{|\beta|\leqslant P^{2\theta_{0}}}I(\beta)d\beta= P\int_{|\beta|\leqslant P^{-1+2\theta_{0}}}I(P\beta)d\beta
\end{equation} on a alors : 
\begin{multline*}
N(P_{1},P_{2},P_{3})= P_{1}^{n}P_{2}^{n}P_{3}^{n}\mathfrak{S}(P^{2\theta_{0}})J(P^{2\theta_{0}})  + O\left( P_{1}^{n+1}P_{2}^{n+1}P_{3}^{n+1}P^{-1+10\theta_{0}}P_{1}^{-1}\right).
\end{multline*}
Or, \begin{align*} P_{1}^{n+1}P_{2}^{n+1}P_{3}^{n+1}P^{-1+8\theta_{0}}P_{1}^{-1} & =P_{1}^{n}P_{2}^{n}P_{3}^{n}P^{10\theta_{0}-\frac{1}{1+b+b'}} \\ & \leqslant P_{1}^{n}P_{2}^{n}P_{3}^{n}P^{-\delta} 
\end{align*}
car on a suppos\'e $ 10\theta_{0}+\delta<(1+b+b')^{-1} $ (cf \ref{cond3}).
On d\'efinit par la suite \begin{equation}
\mathfrak{S}=\sum_{q=1}^{\infty}q^{-3n-3}\sum_{\substack{0\leqslant a< q \\ \PGCD(a,q)=1}}S_{a,q},
\end{equation} 
et 
\begin{equation}
J=\int_{\RR}I(\beta)d\beta
\end{equation} 

\begin{lemma}
La s\'erie $ \mathfrak{S} $ est absolument convergente, et on a, pour $ Q $ assez grand : \[ |\mathfrak{S}-\mathfrak{S}(Q)|\ll Q^{-\frac{n}{2}+2}. \]
\end{lemma}\label{convergeS}
\begin{proof}
Nous allons appliquer les lemmes pr\'ec\'edents avec $ P_{1}=P_{2}=P_{3}=q $. On d\'efinit un \'el\'ement $ \theta\in [0,1] $ par $ 6\theta=1-\varepsilon $ (pour un $  \varepsilon>0 $ fix\'e). On remarque que pour $ \alpha=\frac{a}{q} $, alors $ \alpha $ n'appartient \`a aucun arc majeur $ \mathfrak{M}_{a',q'}'(\theta) $. En effet, si on avait $ \frac{a}{q}\in \mathfrak{M}_{a',q'}'(\theta) $, alors on aurait \[ q'\ll P^{2\theta}=(P_{1}P_{2}P_{3})^{2\theta}=q^{6\theta}=q^{1-\epsilon}<q \]
et par ailleurs \[ 1\leqslant |q'a-a'q|<qq'P^{-1+2\theta}<q^{-1+6\theta}=q^{-\epsilon}, \] ce qui est absurde. 
On a alors (d'apr\`es le lemme \ref{lemme13}): \begin{align*}
|S(\alpha)|=|S_{a,q}| & \ll  P_{1}^{n+1}P_{2}^{n+1}P_{3}^{n+1}P^{-K\theta+\varepsilon} \\ & = q^{3n+3}q^{-3K\theta+\varepsilon} \\  & \ll q^{3n+3}q^{-K/2+\varepsilon'}=q^{3n+3}q^{-\frac{n+1}{2}+\varepsilon'}
\end{align*} 
Donc \begin{align*}
|\mathfrak{S}-\mathfrak{S}(Q)| & =\left|\sum_{q>Q}q^{-3n-3}\sum_{\substack{0\leqslant a < q \\ \PGCD(a,q)=1}}S_{a,q}\right| \\ & \ll \sum_{q>Q}q^{-3n-3}\sum_{\substack{0\leqslant a < q\\ \PGCD(a,q)=1}}|S_{a,q}| \\ & \ll \sum_{q>Q}q^{-\frac{n+1}{2}+1+\varepsilon'} \ll Q^{-\frac{n}{2}+\frac{3}{2}+\varepsilon'}.
\end{align*} 
\end{proof}
\begin{lemma}\label{ConvergeJ}
L'int\'egrale $ J $ est absolument convergente et on a de plus, pour $ \phi $ assez grand : \[ |J-J(\phi)|\ll \phi^{-1}. \]
\end{lemma}
\begin{proof}
Nous allons appliquer les lemmes de la section \ref{ssection2} avec $ \theta=\theta_{0} $ et $ P $ tel que $ P^{2\theta_{0}}=\beta $. On a alors pour tout $ \beta $ que $ P^{-1}\beta\in \mathfrak{M}_{0,1}'(\theta_{0}) $. Le lemme \ref{lemme33} donne alors : 
\begin{equation}\label{eqI(beta)}
S(P^{-1}\beta)=P_{1}^{n+1}P_{2}^{n+1}P_{3}^{n+1}I(\beta)+ O(P_{1}^{n+1}P_{2}^{n+1}P_{3}^{n+1}P^{2\theta_{0}}P_{1}^{-1}).
\end{equation}
De plus, \'etant donn\'e que $ P^{-1}\beta $ appartient au bord de $ \mathfrak{M}'_{0,1}(\theta_{0}) $ (car $ P^{-1}\beta= P^{-1+2\theta_{0}}$), donc au bord de $ \mathfrak{M}'(\theta_{0}) $, le lemme \ref{lemme13} donne : \[ |S(P^{-1}\beta)|\ll P_{1}^{n+1}P_{2}^{n+1}P_{3}^{n+1}P^{-K\theta_{0}+\varepsilon}.\] Par cons\'equent l'\'equation \eqref{eqI(beta)} donne : \begin{multline*}
P_{1}^{n+1}P_{2}^{n+1}P_{3}^{n+1}I(\beta)+ O(P_{1}^{n+1}P_{2}^{n+1}P_{3}^{n+1}P^{2\theta_{0}}P_{1}^{-1}) \\=O(P_{1}^{n+1}P_{2}^{n+1}P_{3}^{n+1}P^{-K\theta_{0}+\varepsilon}), \end{multline*} donc \[ I(\beta)=O(P^{-K\theta_{0}+\varepsilon}+P^{2\theta_{0}}P_{1}^{-1})=O(P^{-K\theta_{0}+\varepsilon}+P^{2\theta_{0}-\frac{1}{1+b+b'}}). \]
Or, on a par ailleurs : \[ P^{2\theta_{0}-\frac{1}{1+b+b'}}<P^{-8\theta_{0}-\delta}\ll \beta^{-4-\delta'}\ll \beta^{-4} ,\]
(voir \eqref{cond3}) ainsi que \[ P^{-K\theta_{0}+\varepsilon} \ll P^{-4\theta_{0}-2\delta}\ll \beta^{-2}, \] (cf. \eqref{cond1}).
On a ainsi $ |I(\beta)|\ll \beta^{-2} $, et donc finalement : \[ |J-J(\phi)|\ll \int_{|\beta|>\phi}|I(\beta)|d\beta\ll \phi^{-1}. \]
\end{proof}

On a donc \'etabli le r\'esultat suivant : 
\begin{prop}\label{prop1}
On suppose $ P_{1}\leqslant P_{2} \leqslant P_{3} $, $ P_{2}=P_{1}^{b} $, $ P_{3}=P_{1}^{b'} $. Pour $ (n+1)>b'+b+1 $, si $ \sigma=\mathfrak{S}J $, on a alors : \[ N(P_{1},P_{2},P_{3})=\sigma P_{1}^{n}P_{2}^{n}P_{3}^{n}+O(P_{1}^{n}P_{2}^{n}P_{3}^{n}P^{-\delta}), \]
(avec $ P=P_{1}P_{2}P_{3} $ et $ \delta>0 $ arbitrairement petit).
\end{prop}

\section{Deuxi\`eme \'etape}

Pour cette partie, nous introduisons les notations suivantes. On fixe un \'el\'ement $ \lambda \in \NN^{\ast} $. Pour cet entier, on note \begin{equation}
\mathcal{A}_{1,\lambda}=\{ \xx\in \CC^{n+1}\; |\; \dim V_{2,\xx}^{\ast}<\lambda, \; \dim V_{3,\xx}^{\ast}<\lambda \}, 
\end{equation}
o\`u l'on a not\'e : \begin{equation}
 V_{3,\xx}^{\ast}=\left\{ \yy \in \CC^{n+1} \; | \; \forall k\in\{0,...,n\},\;  B_{k}(\xx,\yy)=0 \right\},
\end{equation}
\begin{equation}
 V_{2,\xx}^{\ast}=\left\{ \zz \in \CC^{n+1} \; | \; \forall j\in\{0,...,n\}, \; B_{j}'(\xx,\zz)=0  \right\},
\end{equation} et on pose, par abus de langage : \begin{equation}
\mathcal{A}_{1,\lambda}(\ZZ)=\mathcal{A}_{1,\lambda}\cap \ZZ^{n+1}.
\end{equation}
On d\'efinit de m\^eme \begin{equation}
\mathcal{A}_{2,\lambda}=\{ \yy\in \CC^{n+1}\; |\; \dim V_{1,\yy}^{\ast}<\lambda, \; \dim V_{3,\yy}^{\ast}<\lambda \},  
\end{equation}
avec : \begin{equation}
 V_{3,\yy}^{\ast}=\left\{ \xx \in \CC^{n+1} \; | \; \forall k\in\{0,...,n\}, \;B_{k}(\xx,\yy)=0  \right\},
\end{equation}
\begin{equation}
 V_{1,\yy}^{\ast}=\left\{ \zz \in \CC^{n+1} \; | \; \forall i\in\{0,...,n\}, \;B_{i}''(\yy,\zz)=0  \right\},
\end{equation} \begin{equation}
\mathcal{A}_{2,\lambda}(\ZZ)=\mathcal{A}_{2,\lambda}\cap \ZZ^{n+1},
\end{equation} et \begin{equation}
\mathcal{A}_{3,\lambda}=\{ \zz\in \CC^{n+1}\; |\; \dim V_{1,\zz}^{\ast}<\lambda, \; \dim V_{2,\zz}^{\ast}<\lambda \},
\end{equation}
avec : \begin{equation}
 V_{1,\zz}^{\ast}=\left\{ \yy \in \CC^{n+1} \; | \; \forall i\in\{0,...,n\}, \;B_{i}''(\yy,\zz)=0  \right\},
\end{equation}
\begin{equation}
 V_{2,\zz}^{\ast}=\left\{ \xx \in \CC^{n+1} \; | \; \forall j\in\{0,...,n\}, \; B_{j}'(\xx,\zz)=0 \right\},
\end{equation} \begin{equation}
\mathcal{A}_{3,\lambda}(\ZZ)=\mathcal{A}_{3,\lambda}\cap \ZZ^{n+1}.
\end{equation}
Dans ce qui va suivre nous aurons besoin du lemme suivant : 
\begin{lemma}\label{complA}
On a la majoration \[ \card\{\xx\in [-P_{1},P_{1}]^{n+1}\cap \mathcal{A}_{1,\lambda}(\ZZ)^{c}\}\ll P_{1}^{n+1-\lambda}. \]
De plus, l'ensemble $ \mathcal{A}_{1,\lambda}^{c}  $ des vecteurs $ \xx  $ tels que $ \dim V_{3,\xx}^{\ast}\geqslant  \lambda $ ou $ \dim V_{2,\xx}^{\ast}\geqslant  \lambda $ est un ferm\'e de Zariski de $ \AA^{n+1}_{\CC} $.
\end{lemma}
\begin{proof}
On commence par montrer que $ \{\xx\in \AA^{n+1}_{\CC} \; | \; \dim V_{3,\xx}^{\ast}\geqslant \lambda \} $ est un ferm\'e de Zariski de $ \AA^{n+1}_{\CC} $. On aura alors de la m\^eme mani\`ere que $ \{\xx\in \AA^{n+1}_{\CC} \; | \; \dim V_{2,\xx}^{\ast}\geqslant \lambda \} $ est un ferm\'e, et donc que $ \mathcal{A}_{1,\lambda}^{c}  $  est un ferm\'e de $ \AA^{n+1}_{\CC} $. \\

Notons $ Y $ le ferm\'e de $ \AA_{\CC}^{n+1}\times\PP^{n}_{\CC} $ d\'efinit par : \[ Y=\{(\xx,\yy)\in\AA_{\CC}^{n+1}\times\PP^{n}_{\CC}\; | \; \forall k\in \{0,...,n\}, \; B_{k}(\xx,\yy)=0 \}. \] La projection canonique \[ \pi : Y\subset \AA_{\CC}^{n+1}\times\PP^{n}_{\CC} \ra \AA^{n+1}_{\CC}, \] est un morphisme projectif, donc ferm\'e. Par cons\'equent, d'apr\`es \cite[Corollaire 13.1.5]{G-D}, \[ \{\xx\in \AA^{n+1}_{\CC} \; | \; \dim Y_{\xx}\geqslant \lambda-1 \} \] est un ferm\'e, et puisque $ \dim Y_{\xx}=  \dim V_{3,\xx}^{\ast}-1 $, l'ensemble \[ \{\xx\in \AA^{n+1}_{\CC} \; | \; \dim V_{3,\xx}^{\ast}\geqslant \lambda\} \] est un ferm\'e de Zariski de $ \AA^{n+1}_{\CC}  $.\\

Nous allons \`a pr\'esent montrer que $ \dim\mathcal{A}_{1,\lambda}^{c}\leqslant n+1-\lambda $. Pour cela nous allons montrer que la dimension de \[ \mathcal{A}_{1,\lambda,3}^{c}=\{\xx\in \AA^{n+1}_{\CC} \; | \; \dim V_{3,\xx}^{\ast}\geqslant \lambda \} \] est inf\'erieure \`a $ n+1-\lambda $. On remarque que \[ Y\cap (\mathcal{A}_{1,\lambda,3}^{c}\times\PP^{n}_{\CC})=\bigsqcup_{\xx\in \mathcal{A}_{1,\lambda,3}^{c} }\pi^{-1}(\xx). \]
On a alors \[ \dim\mathcal{A}_{1,\lambda,3}^{c}+\lambda-1\leqslant \dim Y=\dim V_{3}^{\ast}-1=n, \] ce qui implique \[ \dim\mathcal{A}_{1,\lambda,3}^{c}\leqslant n+1-\lambda. \]
Ainsi, $ \dim \mathcal{A}_{1,\lambda}^{c}\leqslant n+1-\lambda $, et donc \[ \card\{\xx\in [-P_{1},P_{1}]^{n+1}\cap \mathcal{A}_{1,\lambda}(\ZZ)^{c}\}\ll P_{1}^{n+1-\lambda}\] (cf. d\'emonstration de \cite[Th\'eor\`eme 3.1]{Br})
\end{proof}

On d\'eduit en particulier de ce lemme que les ensembles $ \mathcal{A}_{i,\lambda} $ sont des ouverts de Zariski. On notera dor\'enavant $ U $ l'ouvert de $ V $ : \begin{multline} U=\{ (\xx,\yy,\zz)\in \mathcal{A}_{1,\lambda}\times \mathcal{A}_{2,\lambda} \times \mathcal{A}_{3,\lambda} \; | \; \max_{k}|B_{k}(\xx,\yy)|\neq 0, \\ \;  \max_{j}|B_{j}'(\xx,\zz)|\neq 0, \;  \max_{i}|B_{i}''(\yy,\zz)|\neq 0\; \et \; F(\xx,\yy,\zz)=0 \}. \end{multline}

Notre objectif est d'\'etablir, pour cet ouvert $ U $, une formule asymptotique pour $ N_{U}(B) $ (avec les notations de \eqref{DEFNU}) pour un choix du param\`etre $ \lambda $ que nous pr\'eciserons ult\'erieurement. \`A cette fin, nous allons chercher \`a donner une formule asymptotique pour \begin{multline}
N_{U}(P_{1},P_{2},P_{3})=\card\{ (\xx,\yy,\zz)\in (\mathcal{A}_{1,\lambda}(\ZZ)\cap P_{1}\BB_{1})\times (\mathcal{A}_{2,\lambda}(\ZZ)\cap P_{2}\BB_{2}) \\ \times (\mathcal{A}_{3,\lambda}(\ZZ)\cap P_{3}\BB_{3}) \; | \; \max_{k}|B_{k}(\xx,\yy)|\neq 0, \; \max_{j}|B_{j}'(\xx,\zz)|\neq 0, \\ \; \max_{i}|B_{i}''(\yy,\zz)|\neq 0 \; \et \;
 F(\xx,\yy,\zz)=0 \} \\ =\card\{ (\xx,\yy,\zz)\in U\cap (P_{1}\BB_{1}\times P_{2}\BB_{2}\times P_{3}\BB_{3})\cap \ZZ^{3n+3} \}.
\end{multline}
Pour $ P_{1}\leqslant P_{2}\leqslant P_{3} $, par des m\'ethodes analogues \`a celles d\'evelopp\'ees dans la section pr\'ec\'edente, nous allons d'abord \'etablir une formule asymptotique pour 
\begin{multline}
N_{1}(P_{1},P_{2},P_{3})=\card\{ (\xx,\yy,\zz)\in (\mathcal{A}_{1,\lambda}(\ZZ)\cap P_{1}\BB_{1})  \times (\ZZ^{n+1}\cap P_{2}\BB_{2})\\ \times (\ZZ^{n+1}\cap P_{3}\BB_{3})\; | \;\max_{k}|B_{k}(\xx,\yy)|\neq 0 \; \et \; F(\xx,\yy,\zz)=0\}.
\end{multline}
Ce qui nous permettra d'\'etablir une formule asymptotique pour $ N_{U}(P_{1},P_{2},P_{3}) $ pour le cas o\`u $ P_{1}\leqslant P_{2}\leqslant P_{3} $ (par sym\'etrie on en d\'eduira la m\^eme formules pour les autres cas).

\subsection{Sommes d'exponentielles}
Pour tout ce qui va suivre, on fixe $ \xx\in \mathcal{A}_{1,\lambda}(\ZZ) $. On pose \begin{multline}
N_{\xx}(P_{2},P_{3})=\card\{ (\yy,\zz)\in (\ZZ^{n+1}\cap P_{2}\BB_{2})\times (\ZZ^{n+1}\cap P_{3}\BB_{3})\; | \; \\ \max_{k}|B_{k}(\xx,\yy)|\neq 0 \; \et \;  F(\xx,\yy,\zz)=0\}.
\end{multline}
On a alors \[ N_{\xx}(P_{2},P_{3})=\int_{0}^{1}S_{\xx}(\alpha)d\alpha, \]
pour \begin{equation}
S_{\xx}(\alpha)=\sum_{\substack{\yy\in \ZZ^{n+1}\cap P_{2}\BB_{2}\\ \max_{k}|B_{k}(\xx,\yy)|\neq 0}}\sum_{\zz\in \ZZ^{n+1}\cap P_{3}\BB_{3}}e(\alpha F(\xx,\yy,\zz)).
\end{equation}
Nous allons donc, dans un premier temps, chercher une formule asymptotique pour $ S_{\xx}(\alpha) $. Pour tous r\'eels strictement positifs, $ H_{1},H_{2} $ on notera   \[ M_{\xx}^{2}(H_{1},H_{2})=\card\{\zz\in \ZZ^{n+1}\; | \; |\zz|\leqslant H_{1}, \;\forall j\in \{0,...,n\} \;  ||\alpha B_{j}'(\xx,\zz)||<H_{2} \},  \]\[  M_{\xx}^{3}(H_{1},H_{2})
=\card\{ \yy\in \ZZ^{n+1} \; | \; |\yy|\leqslant H_{1}, \; \forall k \in \{0,...,n\}\;  ||\alpha B_{k}(\xx,\yy)||<H_{2} \}.\]
Remarquons avant tout que l'on a \[ |S_{\xx}(\alpha)|\ll \sum_{\substack{\yy\in \ZZ^{n+1}\cap P_{2}\BB_{2}\\ \max_{k}|B_{k}(\xx,\yy)|\neq 0}}\prod_{k=0}^{n}\min\{P_{3}, ||\alpha B_{k}(\xx,\yy)||^{-1}\}. \] \`A partir de l\`a, on montre comme dans la section \ref{ssection1} que : \[ |S_{\xx}(\alpha)|\leqslant M_{\xx}^{3}(P_{2},P_{3}^{-1})P_{3}^{n+1}\log(P_{3})^{n+1}, \] 
 On \'etablit alors le lemme ci-dessous : 
\begin{lemma}
Soit $ P>0 $, $ \varepsilon>0 $ arbitrairement petit, $ \theta_{2}\in [0,1] $, et $ \kappa>0 $. L'une au moins des assertions suivantes est v\'erifi\'ee : \begin{enumerate}
\item $ |S_{\xx}(\alpha)|\leqslant P_{2}^{n+1}P_{3}^{n+1}P^{-\kappa+\varepsilon} $,
\item $ \max_{i\in\{2,3\}}M_{\xx}^{i}(P_{2}^{\theta_{2}},P_{2}^{-1+\theta_{2}}P_{3}^{-1})\gg P_{2}^{(n+1)\theta_{2}}P^{-\kappa}. $\end{enumerate}

\end{lemma}
\begin{proof}
Il suffit d'utiliser \`a nouveau le lemme \ref{lemme31}, avec \[ (\lambda_{k,j})_{k,j}=\left(\alpha\sum_{i=0}^{n}\alpha_{i,j,k}x_{i}\right)_{k,j}, \]\[ aZ_{2}=P_{2}, \; a^{-1}Z_{2}=P_{3}^{-1}, \; aZ_{1}=P_{2}^{\theta_{2}}, \; a^{-1}Z_{1}=P_{2}^{-1+\theta_{2}}P_{3}^{-1}, \] et on obtient imm\'ediatement le r\'esultat, comme dans la section \ref{ssection2}.
\end{proof}

On consid\`ere un \'el\'ement $ \yy\in M_{\xx}^{3}(P_{2}^{\theta_{2}},P_{2}^{-1+\theta_{2}}P_{3}^{-1}) $. Supposons qu'il existe un certain $ k_{0}\in \{0,...,n\} $ tel que $ B_{k_{0}}(\xx,\yy)\neq 0 $. On note alors $ q=|B_{k_{0}}(\xx,\yy)|\geqslant 1 $ et on pose $ \alpha|B_{k_{0}}(\xx,\yy)|=a+\delta $, avec $ a\in \ZZ $ et $ |\delta|<P_{2}^{-1+\theta_{2}}P_{3}^{-1} $. On a donc \[ |q\alpha-a|<P_{2}^{-1+\theta_{2}}P_{3}^{-1}, \;\;  |q|\ll |\xx|P_{2}^{\theta_{2}}. \] Quitte \`a changer $ \theta_{2} $, on peut supposer \[ |q\alpha-a|<\frac{1}{2}P_{2}^{-1+\theta_{2}}P_{3}^{-1}, \;\;  |q|\leqslant |\xx|P_{2}^{\theta_{2}}. \]D'o\`u le lemme suivant : 

\begin{lemma}
Soient $ P,\varepsilon,\kappa>0 $ et $ \theta_{2}\in [0,1] $ fix\'es. Il existe une constante $ C_{1} $ telle que l'une au moins des assertions suivantes est vraie : \begin{enumerate}
\item \label{i'} $ |S_{\xx}(\alpha)|\leqslant P_{2}^{n+1}P_{3}^{n+1}P^{-\kappa+\varepsilon} $,
\item \label{ii'}Il existe des entiers $ a,q $ tels que $ 1\leqslant q \leqslant |\xx|P_{2}^{\theta_{2}} $ et $ \PGCD(a,q)=1 $, $ 2|q\alpha-a|\leqslant P_{2}^{-1+\theta_{2}}P_{3}^{-1} $;
\item \label{iii'}On a \begin{multline*} \card\{\yy\in \ZZ^{n+1}\; | \;   |\yy|\leqslant P_{2}^{\theta_{2}}, \;   B_{k}(\xx,\yy)=0 \;   \\ \forall k \in \{0,...,n\}\}\geqslant C_{1}P_{2}^{(n+1)\theta_{2}}P^{-\kappa}. \end{multline*}
\item \label{iv'}On a \begin{multline*}\card\{\zz\in \ZZ^{n+1}\; | \;  |\zz|\leqslant P_{2}^{\theta_{2}}, \;   B_{j}'(\xx,\zz)=0 \;   \\ \forall j \in \{0,...,n\}\}\geqslant C_{1}P_{2}^{(n+1)\theta_{2}}P^{-\kappa}. \end{multline*}
\end{enumerate}
\end{lemma}
On choisit \`a pr\'esent $ P=P_{2}P_{3} $ et soit $ \theta $ tel que $ P_{2}^{\theta_{2}}=P^{\theta} $. On a alors $ P_{2}^{(n+1)\theta_{2}}P^{-\kappa}=P^{(n+1)\theta-\kappa} $.
Par cons\'equent, les conditions \ref{iii'} et \ref{iv'} impliquent respectivement que $ P^{\theta\dim V_{3,\xx}^{\ast}}\gg P^{\theta((n+1)-\frac{\kappa}{\theta})} $ et $ P^{\theta\dim V_{2,\xx}^{\ast}}\gg P^{\theta((n+1)-\frac{\kappa}{\theta})} $ (cf. d\'emonstration de \cite[Th\'eor\`eme 3.1]{Br}. On pose \`a pr\'esent $ K_{1}=(n+1)-\lambda $ et on choisit $ \kappa=K_{1}\theta $. On a par ailleurs, puisque $ \xx\in \mathcal{A}_{1,\lambda}(\ZZ) $, $ \dim V_{3,\xx}^{\ast}\leqslant \lambda-1 $ et $ \dim V_{2,\xx}^{\ast}\leqslant \lambda-1 $. Par cons\'equent, si les conditions \ref{iii'} ou \ref{iv'} sont vraies, alors il existe une constante $ C_{2} $ telle que : \[ C_{1}P^{\theta(n+1)-K_{1}\theta} <C_{2}P_{2}^{\theta_{2}(\lambda-1)}=C_{2}P^{\theta(\lambda-1)}, \] ce qui \'equivaut \`a dire que : \[ P^{\theta}<C_{2}/C_{1}. \]
D'o\`u le r\'esultat ci-dessous :
\begin{lemma}\label{2cond}  
Il existe une constante $ C_{3} $ telle que, si $ 0<\theta\leqslant 1 $ et $ P^{\theta}\geqslant C_{3} $, alors au moins l'une de assertions ci-dessous est vraie : 
 \begin{enumerate}
\item  $ |S_{\xx}(\alpha)|\leqslant P_{2}^{n+1}P_{3}^{n+1}P^{-K_{1}\theta+\varepsilon} $,
\item Il existe des entiers $ a,q $ tels que $ 1\leqslant q \leqslant |\xx|P_{2}^{\theta_{2}} $ et $ \PGCD(a,q)=1 $, $ 2|q\alpha-a|\leqslant P_{2}^{-1+\theta_{2}}P_{3}^{-1} $.
\end{enumerate}
\end{lemma}

\subsection{La m\'ethode du cercle}

Pour des entiers $ a,q $ tels que $ \PGCD(a,q)=1 $, $ |q|\leqslant |\xx|P_{2}^{\theta_{2}} $, on d\'efinit les arcs majeurs : 
\begin{equation} \mathfrak{M}_{a,q}^{\xx}(\theta)=\{ \alpha\in [0,1]\; | \; 2|q\alpha-a|\leqslant P_{2}^{-1+\theta_{2}}P_{3}^{-1} \}, \end{equation}
\begin{equation}
\mathfrak{M}^{\xx}(\theta)=\bigcup_{1\leqslant q \leqslant |\xx|P_{2}^{\theta_{2}}}\bigcup_{\substack{0\leqslant a<q\\ \PGCD(a,q)=1}}\mathfrak{M}_{a,q}^{\xx}(\theta),
\end{equation} 
\begin{equation} \mathfrak{M}_{a,q}^{'\xx}(\theta)=\{ \alpha\in [0,1]\; | \; 2|q\alpha-a|\leqslant qP_{2}^{-1+\theta_{2}}P_{3}^{-1} \}, \end{equation}
\begin{equation}
\mathfrak{M}^{'\xx}(\theta)=\bigcup_{1\leqslant q \leqslant |\xx|P_{2}^{\theta_{2}}}\bigcup_{\substack{0\leqslant a<q\\ \PGCD(a,q)=1}}\mathfrak{M}_{a,q}^{'\xx}(\theta).
\end{equation} 

\begin{lemma}\label{separat}
Si l'on suppose $ |\xx|^{2}P_{2}^{-1+3\theta_{2}}P_{3}^{-1}< 1 $ alors les arcs majeurs $ \mathfrak{M}_{a,q}^{'\xx}(\theta) $ sont disjoints deux \`a deux.
\end{lemma}
\begin{proof}
Supposons, par l'absurde qu'il existe $ \alpha\in \mathfrak{M}_{a,q}^{'\xx}(\theta)\cap \mathfrak{M}_{a',q'}^{'\xx}(\theta) $, avec $ a,q,a',q' $ v\'erifiant les hypoth\`eses mentionn\'ee pr\'ec\'edemment. On a alors : \[ \frac{1}{qq'}\leqslant \left| \frac{a}{q}-\frac{a'}{q'}\right|\leqslant P_{2}^{-1+\theta_{2}}P_{3}^{-1}, \] et ceci implique \[ 1\leqslant qq'P_{2}^{-1+\theta_{2}}P_{3}^{-1}\leqslant |\xx|^{2}P_{2}^{-1+3\theta_{2}}P_{3}^{-1} \] ce qui contredit l'hypoth\`ese du lemme.
\end{proof}
A partir d'ici, on supposera $ P^{\theta}>C_{3} $, et on supposera que l'on a bien $ |\xx|^{2}P_{2}^{-1+3\theta_{2}}P_{3}^{-1}< 1 $. On suppose de plus que $ K_{1}>2 $. On d\'efinit par ailleurs : $ \phi(\xx)=|\xx|P_{2}^{\theta_{2}}=|\xx|P^{\theta} $, et $ \Delta(\theta,K_{1})=\theta(K_{1}-2)>0 $. 
\begin{lemma}\label{lemme22}
Pour un $ \varepsilon>0 $ fix\'e, on a la formule asymptotique : 
\begin{multline*} N_{\xx}(P_{2},P_{3})=\sum_{q\leqslant \phi(\xx)}\sum_{\substack{0\leqslant a<q \\ \PGCD(a,q)=1}}\int_{\mathfrak{M}_{a,q}^{'\xx}(\theta)}S_{\xx}(\alpha)d\alpha \\ +O\left(|\xx|P_{2}^{n-\Delta(\theta,K_{1})+\varepsilon}P_{3}^{n-\Delta(\theta,K_{1})+\varepsilon}\right).  \end{multline*}
\end{lemma}
\begin{proof}
On a \[ N_{\xx}(P_{2},P_{3})=\sum_{q\leqslant \phi(\xx)}\sum_{\substack{0\leqslant a<q \\ \PGCD(a,q)=1}}\int_{\mathfrak{M}_{a,q}^{'\xx}(\theta)}S_{\xx}(\alpha)d\alpha+O\left(\mathcal{E}(\xx)\right), \]
avec $ \mathcal{E}(\xx)=\int_{\alpha\notin \mathfrak{M}^{\xx}(\theta)}|S_{\xx}(\alpha)|d\alpha $. 
Remarquons que l'on a \[ \Vol\left(\mathfrak{M}^{\xx}(\theta)\right)\ll \sum_{q\leqslant \phi(\xx)}\sum_{\substack{0\leqslant a<q \\ \PGCD(a,q)=1}}q^{-1}P_{2}^{-1+\theta_{2}}P_{3}^{-1} \ll |\xx|P_{2}^{-1+2\theta_{2}}P_{3}^{-1}. \]
On choisit une suite de r\'eels $ 0<\theta=\theta_{0}'<\theta_{1}'<...<\theta_{T}'= \frac{1}{2} $, avec $ 2(\theta_{i+1}'-\theta_{i}')<\varepsilon $ pour un certain $ \varepsilon>0 $. De plus, $ \varepsilon $ \'etant fix\'e, on peut supposer $ T\ll P^{\varepsilon} $. On a alors que \begin{align*}
\int_{\alpha\notin \mathfrak{M}^{\xx}(\theta_{T}')}|S_{\xx}(\alpha)|d\alpha & \ll P_{2}^{n+1+\varepsilon}P_{3}^{n+1+\varepsilon}P^{-K_{1}\theta_{T}'} \\ & \ll P_{2}^{n+1-K_{1}\theta_{T}'+\varepsilon}P_{3}^{n+1-K_{1}\theta_{T}'+\varepsilon} \\ & \ll P_{2}^{n-\Delta(\theta_{T}',K_{1})+\varepsilon}P_{3}^{n-\Delta(\theta_{T}',K_{1})+\varepsilon}\\ & \ll P_{2}^{n-\Delta(\theta,K_{1})+\varepsilon}P_{3}^{n-\Delta(\theta,K_{1})+\varepsilon}.
\end{align*}

On a de plus : \begin{align*} \int_{\alpha\in \mathfrak{M}^{\xx}(\theta_{i+1}')\setminus\mathfrak{M}^{\xx}(\theta_{i}')}|S_{\xx}(\alpha)|d\alpha & \ll \Vol( \mathfrak{M}^{\xx}(\theta_{i+1}'))P_{2}^{n+1+\varepsilon}P_{3}^{n+1+\varepsilon}P^{-K_{1}\theta_{i}'} \\ &   \ll |\xx|P_{2}^{n+\varepsilon}P_{3}^{n+\varepsilon}P^{2\theta_{i+1}'-K_{1}\theta_{i}'} \\ &  =|\xx|P_{2}^{n+\varepsilon}P_{3}^{n+\varepsilon}P^{2(\theta_{i+1}'-\theta_{i}')-(K_{1}-2)\theta_{i}'} \\ & \ll |\xx|P_{2}^{n+\varepsilon}P_{3}^{n+\varepsilon}P^{\varepsilon-\Delta(\theta_{i}',K_{1})} \\ & \ll |\xx| P_{2}^{n-\Delta(\theta,K_{1})+\varepsilon'}P_{3}^{n-\Delta(\theta,K_{1})+\varepsilon'}
\end{align*}
Et on obtient alors \begin{multline*} \mathcal{E}(\xx)\ll \int_{\alpha\notin \mathfrak{M}^{\xx}(\theta_{T})}|S_{\xx}(\alpha)|d\alpha \\ +\sum_{i=0}^{T} \int_{\alpha\in \mathfrak{M}^{\xx}(\theta_{i+1})\setminus\mathfrak{M}^{\xx}(\theta_{i})}|S_{\xx}(\alpha)|d\alpha \ll  |\xx| P_{2}^{n-\Delta(\theta,K_{1})+\varepsilon''}P_{3}^{n-\Delta(\theta,K_{1})+\varepsilon''}. \end{multline*}

\end{proof}

\subsection{Les arcs majeurs}

Dans tout ce qui suit, pour un $ \xx\in \mathcal{A}_{1,\lambda}(\ZZ) $ fix\'e (avec un $ \lambda $ que nous supposerons inf\'erieur \`a $ n $), $ a,q\in \ZZ $, $ \beta\in \RR $, on introduit les notations suivantes : 
\begin{equation}
S_{a,q}(\xx)=\sum_{\bb^{(1)}\in (\ZZ/q\ZZ)^{n+1}}\sum_{\bb^{(2)}\in (\ZZ/q\ZZ)^{n+1}}e\left(\frac{a}{q}F(\xx,\bb^{(1)},\bb^{(2)})\right),
\end{equation}
\begin{equation}
I_{\xx}(\beta)=\int_{\BB_{2}\times \BB_{3}}e\left(\beta F(\xx,\vv,\ww)\right)d\vv d\ww.
\end{equation}
\begin{lemma}\label{arcmaj1} Soient $ a,q\in \ZZ $ tels que $ 1\leqslant q \leqslant |\xx|P_{2}^{\theta_{2}}=|\xx|P^{\theta} $, $ 0\leqslant a <q $, $ \PGCD(a,q)=1 $ et soit $ \alpha\in \mathcal{M}^{'\xx}_{a,q}(\theta) $. On pose alors $ \beta=\alpha-\frac{a}{q} $. On a alors que : \[ S_{\xx}(\alpha)=P_{2}^{n+1}P_{3}^{n+1}q^{-2n-2}S_{a,q}(\xx)I_{\xx}(P\beta)+O(|\xx|^{2}P_{2}^{n+2\theta_{2}}P_{3}^{n+1}). \]
\end{lemma}
\begin{proof}
On commence par \'ecrire \begin{multline*}
S_{\xx}(\alpha)=\sum_{\substack{\yy\in \ZZ^{n+1}\cap P_{2}\BB_{2}}}\sum_{\zz\in \ZZ^{n+1}\cap P_{3}\BB_{3}}e(\alpha F(\xx,\yy,\zz))\\  + \card\{ \yy\in \ZZ^{n+1}\cap P_{2}\BB_{2} \; | \; \forall k\in \{0,...,n\}, \; B_{k}(\xx,\yy)=0\}P_{3}^{n+1}.
\end{multline*} Puisque $ \xx\in \mathcal{A}_{1,\lambda}(\ZZ) $, on obtient : \begin{multline*} S_{\xx}(\alpha)=\sum_{\substack{\yy\in \ZZ^{n+1}\cap P_{2}\BB_{2}}}\sum_{\zz\in \ZZ^{n+1}\cap P_{3}\BB_{3}}e(\alpha F(\xx,\yy,\zz)) +O(P_{2}^{\lambda-1}P_{3}^{n+1}),
\end{multline*} que l'on peut r\'e\'ecrire :
\begin{multline*} S_{\xx}(\alpha)=\sum_{\bb^{(1)}\in (\ZZ/q\ZZ)^{n+1}}\sum_{\bb^{(2)}\in (\ZZ/q\ZZ)^{n+1}}e\left(\frac{a}{q}F(\xx,\bb^{(1)},\bb^{(2)})\right)S_{3}(\bb^{(1)},\bb^{(2)})\\ +O(P_{2}^{\lambda-1}P_{3}^{n+1}) \end{multline*} avec \[ S_{3}(\bb^{(1)},\bb^{(2)})=\sum_{\substack{q\yy'+\bb^{(1)}\in P_{2}\BB_{2}\\ q\zz'+\bb^{(2)}\in P_{3}\BB_{3} }}e(\beta F(\xx,q\yy'+\bb^{(1)},q\zz'+\bb^{(2)})). \]
On remarque que, pour $ \yy',\yy'',\zz',\zz'' $ tels que $ |\yy'-\yy''|\ll 1 $ et $ |\zz'-\zz''|\ll 1 $, on a 
\[
|F(\xx,q\yy'+\bb^{(1)},q\zz'+\bb^{(2)})-F(\xx,q\yy''+\bb^{(1)},q\zz''+\bb^{(2)})| \ll q|\xx|P_{2}+q|\xx|P_{3} \ll q|\xx|P_{3} \] 
On a donc, en rempla\c{c}ant la s\'erie par une int\'egrale, que \begin{multline*} S_{3}(\bb^{(1)},\bb^{(2)})=\int_{\substack{q\tilde{\vv}\in P_{2}\BB_{2} \\ q\tilde{\ww}\in P_{3}\BB_{3} }}e(\beta F(\xx, q\tilde{\vv},q\tilde{\ww}))d\tilde{\vv}d\tilde{\ww} \\ +O\left(|\beta|q|\xx|P_{3}\left(\frac{P_{3}}{q}\frac{P_{2}}{q}\right)^{n+1}+\left(\frac{P_{2}}{q}\right)^{n}\left(\frac{P_{3}}{q}\right)^{n+1}\right) \\= P_{2}^{n+1}P_{3}^{n+1}q^{-2n-2}I_{\xx}(P\beta)+O\left(q^{-2n-1}|\xx|P_{2}^{n+\theta_{2}}P_{3 }^{n+1}\right),
\end{multline*}
par changement de variables $ \vv=qP_{2}^{-1}\tilde{\vv} $ et $ \ww=qP_{3}^{-1}\tilde{\ww} $. 
On en d\'eduit finalement : \[ S_{\xx}(\alpha)=P_{2}^{n+1}P_{3}^{n+1}q^{-2n-2}S_{a,q}(\xx)I_{\xx}(P\beta)+O(E), \]
avec \begin{align*} E & =P_{2}^{\lambda-1}P_{3}^{n+1}+\sum_{\bb^{(1)}\in (\ZZ/q\ZZ)^{n+1}}\sum_{\bb^{(2)}\in (\ZZ/q\ZZ)^{n+1}}q^{-2n-1}|\xx|P_{2}^{n+\theta_{2}}P_{3 }^{n+1} \\ & \ll|\xx|q P_{2}^{n+\theta_{2}}P_{3 }^{n+1} \ll |\xx|^{2} P_{2}^{n+2\theta_{2}}P_{3 }^{n+1}. \end{align*} 
\end{proof}
\`A partir d'ici on notera : \begin{equation}
\mathfrak{S}_{\xx}(Q)=\sum_{q\ll Q}q^{-2n-2}\sum_{\substack{0\leqslant a<q \\ \PGCD(a,q)=1 }}S_{a,q}(\xx)
\end{equation}
\begin{equation}
J_{\xx}(\phi)=\int_{|\beta|\leqslant \phi}I_{\xx}(\beta)d\beta.
\end{equation}
\begin{lemma}\label{Nx1}
Si $ \xx\in \mathcal{A}_{1,\lambda}(\ZZ) $, on a, pour tout $ \varepsilon>0 $ : 
\begin{multline*} N_{\xx}(P_{2},P_{3})=P_{2}^{n}P_{3}^{n}\mathfrak{S}_{\xx}(\phi(\xx))J_{\xx}\left(\frac{1}{2}P_{2}^{\theta_{2}}\right)\\+O(|\xx|^{4}P_{2}^{n-1+5\theta_{2}}P_{3}^{n}+ |\xx|P_{2}^{n-\Delta(\theta,K_{1})+\varepsilon}P_{3}^{n-\Delta(\theta,K_{1})+\varepsilon}).\end{multline*}

\end{lemma}

\begin{proof}
On notera $ E_{1}=|\xx|P_{2}^{n-\Delta(\theta,K_{1})+\varepsilon}P_{3}^{n-\Delta(\theta,K_{1})+\varepsilon} $. Par application des lemmes \ref{lemme22} et \ref{arcmaj1}, on a : \begin{align*}
N_{\xx}(P_{2},P_{3})&=\sum_{q\leqslant \phi(\xx)}\sum_{\substack{1\leqslant a<q \\ \PGCD(a,q)=1}}\int_{\mathfrak{M}_{a,q}^{'\xx}(\theta)}S_{\xx}(\alpha)d\alpha +O(E_{1}) \\ & =P_{2}^{n+1}P_{3}^{n+1}\sum_{q\leqslant \phi(\xx)}\sum_{\substack{1\leqslant a<q \\ \PGCD(a,q)=1}}S_{a,q}(\xx)\int_{|\beta|\leqslant \frac{1}{2}P_{2}^{-1+\theta_{2}}P_{3}^{-1}}I_{\xx}(P_{2}P_{3}\beta)d\beta \\ & \; \; +O(E_{1})+O(E_{2}),
\end{align*} 
avec $ E_{2}=\Vol(\mathfrak{M}^{'\xx}(\theta))|\xx|^{2}P_{2}^{n+2\theta_{2}}P_{3}^{n+1} $. On remarque que : 
\begin{align*}
\Vol(\mathfrak{M}^{'\xx}(\theta)) & \ll \sum_{q\leqslant  \phi(\xx)}\sum_{\substack{1\leqslant a<q \\ \PGCD(a,q)=1}}P_{2}^{-1+\theta_{2}}P_{3}^{-1} \\ & \ll P_{2}^{-1+\theta_{2}}P_{3}^{-1}(|\xx|P_{2}^{\theta_{2}})^{2} \\ & \ll |\xx|^{2}P_{2}^{-1+3\theta_{2}}P_{3}^{-1}.
\end{align*}
Par cons\'equent, \[ E_{2} \ll |\xx|^{4}P_{2}^{n-1+5\theta_{2}}P_{3}^{n}, \] d'o\`u le r\'esultat.

\end{proof}
\begin{lemma}\label{Jx}
Soit $ \xx\in  \mathcal{A}_{1,\lambda}(\ZZ) $. On suppose que l'on a de plus $ K_{1}>2 $. Alors, l'int\'egrale $ J_{\xx}=\int_{\RR}I_{\xx}(\beta)d\beta $ est absolument convergente, et on a : \[ \left|J_{\xx}-J_{\xx}\left(\frac{1}{2}P_{2}^{\theta_{2}}\right)\right|\ll P_{2}^{\theta_{2}(1-K_{1}+\varepsilon')} \](pour un $ \varepsilon'>0 $ arbitrairement petit), et on a $ |J_{\xx}|\ll 1 $.
\end{lemma}

\begin{proof}
On consid\`ere un r\'eel $ \beta $ tel que $ |\beta|>C_{3} $. On choisit $ \theta',\theta_{2}' $ et $ P_{3} $ tels que $ 2|\beta|=P^{\theta'}=P_{2}^{\theta_{2}'} $ et $ P^{-K_{1}\theta'}=|\xx|^{2}P_{2}^{-1+2\theta_{2}'} $ (avec $ P=P_{2}P_{3} $). Remarquons que cette derni\`ere condition implique : \[ |\xx|^{2}P_{2}^{-1+3\theta_{2}}P_{3}^{-1}=P^{(-K_{1}+1)\theta'}P_{3}^{-1}<1 \] et donc la condition du lemme \ref{separat} est satisfaite. On a alors, d'apr\`es le lemme \ref{2cond}, si $ P^{\theta'}>C_{3} $ alors : \[ |S_{\xx}(P^{-1}\beta)|<P_{2}^{n+1}P_{3}^{n+1}P^{-K_{1}\theta'+\varepsilon}, \] (car $ P^{-1}\beta $ appartient au bord de $ \mathfrak{M}_{0,1}^{\xx}(\theta') $). On a par ailleurs, d'apr\`es le lemme \ref{arcmaj1} appliqu\'e \`a $ (a,q)=(0,1) $ : \[ P^{n+1}|I_{\xx}(\beta)|\ll |S_{\xx}(P^{-1}\beta)| +O(|\xx|^{2}P_{2}^{n+2\theta_{2}'}P_{3}^{n+1}). \] On obtient alors la borne \begin{align*} I_{\xx}(\beta)& \ll P^{\varepsilon-K_{1}\theta'}+|\xx|^{2}P_{2}^{-1+2\theta_{2}'} \\  & \ll P^{\varepsilon-K_{1}\theta'} \\& =|\beta|^{-K_{1}+\varepsilon/\theta'} \\ & \ll |\beta|^{-K_{1}+\varepsilon/\theta}=|\beta|^{-K_{1}+\varepsilon'}  \end{align*}($ \theta $ \'etant fix\'e). Par cons\'equent, si $ P^{\theta'}>C_{3} $, on a que \begin{align*} \left|J_{\xx}\left(\frac{1}{2}P_{2}^{\theta_{2}}\right)-J_{\xx}\right| & \ll  \int_{|\beta|>\frac{1}{2}P_{2}^{\theta_{2}}}|\beta|^{\varepsilon'-K_{1}}d\beta \\ & \ll P_{2}^{\theta_{2}(1-K_{1}+\varepsilon')}. \end{align*}
Par ailleurs, si l'on choisit $ P_{2} $ tr\`es petit (de sorte que $ \frac{1}{2}P_{2}^{\theta_{2}}\asymp C_{3} $), on obtient $ |J_{\xx}(C_{3})-J_{\xx}|\ll P_{2}^{\theta_{2}(1-K_{1}+\varepsilon')} \ll 1 $, et donc $ |J_{\xx}|\ll 1 $ (car $ |J_{\xx}(C_{3})|\ll 1 $). 

\end{proof}

\begin{lemma}\label{Sx}
Soit $ \xx \in \mathcal{A}_{1,\lambda}(\ZZ) $. On suppose que l'on a $ K_{1}>2 $. Alors, pour tout $ \varepsilon>0 $ fix\'e, la s\'erie \[ \mathfrak{S}_{\xx}=\sum_{q=1}^{\infty}q^{-2n-2}\sum_{\substack{0\leqslant a<q \\ \PGCD(a,q)=1 }}S_{a,q}(\xx) \] converge absolument, et on a \[ \left|\mathfrak{S}_{\xx}(\phi(\xx))-\mathfrak{S}_{\xx}\right|\ll |\xx|^{2+\varepsilon}P_{2}^{\theta_{2}(2-K_{1}+\varepsilon)}. \]
On a de plus la borne $ |\mathfrak{S}_{\xx}|\ll |\xx|^{2+\varepsilon} $.
\end{lemma}
\begin{proof}
Remarquons d'abord que $ S_{a,q}(\xx)=S_{\xx}(\alpha) $ pour $ P_{2}=P_{3}=q $ et $ \alpha=\frac{a}{q}\in \mathfrak{M}_{a,q}^{'\xx}(\theta) $. Supposons $ \theta'\in [0,1] $ tel que $ q^{\theta'}>C_{3} $, alors, par le lemme \ref{2cond}, on a : \[ |S_{a,q}(\xx)|<q^{2(n+1)-K_{1}\theta'+\varepsilon} \] ou alors il existe $ q',a'\in \ZZ $ tels que $ 1\leqslant q'\leqslant |\xx|q^{\theta'} $ et $ |q'a-a'q|\leqslant q^{-1+\theta'}  $. Ce deuxi\`eme cas est alors impossible lorsque $ q'\neq q $, donc en particulier lorsque $ |\xx|q^{\theta'}<q $. Quitte \`a supposer $ q $ tel que $ q>C_{3}|\xx| $, on choisit alors $ \theta' $ tel que $ q^{\theta'}=|\xx|^{-1}q^{-\varepsilon+1} $, et on a alors \[  |S_{a,q}(\xx)|<q^{2(n+1)-K_{1}\theta'+\varepsilon}=q^{2(n+1)-K_{1}+\varepsilon'}|\xx|^{K_{1}}. \] On remarque par ailleurs que pour $ P^{\theta}=P_{2}^{\theta_{2}}>C_{3} $, on a $ \phi(\xx)=P_{2}^{\theta_{2}}|\xx|>C_{3}|\xx| $, et donc, par ce qui pr\'ec\`ede on a l'estimation :

\begin{align*}
|\mathfrak{S}_{\xx}(\phi(\xx))-\mathfrak{S}_{\xx}|&\ll \sum_{q>\phi(\xx)}q^{-2n-2}\sum_{\substack{0\leqslant a<q \\ \PGCD(a,q)=1 }}|S_{a,q}(\xx)| \\ &\ll  \sum_{q>\phi(\xx)}q^{1-K_{1}+\varepsilon'}|\xx]^{K_{1}} \\ & \ll P_{2}^{\theta_{2}(2-K_{1}+\varepsilon')}|\xx|^{2+\varepsilon'}.
\end{align*}
Par le m\^eme calcul, on trouve : $ |\mathfrak{S}_{\xx}(C_{3}|\xx|)-\mathfrak{S}_{\xx}|\ll|\xx|^{2+\varepsilon'} $ et en utilisant l'estimation triviale $ |\mathfrak{S}_{\xx}(C_{3}|\xx|)|\ll |\xx|^{2+\varepsilon'} $ (obtenue en majorant trivialement $ S_{a,q}(\xx) $ par $ q^{2n+2} $), on a alors $ |\mathfrak{S}_{\xx}|\ll |\xx|^{2+\varepsilon'} $. 
\end{proof}
\begin{lemma}\label{Nx2}
Soit $ \xx \in \mathcal{A}_{1,\lambda}(\ZZ) $. On suppose que l'on a $ 0<\theta\leqslant 1 $ et $ P_{2}\geqslant 1 $ tel que $ P_{2}^{\theta_{2}}>C_{3} $ et tel que $ |\xx|^{2}P_{2}^{-1+3\theta_{2}}P_{3}^{-1}<1 $. On suppose de plus que $ K_{1}>2 $. Pour un $ \varepsilon>0 $ arbitrairement petit, on a la formule suivante : \[ N_{\xx}(P_{2},P_{3})=\mathfrak{S}_{\xx}J_{\xx}P_{2}^{n}P_{3}^{n}+O(E_{2}(\xx))+O(E_{3}(\xx)), \] avec \[ E_{2}(\xx)=|\xx|^{4}P_{2}^{n-(1-5\theta_{2})}P_{3}^{n}, \] \[ E_{3}(\xx)=|\xx|^{2+\varepsilon}P_{2}^{n-\Delta(\theta,K_{1})+\varepsilon}P_{3}^{n} .\]
\end{lemma}
\begin{proof}
D'apr\`es le lemme \ref{Nx1} on a \[ N_{\xx}(P_{2},P_{3})=P_{2}^{n}P_{3}^{n}\mathfrak{S}_{\xx}(\phi(\xx))J_{\xx}\left(\frac{1}{2}P_{2}^{\theta_{2}}\right)+O(E_{1})+O(E_{2}), \] avec $ E_{1}=|\xx|P_{2}^{n-\Delta(\theta,K_{1})+\varepsilon}P_{3}^{n-\Delta(\theta,K_{1})+\varepsilon}). $ On a donc $ E_{1}\ll E_{3} $. De plus, par les lemmes \ref{Jx} et \ref{Sx} : \begin{align*}
 & \left|\mathfrak{S}_{\xx}J_{\xx}-\mathfrak{S}_{\xx}(\phi(\xx))J_{\xx}\left(\frac{1}{2}P_{2}^{\theta_{2}}\right)\right| \\ & \ll |\mathfrak{S}_{\xx}-\mathfrak{S}_{\xx}(\phi(\xx))|\left|J_{\xx}\left(\frac{1}{2}P_{2}^{\theta_{2}}\right)\right| +|\mathfrak{S}_{\xx}|\left|J_{\xx}\left(\frac{1}{2}P_{2}^{\theta_{2}}\right)-J_{\xx}\right| \\ & \ll |\xx|^{2+\varepsilon}P_{2}^{\theta_{2}(2-K_{1}+\varepsilon)}+P_{2}^{\theta_{2}(1-K_{1}+\varepsilon)}|\xx|^{2+\varepsilon} \\ & \ll |\xx|^{2+\varepsilon}P_{2}^{\theta_{2}(2-K_{1})+\varepsilon}\\ & \ll |\xx|^{2+\varepsilon}P_{2}^{-\Delta(\theta_{2},K_{1})+\varepsilon} \leqslant |\xx|^{2+\varepsilon}P_{2}^{-\Delta(\theta,K_{1})+\varepsilon},
\end{align*}
(rappelons que $ \theta(1+\frac{b'}{b})=\theta_{2} $, donc $ \theta\leqslant \theta_{2} $ et $ \Delta(\theta,K_{1})\leqslant \Delta(\theta_{2},K_{1}) $). D'o\`u le r\'esultat. 
\end{proof} 

On a en particulier le corollaire suivant :
\begin{cor}\label{lecorollaire}
Si $ \xx \in \mathcal{A}_{1,\lambda}(\ZZ) $ et si $ K_{1}>2 $, alors il existe $ \delta>0 $ tel que \[ N_{\xx}(P_{2},P_{3})=\mathfrak{S}_{\xx}J_{\xx}P_{2}^{n}P_{3}^{n}+O([\xx|^{4}P_{2}^{n-\delta}P_{3}^{n}) \] uniform\'ement pour tout $ \xx $ tel que $ |\xx|<(P_{2}^{1-3\theta_{2}}P_{3})^{\frac{1}{2}} $ (pour un $ \theta_{2}<\frac{1}{5} $ fix\'e). 
\end{cor}
On pose pour tout $ \delta>0 $ : \begin{equation}
g_{1}(b,b',\delta)=\left(1+\frac{b'}{b}\right)\left(1-\frac{5}{b}-\delta\right)^{-1}5\left(\frac{3}{b}+2\delta\right). 
\end{equation} 
Nous sommes \`a pr\'esent en mesure de d\'emontrer le r\'esultat suivant : 
\begin{prop}\label{prop2}
Soit $ \delta>0 $. On suppose que l'on a $ \frac{5}{b}+\delta<1 $ (o\`u $ P_{2}=P_{1}^{b}) $. De plus, si $ K_{1}=n+1-\lambda $ v\'erifie : \[ K_{1}-2>g_{1}(b,b',\delta), \] et si $ P_{2}^{\frac{1-\delta-5/b}{5}}>C_{3} $ alors : \[ N_{1}(P_{1},P_{2},P_{3})=
P_{2}^{n}P_{3}^{n}\sum_{\xx\in P_{1}\BB_{1}\cap \mathcal{A}_{1,\lambda}(\ZZ) }\mathfrak{S}_{\xx}J_{\xx}+O(P_{1}^{n}P_{2}^{n-\delta}P_{3}^{n}). \]
\end{prop}
\begin{proof}
Par d\'efinition, on a \[ N_{1}(P_{1},P_{2},P_{3})=\sum_{\xx \in P_{1}\BB_{1}\cap \mathcal{A}_{1,\lambda}(\ZZ)}N_{\xx}(P_{2},P_{3}). \] Donc pour $ \theta,\theta_{2} $ satisfaisant les hypoth\`eses du lemme \ref{Nx2}, on a : \[ N_{1}(P_{1},P_{2},P_{3})=P_{2}^{n}P_{3}^{n}\sum_{\xx\in P_{1}\BB_{1}\cap \mathcal{A}_{1,\lambda}(\ZZ) }\mathfrak{S}_{\xx}J_{\xx} +O(\mathcal{E}_{2})+O(\mathcal{E}_{3}), \] o\`u \begin{align*}
\mathcal{E}_{2} & =\sum_{\xx \in P_{1}\BB_{1}\cap \mathcal{A}_{1,\lambda}(\ZZ)}E_{2}(\xx) \\ & = P_{2}^{n-1+5\theta_{2}}P_{3}^{n}\sum_{\xx \in P_{1}\BB_{1}\cap \mathcal{A}_{1,\lambda}(\ZZ)}|\xx|^{4} \\ &\ll P_{1}^{n+5}P_{2}^{n-1+5\theta_{2}}P_{3}^{n} \\ & = P_{1}^{n}P_{2}^{n}P_{3}^{n}P_{2}^{-1+5\theta_{2}+5/b},
\end{align*} 
et \begin{align*}
\mathcal{E}_{3} & =\sum_{\xx \in P_{1}\BB_{1}\cap \mathcal{A}_{1,\lambda}(\ZZ)}E_{3}(\xx) \\ & = P_{2}^{n-\Delta(\theta,K_{1})+\varepsilon}P_{3}^{n}\sum_{\xx \in P_{1}\BB_{1}\cap \mathcal{A}_{1,\lambda}(\ZZ)}|\xx|^{2+\varepsilon} \\ &\ll P_{1}^{n+3}P_{2}^{n-\Delta(\theta,K_{1})+2\varepsilon}P_{3}^{n} \\ & =P_{1}^{n}P_{2}^{n}P_{3}^{n}P_{2}^{3/b-\Delta(\theta,K_{1})+2\varepsilon}.
\end{align*}
On choisit ensuite $ \theta_{2} $ tel que $ -1+5\theta_{2}+5/b=-\delta $ (ce qui est possible, car on a $ 5/b+\delta<1 $, par hypoth\`ese) et donc $ \theta_{2}=(1-5/b-\delta)/5 $. L'hypoth\`ese $ K_{1}-2>g_{1}(b,b',\delta) $ implique \begin{align*}
 K_{1}-2 & >\left(1+\frac{b'}{b}\right)\left(1-\frac{5}{b}-\delta\right)^{-1}5\left(\frac{3}{b}+2\delta\right)\\ & =\left(1+\frac{b'}{b}\right)\theta_{2}^{-1}\left(\frac{3}{b}+2\delta\right) \\ & =\theta^{-1}\left(\frac{3}{b}+2\delta\right) 
\end{align*}
(car $ \theta_{2}=\theta(1+\frac{b'}{b}) $). On a alors \[ \Delta(\theta,K_{1})=\theta(K_{1}-2)>\left(\frac{3}{b}+2\delta\right). \] On en d\'eduit : \[ 3/b-\Delta(\theta,K_{1})+2\varepsilon<-2\delta+2\varepsilon<-\delta \] (pour $ \varepsilon $ assez petit). On a donc d\'emontr\'e la proposition pour $ P_{2}^{\theta_{2}}=P_{2}^{\frac{1-\delta-5/b}{5}}>C_{3} $. 
\end{proof}

\section{Troisi\`eme \'etape}

On notera $ b_{1} $ le r\'eel strictement sup\'erieur \`a $ 5 $ minimisant la fonction $ g_{1}(b,b+1+\nu,\delta)+(b+(b+1+\nu)+1+\delta)+2 $, pour $ \delta,\nu>0 $ fix\'es et arbitrairement petits. \begin{rem}
On peut en fait v\'erifier que $ b_{1}\in [8,9] $, pour $ \delta,\nu $ assez petits, et que le minimum obtenu est strictement inf\'erieur \`a $ 29 $. \end{rem}

 On pose $ b_{1}'=b_{1}+1+\nu $. On supposera dor\'enavant que \[ n+1>g_{1}(b_{1},b_{1}',\delta)+(b_{1}+b_{1}'+1+\delta)+2>b_{1}+b_{1}'+\delta+3 \] (ceci est en particulier vrai lorsque $ n\geqslant 28 $, d'apr\`es la remarque pr\'ec\'edente). Si $ P_{1}>1 $ est fix\'e quelconque et $ P_{2}=P_{1}^{b_{1}} $, $ P_{3}=P_{1}^{b_{1}'} $, alors par la proposition \ref{prop1}, on a : \begin{equation}\label{form1}
N(P_{1},P_{2},P_{3})=\sigma P_{1}^{n}P_{2}^{n}P_{3}^{n}+O(P_{1}^{n-\delta}P_{2}^{n-\delta}P_{3}^{n-\delta}).
\end{equation} 
On remarque d'autre part que \begin{align}\label{form2} N(P_{1},P_{2},P_{3})&=N_{1}(P_{1},P_{2},P_{3})+O\left(\sum_{\xx\in P_{1}\BB_{1}\cap\mathcal{A}_{1,\lambda}(\ZZ)^{c} }P_{2}^{n+1}P_{3}^{n+1}\right) \\ & =N_{1}(P_{1},P_{2},P_{3})+O\left(P_{1}^{n+1-\lambda}P_{2}^{n+1}P_{3}^{n+1}\right) \end{align}
(d'apr\`es le lemme \ref{complA}). A partir d'ici nous fixerons $ \lambda=\lceil b_{1}+b_{1}'+1+\delta\rceil $, de sorte que : \[ P_{1}^{n+1-\lambda}P_{2}^{n+1}P_{3}^{n+1}\ll P_{1}^{n-\delta}P_{2}^{n}P_{3}^{n}. \] Enfin, puisque $ n+1>g_{1}(b_{1},b_{1}',\delta)+\lambda +2 $ (i.e. $ K_{1}-2> g_{1}(b_{1},b_{1}',\delta)$), la proposition \ref{prop2} donne \begin{equation}\label{form3}
N_{1}(P_{1},P_{2},P_{3})=
P_{2}^{n}P_{3}^{n}\sum_{\xx\in P_{1}\BB_{1}\cap \mathcal{A}_{1,\lambda}(\ZZ) }\mathfrak{S}_{\xx}J_{\xx}+O(P_{1}^{n}P_{2}^{n-\delta}P_{3}^{n}).
\end{equation}
Ainsi, en regroupant les formules \eqref{form1}, \eqref{form2} et \eqref{form3}, on trouve : \[ P_{2}^{n}P_{3}^{n}\sum_{\xx\in P_{1}\BB_{1}\cap \mathcal{A}_{1,\lambda}(\ZZ) }\mathfrak{S}_{\xx}J_{\xx}=\sigma P_{1}^{n}P_{2}^{n}P_{3}^{n}+O(P_{1}^{n-\delta}P_{2}^{n}P_{3}^{n}), \] et donc \[\sum_{\xx\in P_{1}\BB_{1}\cap \mathcal{A}_{1,\lambda}(\ZZ) }\mathfrak{S}_{\xx}J_{\xx}=\sigma P_{1}^{n}+O(P_{1}^{n-\delta}), \] et cette relation est ind\'ependante de $ P_{2},P_{3} $. On a donc \'etabli le lemme ci-dessous : 
\begin{lemma}\label{lemplus}
Si l'on a $ n+1>g_{1}(b_{1},b_{1}',\delta)+(b_{1}+b_{1}'+1+\delta)+2 $ (en particulier, si $ n\geqslant 28 $), alors pour tout $ P_{1}\geqslant 1 $ : \[ \sum_{\xx\in P_{1}\BB_{1}\cap \mathcal{A}_{1,\lambda}(\ZZ) }\mathfrak{S}_{\xx}J_{\xx}=\sigma P_{1}^{n}+O(P_{1}^{n-\delta}). \]
\end{lemma}
Nous sommes \`a pr\'esent en mesure de d\'emontrer le r\'esultat suivant : 
\begin{prop}\label{prop3}
On suppose que l'on a $ P_{1}\leqslant P_{2}\leqslant P_{3} $, $ P_{2}=P_{1}^{b} $ et $ P_{3}=P_{1}^{b'} $. On suppose de plus que $ b'\leqslant b+1+\nu $ et que \[ n+1>g_{1}(b_{1},b_{1}',\delta)+(b_{1}+b_{1}'+1+\delta)+2>b_{1}+b_{1}'+\delta \] ($ \delta,\nu>0 $ \'etant fix\'es et arbitrairement petits). On a alors que  \[N_{1}(P_{1},P_{2},P_{3})=\sigma P_{1}^{n}P_{2}^{n}P_{3}^{n}+O(P_{1}^{n-\delta}P_{2}^{n}P_{3}^{n}). \]
\end{prop}
\begin{proof}
\begin{itemize}
\item Premier cas : on suppose $ b_{1}\leqslant b $. On a puisque $ b'\leqslant b+1+\nu $ : \begin{align*}
1+\frac{b'}{b}& \leqslant 1+1+\frac{1}{b}+\frac{\nu}{b} \\ &\leqslant  1+1+\frac{1}{b_{1}}+\frac{\nu}{b_{1}} \\ & =1+\frac{b_{1}+1+\nu}{b_{1}} \\ & =1+\frac{b_{1}'}{b_{1}}. 
\end{align*} 
\end{itemize}
On a \'egalement \[ \left(1-\frac{5}{b}-\delta\right)^{-1} \leqslant\left(1-\frac{5}{b_{1}}-\delta\right)^{-1} \] (en effet, puisque $ b_{1}\leqslant b $ et $ b_{1}\in [8,9] $, on a $ \frac{5}{b}+\delta\leqslant \frac{5}{b_{1}}+\delta<1 $) et \[ \left(\frac{3}{b}+2\delta\right)\leqslant \left(\frac{3}{b_{1}}+2\delta\right). \] Ceci implique \[ g_{1}(b,b',\delta)\leqslant g_{1}(b_{1},b_{1}',\delta)<K_{1}-2 .\]

Ainsi, par la proposition \ref{prop2}, on obtient \begin{align*}
N_{1}(P_{1},P_{2},P_{3}) &=
P_{2}^{n}P_{3}^{n}\sum_{\xx\in P_{1}\BB_{1}\cap \mathcal{A}_{1,\lambda}(\ZZ) }\mathfrak{S}_{\xx}J_{\xx}+O(P_{1}^{n}P_{2}^{n-\delta}P_{3}^{n}) \\ & =\sigma P_{1}^{n}P_{2}^{n}P_{3}^{n}+O(P_{1}^{n}P_{2}^{n-\delta}P_{3}^{n})
\end{align*}
(d'apr\`es le lemme \ref{lemplus}).
\item Deuxi\`eme cas : supposons $ b_{1}> b $. On a alors \[ b'\leqslant 1+b+\nu <1+b_{1}+\nu=b_{1}', \] donc \[ b'+b+1\leqslant b_{1}'+b_{1}+1<n+1 \] et on peut alors appliquer la proposition \ref{prop1}, et on a : \[ N(P_{1},P_{2},P_{3})=\sigma P_{1}^{n}P_{2}^{n}P_{3}^{n}+O(P_{1}^{n-\delta}P_{2}^{n-\delta}P_{3}^{n-\delta}).\] Par cons\'equent, en utilisant \eqref{form2}, puisque $ \lambda=\lceil b_{1}+b_{1}'+1+\delta\rceil \geqslant \lceil b+b'+1+\delta\rceil $, on trouve bien \[ N_{1}(P_{1},P_{2},P_{3}) =\sigma P_{1}^{n}P_{2}^{n}P_{3}^{n}+O(P_{1}^{n}P_{2}^{n-\delta}P_{3}^{n}). \]
\end{proof}
On a alors le r\'esultat suivant : 
\begin{prop}\label{prop4} On suppose que l'on a $ P_{1}\leqslant P_{2}\leqslant P_{3} $, $ P_{2}=P_{1}^{b} $ et $ P_{3}=P_{1}^{b'} $. On suppose de plus que $ b'\leqslant b+1+\nu $ et que \[ n+1>g_{1}(b_{1},b_{1}',\delta)+(b_{1}+b_{1}'+1+\delta)+2>b_{1}+b_{1}'+\delta \] ($ \delta,\nu>0 $ \'etant fix\'es et arbitrairement petits). On a alors que  \[N_{U}(P_{1},P_{2},P_{3})=\sigma P_{1}^{n}P_{2}^{n}P_{3}^{n}+O(P_{1}^{n-\delta}P_{2}^{n}P_{3}^{n}). \]
\end{prop}
\begin{proof}
Rappelons que, par d\'efinition : \begin{multline*} N_{U}(P_{1},P_{2},P_{3})=\card\{ (\xx,\yy,\zz)\in (\mathcal{A}_{1,\lambda}(\ZZ)\cap P_{1}\BB_{1})\times (\mathcal{A}_{2,\lambda}(\ZZ)\cap P_{2}\BB_{2}) \\ \times (\mathcal{A}_{3,\lambda}(\ZZ)\cap P_{3}\BB_{3}) \; | \; \max_{k}|B_{k}(\xx,\yy)|\neq 0, \; \max_{j}|B_{j}'(\xx,\zz)|\neq 0, \\ \; \max_{i}|B_{i}''(\yy,\zz)|\neq 0, \; 
 F(\xx,\yy,\zz)=0 \}.
\end{multline*}
On a donc \[ N_{U}(P_{1},P_{2},P_{3})=N_{1}(P_{1},P_{2},P_{3})+ O(T_{1})+O(T_{2})+O(T_{3})+O(T_{4}), \]
o\`u \begin{multline}\label{T1} T_{1}=\card\{(\xx,\yy,\zz)\in (\mathcal{A}_{1,\lambda}(\ZZ)\cap P_{1}\BB_{1}) \times (\mathcal{A}_{2,\lambda}(\ZZ)^{c}\cap P_{2}\BB_{2})  \\ \times (\ZZ^{n+1}\cap P_{3}\BB_{3})\; | \; F(\xx,\yy,\zz)=0\}, \end{multline}
\begin{multline}\label{T2} T_{2}=\card\{(\xx,\yy,\zz)\in (\mathcal{A}_{1,\lambda}(\ZZ)\cap P_{1}\BB_{1}) \times (\ZZ^{n+1}\cap P_{2}\BB_{2})  \\\times (\mathcal{A}_{3,\lambda}(\ZZ)^{c}\cap P_{3}\BB_{3}) \; | \; F(\xx,\yy,\zz)=0\}, \end{multline}
\begin{multline}\label{T3} T_{3}=\card\{(\xx,\yy,\zz)\in (\mathcal{A}_{1,\lambda}(\ZZ)\cap P_{1}\BB_{1}) \times (\mathcal{A}_{2,\lambda}(\ZZ)\cap P_{2}\BB_{2})  \\ \times(\mathcal{A}_{3,\lambda}(\ZZ)\cap P_{3}\BB_{3}) \; | \;  B_{j}'(\xx,\zz)=0 \; \forall j\}, \end{multline}
\begin{multline}\label{T4} T_{4}=\card\{(\xx,\yy,\zz)\in (\mathcal{A}_{1,\lambda}(\ZZ)\cap P_{1}\BB_{1}) \times (\mathcal{A}_{2,\lambda}(\ZZ)\cap P_{2}\BB_{2})  \\ \times(\mathcal{A}_{3,\lambda}(\ZZ)\cap P_{3}\BB_{3}) \; | \;  B_{i}''(\yy,\zz)=0 \; \forall i\}. \end{multline}
On remarque que, d'apr\`es le lemme \ref{complA} \[ T_{1}\ll P_{1}^{n+1}P_{2}^{n+1-\lambda}P_{3}^{n+1}=P_{1}^{n+1+b'-(\lambda-1)b}P_{2}^{n}P_{3}^{n}. \]
Or, on a fix\'e $ \lambda=\lceil b_{1}+b_{1}'+1+\delta\rceil $ , avec $ 5<b_{1}\leqslant b_{1}' $, on a donc clairement $ n+1+b'-(\lambda-1)b\leqslant n+1+b'-5b\leqslant n-1 $ puisque $ b'\leqslant b+1+\nu $ et $ b\geqslant 1 $. On a donc $ T_{1}\ll P_{1}^{n-1}P_{2}^{n}P_{3}^{n} $. De la m\^eme mani\`ere, on montre que $ T_{2}\ll P_{1}^{n-1}P_{2}^{n}P_{3}^{n} $. \\

Par ailleurs, pour $ \xx $ fix\'e, si $ B_{j}'(\xx,\zz)=0 $ pour tout $ j $, alors $ \zz\in V_{3,\xx}^{\ast} $. Par cons\'equent, puisque pour tout $ \xx\in \mathcal{A}_{1,\lambda}(\ZZ) $, $ \dim V_{3,\xx}^{\ast}<\lambda $, on a alors : \[ T_{3}\ll P_{1}^{n+1}P_{2}^{n+1}P_{3}^{\lambda}=P_{1}^{n-1}P_{2}^{n}P_{3}^{n}, \] car $ n+1>\lambda+3 $ (car $  n+1>g_{1}(b_{1},b_{1}',\delta)+(b_{1}+b_{1}'+1+\delta)+2 $, $ g_{1}(b_{1},b_{1}',\delta)\geqslant 2 $ et $ (b_{1}+b_{1}'+1+\delta)+2\geqslant \lambda +1 $ par hypoth\`ese). De m\^eme on montre que $ T_{4}\ll P_{1}^{n-1}P_{2}^{n}P_{3}^{n} $. En r\'esum\'e on a donc \[ N_{U}(P_{1},P_{2},P_{3})=N_{1}(P_{1},P_{2},P_{3})+O(P_{1}^{n-1}P_{2}^{n}P_{3}^{n}), \] et la proposition \ref{prop3} permet de conclure. 

\end{proof}

\section{Quatri\`eme \'etape }

Il nous reste donc \`a traiter le cas o\`u $ b'\geqslant b+1+\nu $ (i.e. $ P_{3}\geqslant P_{1}^{1+\nu}P_{2} $). Nous allons r\'esoudre ce probl\`eme en utilisant des r\'esultats de g\'eom\'etrie des r\'eseaux. \\

Commen\c{c}ons par introduire la d\'efinition suivante (issue de \cite[Definition 2.1]{Wi}) :
\begin{Def}
Soit $ S $ un sous-ensemble de $ \RR^{d} $, et soit $ c $ un entier tel que $ 0\leqslant c \leqslant d $. Pour $ M\in \NN $ et $ L>0 $, on dit que \emph{$ S $ appartient \`a $ \Lip(d,c,M,L) $} s'il existe $ M $ applications $ \phi : [0,1]^{d-c}\ra \RR^{d} $ v\'erifiant : \[ ||\phi(\xx)-\phi(\yy)||_{2}\leqslant L||\xx-\yy||_{2}, \] $ ||.||_{2} $ d\'esignant la norme euclidienne, telles que $ S $ soit recouvert par les images de ces applications. 
\end{Def}

On a le r\'esultat suivant (cf. \cite[Lemme 2]{MV}) : 
\begin{lemma}\label{geomnomb}
Soit $ S\subset\RR^{d} $ un ensemble bord\'e dont le bord $ \partial S $ appartient \`a $ \Lip(d,1,M,L) $. L'ensemble $ S $ est alors mesurable et si $ \Lambda $ est un r\'eseau de $ \RR^{d} $ de premier minimum successif $ \lambda_{1} $, on a \[ \left| \card(S\cap\Lambda)-\frac{\Vol(S)}{\det(\Lambda)}\right|\leqslant c(d)M\left(\frac{L}{\lambda_{1}}+1\right)^{d-1}, \] o\`u $ c(d) $ est une constante ne d\'ependant que de $ d $.
\end{lemma} 
 Pour un couple $ (\xx,\yy) \in \ZZ^{n+1}\times \ZZ^{n+1} $ fix\'e tel que $ \max_{k}|B_{k}(\xx,\yy)|\neq 0 $, on note $ H_{\xx,\yy} $ l'hyperplan de $ \RR^{n+1} $ donn\'e par \[ H_{\xx,\yy}=\{\zz\in \RR^{n+1 }\; | \; F(\xx,\yy,\zz)=\sum_{k=0}^{n}B_{k}(\xx,\yy)z_{k}=0\}. \]
 On note par ailleurs $ C_{\xx,\yy} $ le corps convexe $ \BB_{3}\cap H_{\xx,\yy} $ et $ \Lambda_{\xx,\yy} $ le r\'eseau $ \ZZ^{n+1}\cap H_{\xx,\yy}  $. Nous allons appliquer le lemme \ref{geomnomb} \`a $ S=P_{3}C_{\xx,\yy} $ et $ \Lambda=\Lambda_{\xx,\yy} $ vus respectivement comme un sous-ensemble et un r\'eseau de $ H_{\xx,\yy} $ que l'on identifiera \`a $ \RR^{n} $. Remarquons dans un premier temps que $ \partial\BB_{3}\in \Lip(n+1,1,2n,2 $) : en effet, pour toute face $ F $ du cube $ \BB_{3} $, on peut construire une application $ \phi_{F} : [0,1]^{n}\ra \RR^{n+1} $ qui est $ 2- $lipschitzienne et telle que $ \phi_{F}([0,1]^{n})=F $. Consid\'erons par exemple la face $ F $ correspondant aux points $ \zz\in \BB_{3} $ tels que $ z_{0}=1 $. On pose alors $ \phi_{F}(t_{1},...,t_{n})=(1,2t_{1}-1,...,2t_{n}-1) $ et on a bien $ \phi_{F}([0,1]^{n})=F $ et pour tous $ \tt,\tt'\in [0,1]^{n} $ \begin{align*} ||\phi_{F}(\tt)-\phi_{F}(\tt')||_{2} &\leqslant 2||\tt-\tt'||_{2}. \end{align*}
 
 Montrons \`a pr\'esent que $ \partial C_{\xx,\yy}\in \Lip(n,1,2n,2(n-1)\sqrt{n+1}) $. Une face du polytope $ C_{\xx,\yy} $ est obtenue en prenant l'intersection d'une face $ F $ de $ \BB_{3} $ avec $ H_{\xx,\yy} $. Consid\'erons par exemple l'intersection (suppos\'ee non vide) de la face $ F=\{\zz\in \BB_{3}\; | \; z_{0}=1\} $ avec $ H_{\xx,\yy} $. Pour simplifier les notations, on pose pour tout $ k\in \{0,...,n\} $, $ \alpha_{k}=B_{k}(\xx,\yy) $ de sorte que $ H_{\xx,\yy} $ a pour \'equation $ \alpha_{0}z_{0}+\alpha_{1}z_{1}+...+\alpha_{n}z_{n}=0 $ (les $ \alpha_{k} $ \'etant non tous nuls). Pour tout $ \zz\in F\cap H_{\xx,\yy}  $, on a alors $ \alpha_{0}+\alpha_{1}z_{1}+...+\alpha_{n}z_{n}=0  $, avec $ \max_{1\leqslant k \leqslant n}|\alpha_{k}|\neq 0 $ puisque l'intersection $ F\cap H_{\xx,\yy} $ est non vide. Supposons, par exemple, que $ \max_{1\leqslant k \leqslant n}|\alpha_{k}|=|\alpha_{n}| $, on a alors $ z_{n}=-\frac{\alpha_{0}}{\alpha_{n}}-\sum_{k=1}^{n-1}\frac{\alpha_{k}}{\alpha_{n}}z_{k} $, et on peut construire l'application $ \tilde{\phi}_{F} : [0,1]^{n-1}\ra H_{\xx,\yy}\subset \RR^{n+1} $ d\'efinie par \[ \tilde{\phi}_{F}(t_{1},...,t_{n-1})= \left(1,2t_{1}-1,...,2t_{n-1}-1,-\frac{\alpha_{0}}{\alpha_{n}}-\sum_{k=1}^{n-1}\frac{\alpha_{k}}{\alpha_{n}}(2t_{k}-1)\right). \] On remarque alors que $ F\cap H_{\xx,\yy}\subset \tilde{\phi}_{F}([0,1]^{n-1}) $ et que \begin{align*} ||\tilde{\phi}_{F}(\tt)-\tilde{\phi}_{F}(\tt')||_{2}&\leqslant \sqrt{n+1}||\tilde{\phi}_{F}(\tt)-\tilde{\phi}_{F}(\tt')||_{\infty} \\ &\leqslant \sqrt{n+1}\max\left(2,2\sum_{k=1}^{n-1}\frac{|\alpha_{k}|}{|\alpha_{n}|}\right)||\tt-\tt'||_{\infty} \\  &\leqslant 2(n-1)\sqrt{n+1}||\tt-\tt'||_{2}. \end{align*} On a donc $  \partial C_{\xx,\yy}\in \Lip(n,1,2n,2(n-1)\sqrt{n+1}) $ et par cons\'equent \[  \partial P_{3}C_{\xx,\yy}\in \Lip(n,1,2n,2(n-1)\sqrt{n+1}P_{3}). \] De plus puisque $ \Lambda_{\xx,\yy}\subset \ZZ^{n+1} $ le premier minimum successif de ce r\'eseau est sup\'erieur ou \'egal \`a $ 1 $. Ainsi, si l'on pose
\begin{equation}  N_{\Lambda_{\xx,\yy},\BB_{3}}(P_{3})=\{\zz\in P_{3}\BB_{3}\cap \ZZ^{n+1}\; | \; F(\xx,\yy,\zz)=0 \}=\card(\Lambda_{\xx,\yy}\cap P_{3}C_{\xx,\yy})\end{equation} le lemme \ref{geomnomb} nous donne : \[ \left| N_{\Lambda_{\xx,\yy},\BB_{3}}(P_{3})-\frac{\Vol(P_{3}C_{\xx,\yy})}{\det(\Lambda_{\xx,\yy})}\right|\ll_{n}P_{3}^{n-1}. \]
On a donc que \begin{equation}\label{reseau} N_{\Lambda_{\xx,\yy},\BB_{3}}(P_{3})=\frac{\Vol(C_{\xx,\yy})}{\det(\Lambda_{\xx,\yy})}P_{3}^{n}+O(P_{3}^{n-1}). \end{equation}
Par cons\'equent, si l'on note \begin{multline*} N'(P_{1},P_{2},P_{3})=\card\{ (\xx,\yy,\zz)\in (\mathcal{A}_{1,\lambda}(\ZZ)\cap P_{1}\BB_{1})\times (\mathcal{A}_{2,\lambda}(\ZZ)\cap P_{2}\BB_{2})\\ \times (\ZZ^{n+1}\cap P_{3}\BB_{3}) \; | \;   F(\xx,\yy,\zz)=0, \;\max_{k} |B_{k}(\xx,\yy)|\neq 0  \}, \end{multline*} 
on trouve \[ N'(P_{1},P_{2},P_{3})=P_{3}^{n}\sum_{\substack{\xx\in\mathcal{A}_{1,\lambda}(\ZZ)\cap P_{1}\BB_{1} \\ \yy \in \mathcal{A}_{2,\lambda}(\ZZ)\cap P_{2}\BB_{2} \\ \max_{k}|B_{k}(\xx,\yy)|\neq 0}}\frac{\Vol(C_{\xx,\yy})}{\det(\Lambda_{\xx,\yy})}+O\left(\sum_{\substack{(\xx,\yy)\in P_{1}\BB_{1}\times P_{2}\BB_{2} \\ \max_{k}|B_{k}(\xx,\yy)|\neq 0}}P_{3}^{n-1}\right). \]
On peut montrer par ailleurs (voir par exemple \cite[Lemme 1.(i)]{HB}) que \[ \det(\Lambda_{\xx,\yy})=\frac{||(B_{k}(\xx,\yy))_{k}||_{2}}{\PGCD_{k}(B_{k}(\xx,\yy))}. \]
On a ainsi \begin{multline*}
N'(P_{1},P_{2},P_{3}) =P_{3}^{n}\sum_{\substack{\xx\in\mathcal{A}_{1,\lambda}(\ZZ)\cap P_{1}\BB_{1} \\ \yy \in \mathcal{A}_{2,\lambda}(\ZZ)\cap P_{2}\BB_{2} \\ \max_{k}|B_{k}(\xx,\yy)|\neq 0}}\frac{\PGCD_{k}(B_{k}(\xx,\yy))\Vol(C_{\xx,\yy})}{||(B_{k}(\xx,\yy))_{k}||_{2}} \\ +O\left(P_{1}^{n+1}P_{2}^{n+1}P_{3}^{n-1}\right).
\end{multline*}
Par ailleurs, puisque l'on a suppos\'e $ b'\geqslant b+1+\nu $, on a \[ P_{1}^{n+1}P_{2}^{n+1}P_{3}^{n-1}\ll P_{1}^{n-\nu}P_{2}^{n}P_{3}^{n}. \] Ainsi, on trouve \begin{multline}\label{N'}
N'(P_{1},P_{2},P_{3}) =P_{3}^{n}\sum_{\substack{\xx\in\mathcal{A}_{1,\lambda}(\ZZ)\cap P_{1}\BB_{1} \\ \yy \in \mathcal{A}_{2,\lambda}(\ZZ)\cap P_{2}\BB_{2} \\ \max_{k}|B_{k}(\xx,\yy)|\neq 0}}\frac{\PGCD_{k}(B_{k}(\xx,\yy))\Vol(C_{\xx,\yy})}{||(B_{k}(\xx,\yy))_{k}||_{2}} \\ +O\left(P_{1}^{n-\nu}P_{2}^{n}P_{3}^{n}\right).
\end{multline}

Nous allons \`a pr\'esent montrer que \[ \sum_{\substack{\xx\in\mathcal{A}_{1,\lambda}(\ZZ)\cap P_{1}\BB_{1} \\ \yy \in \mathcal{A}_{2,\lambda}(\ZZ)\cap P_{2}\BB_{2} \\ \max_{k}|B_{k}(\xx,\yy)|\neq 0}}\frac{\PGCD_{k}(B_{k}(\xx,\yy))\Vol(C_{\xx,\yy})}{||(B_{k}(\xx,\yy))_{k}||_{2}}=\sigma P_{1}^{n}P_{2}^{n}+O(P_{1}^{n-\delta'}P_{2}^{n}) \] pour un certain $ \delta'>0 $. On remarque que \[N'(P_{1},P_{2},P_{3})=N_{U}(P_{1},P_{2},P_{3})+O(T_{5})+O(T_{3})+O(T_{4}) \] o\`u \begin{multline}\label{T5} T_{5}=\card\{(\xx,\yy,\zz)\in (\mathcal{A}_{1,\lambda}(\ZZ)\cap P_{1}\BB_{1}) \times (\mathcal{A}_{2,\lambda}(\ZZ)\cap P_{2}\BB_{2})  \\ \times (\mathcal{A}_{3,\lambda}(\ZZ)^{c}\cap P_{3}\BB_{3})\; | \; F(\xx,\yy,\zz)=0\}, \end{multline} et $ T_{3} $, $ T_{4} $ sont d\'efinis par \eqref{T3} et \eqref{T4}. Nous avons d\'ej\`a montr\'e que $ T_{3},T_{4}\ll P_{1}^{n-1}P_{2}^{n}P_{3}^{n} $. On a par ailleurs que \[ T_{5} \ll P_{1}^{n+1}P_{2}^{n+1}P_{3}^{n+1-\lambda}\ll P_{1}^{n-1}P_{2}^{n}P_{3}^{n} \] car $ \lambda\geqslant 4 $. Par cons\'equent \begin{equation}\label{N'NU}
N'(P_{1},P_{2},P_{3})=N_{U}(P_{1},P_{2},P_{3})+O(P_{1}^{n-1}P_{2}^{n}P_{3}^{n}). 
\end{equation}  Par la suite, on choisit $ P_{3} $ tel que $ b'=b+1+\nu $. 
D'apr\`es la proposition \ref{prop4}, la formule \eqref{N'NU} donne : \[ N'(P_{1},P_{2},P_{3})=\sigma P_{1}^{n}P_{2}^{n}P_{3}^{n}+O(P_{1}^{n-\delta}P_{2}^{n}P_{3}^{n}), \] et en utilisant \eqref{N'}, on a alors : \[ \sum_{\substack{(\xx,\yy)\in P_{1}\BB_{1}\times P_{2}\BB_{2} \\ \max_{k}|B_{k}(\xx,\yy)|\neq 0}}\frac{\PGCD_{k}(B_{k}(\xx,\yy))\Vol(C_{\xx,\yy})}{||(B_{k}(\xx,\yy))_{k}||_{2}}=\sigma P_{1}^{n}P_{2}^{n}+O(P_{1}^{n-\delta'}P_{2}^{n}), \] pour $ \delta'=\min\{\delta,\nu\} $, et ceci ind\'ependamment de $ P_{3} $. \\

Par cons\'equent, on d\'eduit de \eqref{N'} que \[ N'(P_{1},P_{2},P_{3}) =\sigma P_{1}^{n}P_{2}^{n}P_{3}^{n}+O(P_{1}^{n-\delta'}P_{2}^{n}P_{3}^{n}) \] pour tout $ n $ tel que $ (n+1)>g_{1}(b_{1},b_{1}',\delta)+(b_{1}+b_{1}'+1+\delta)+2 $, et tous $ (b,b') $ tels que $ b'\geqslant b+1+\nu $. Par cons\'equent, en utilisant la formule \eqref{N'NU} on a, sous les m\^emes hypoth\`eses : \[ N_{U}(P_{1},P_{2},P_{3}) =\sigma P_{1}^{n}P_{2}^{n}P_{3}^{n}+O(P_{1}^{n-\delta'}P_{2}^{n}P_{3}^{n}). \] En regroupant ce r\'esultat avec la proposition \ref{prop4} on obtient la proposition suivante :

\begin{prop}\label{NU}
Si $ n\geqslant 28 $, alors pour tous $ P_{1},P_{2},P_{3}\geqslant 1 $, il existe $ \delta'>0 $ tel que : \[ N_{U}(P_{1},P_{2},P_{3}) =\sigma P_{1}^{n}P_{2}^{n}P_{3}^{n}+O(P_{1}^{n-\delta'}P_{2}^{n}P_{3}^{n}). \]
\end{prop}

\section{Cinqui\`eme \'etape}

Nous allons \`a pr\'esent utiliser la formule obtenue pour $  N_{U}(P_{1},P_{2},P_{3}) $ dans la proposition \ref{NU} pour trouver une formule asymptotique pour $ N_{U}(B) $. Pour r\'esoudre ce probl\`eme, nous allons appliquer la m\'ethode  d\'evelopp\'ee par Blomer et Br\"{u}dern dans \cite{BB} pour le cas les hypersurfaces diagonales des espaces multiprojectifs, et reprise dans la section $ 9 $ de \cite{S2}. On consid\`ere une fonction $ h : \NN^{3}\ra [0,+\infty[ $. Conform\'ement aux notations de \cite{BB}, on dira que $ h $ est une $ (\beta,C,D,\alpha,\delta) $-fonction si elle v\'erifie les trois conditions suivantes : \begin{enumerate}
\item \label{I} On a \[ \sum_{\substack{l\leqslant L \\ m\leqslant M \\ n\leqslant N}}h(l,m,n)=CL^{\beta}M^{\beta}N^{\beta}+O(L^{\beta}M^{\beta}N^{\beta}\min\{L,M,N\}^{-\delta}) \] pour tous $ L,M,N\geqslant 1 $. 
\item \label{II}Il existe des fonctions $ c_{1},c_{2},c_{3} : \NN \ra [0,+\infty[ $ telles que : \[ \sum_{\substack{m \leqslant M \\ n\leqslant N }}h(l,m,n)=c_{1}(l)M^{\beta}N^{\beta}+O\left(l^{D}M^{\beta}N^{\beta}\min\{M,N\}^{-\delta}\right), \] uniform\'ement pour tous $ M,N\geqslant 1 $ et $ l\leqslant M^{\alpha}N^{\alpha} $, \[ \sum_{\substack{l \leqslant L \\ n\leqslant N }}h(l,m,n)=c_{2}(m)L^{\beta}N^{\beta}+O\left(m^{D}L^{\beta}N^{\beta}\min\{L,N\}^{-\delta}\right), \] uniform\'ement pour tous $ L,N\geqslant 1 $ et $ m\leqslant L^{\alpha}N^{\alpha} $,\[ \sum_{\substack{l \leqslant L \\ m\leqslant M }}h(l,m,n)=c_{3}(n)L^{\beta}M^{\beta}+O\left(n^{D}L^{\beta}M^{\beta}\min\{L,M\}^{-\delta}\right) \] uniform\'ement pour tous $ L,M\geqslant 1 $ et $ n\leqslant L^{\alpha}M^{\alpha} $.
\item \label{III}Il existe des fonctions $ \tilde{c_{1}},\tilde{c_{2}} ,\tilde{c_{3}} : \NN^{2} \ra [0,+\infty[ $ telles que : \[ \sum_{l\leqslant L}h(l,m,n)=\tilde{c_{1}}(m,n)L^{\beta} +O(\max\{m,n\}^{D}L^{\beta-\delta}) \] uniform\'ement pour tous $ L\geqslant 1 $ et $ \max\{m,n\}\leqslant L^{\alpha} $, \[ \sum_{m\leqslant M}h(l,m,n)=\tilde{c_{2}}(l,n)M^{\beta} +O(\max\{l,n\}^{D}M^{\beta-\delta}) \] uniform\'ement pour tous $ M\geqslant 1 $ et $ \max\{l,n\}\leqslant M^{\alpha} $, \[ \sum_{n\leqslant N}h(l,m,n)=\tilde{c_{3}}(l,m)N^{\beta} +O(\max\{l,m\}^{D}N^{\beta-\delta}) \] uniform\'ement pour tous $ N\geqslant 1 $ et $ \max\{l,m\}\leqslant N^{\alpha} $.
\end{enumerate}
On a la proposition suivante qui est un corollaire imm\'ediat de \cite[Th\'eor\`eme 2.1]{BB}: \begin{prop}\label{propfin}
Si $ h $ est une $ (\beta,C,D,\alpha,\delta) $-fonction, alors on a la formule asymptotique : \[ \sum_{lmn\leqslant P}h(l,m,n)=\frac{1}{2}C\beta^{2}P^{\beta}\log(P)^{2}+O(P^{\beta}\log(P)). \]
\end{prop}
Nous allons appliquer ce r\'esultat \`a la fonction \[ h(l_{1},l_{2},l_{3})=\card\{(\xx,\yy,\zz)\in U \; | \; |\xx|=l_{1}, \; |\yy|=l_{2}, \; |\zz|=l_{3}, \; F(\xx,\yy,\zz)=0 \} \] (en remarquant que $ N_{U}(B)=\sum_{l_{1}l_{2}l_{3}\leqslant B}h(l_{1},l_{2},l_{3}) $). Pour cela nous allons montrer que cette fonction est bien une $ (\beta,C,D,\alpha,\delta) $-fonction (pour des constantes $ C,\delta,\beta,\alpha,D $ que nous pr\'eciserons).

Remarquons que cette fonction v\'erifie bien la condition \ref{I} avec $ \beta=n $, d'apr\`es la proposition \ref{NU}. D'autre part, par le corollaire \ref{lecorollaire}, on a pour tout $ \xx\in \mathcal{A}_{1,\lambda}(\ZZ) $ et $ P_{2}\leqslant P_{3} $ : \[ N_{\xx}(P_{2},P_{3})=\mathfrak{S}_{\xx}J_{\xx}P_{2}^{n}P_{3}^{n}+O([\xx|^{4}P_{2}^{n-\delta}P_{3}^{n}) \] uniform\'ement pour tout $ \xx $ tel que $ |\xx|<(P_{2}^{1-3\theta_{2}}P_{3})^{\frac{1}{2}} $ pour un $ \theta_{2}<\frac{1}{5} $ fix\'e. Donc en choisissant $ \theta_{2}<\frac{1}{6} $, la formule est vraie pour tout $ \xx $ tel que  $ |\xx|\leqslant P_{2}^{\frac{1}{4}}P_{3}^{\frac{1}{2}} $. Si l'on note \[ N_{U,\xx}(P_{2},P_{3})=\card\{(\yy,\zz)\in (P_{2}\BB_{2}\times P_{3}\BB_{3})\cap \ZZ^{2n+2} \; | \; (\xx,\yy,\zz)\in U,\; F(\xx,\yy,\zz)=0\} \]
On a alors que \[ N_{U,\xx}(P_{2},P_{3})=N_{\xx}(P_{2},P_{3})+O(T_{1,\xx})+O(T_{2,\xx})+O(T_{3,\xx})+O(T_{4,\xx}) \] o\`u \begin{multline}
T_{1,\xx}=\card\{ (\yy,\zz)\in (\mathcal{A}_{2,\lambda}(\ZZ)^{c}\cap P_{2}\BB_{2})\times (\ZZ^{n+1}\cap P_{3}\BB_{3}) \\ \; | \; \max_{k}|B_{k}(\xx,\yy)|\neq 0, \; F(\xx,\yy,\zz)=0 \} 
\end{multline}
\begin{multline}
T_{2,\xx}=\card\{ (\yy,\zz)\in (\ZZ^{n+1}\cap P_{2}\BB_{2})\times (\mathcal{A}_{3,\lambda}(\ZZ)^{c}\cap P_{3}\BB_{3}) \\ \; | \; \max_{k}|B_{k}(\xx,\yy)|\neq 0, \; F(\xx,\yy,\zz)=0 \} 
\end{multline}
\begin{multline}
T_{3,\xx}=\card\{ (\yy,\zz)\in (\mathcal{A}_{2,\lambda}(\ZZ)\cap P_{2}\BB_{2})\times (\mathcal{A}_{3,\lambda}(\zz)\cap P_{3}\BB_{3}) \\ \; | \;  \max_{k}|B_{k}(\xx,\yy)|\neq 0, \; \forall j \;  B_{j}'(\xx,\zz)=0   \} 
\end{multline}
\begin{multline}
T_{4,\xx}=\card\{ (\yy,\zz)\in (\mathcal{A}_{2,\lambda}(\ZZ)\cap P_{2}\BB_{2})\times (\mathcal{A}_{3,\lambda}(\ZZ)\cap P_{3}\BB_{3}) \\ \; | \; \max_{k}|B_{k}(\xx,\yy)|\neq 0, \; \forall i\; B_{i}''(\xx,\zz)=0   \} 
\end{multline}
En reprenant la formule \eqref{reseau}, et en remarquant que, pour tous $ \xx,\yy $ \[ \Vol(C_{\xx,\yy})\ll 1 \] et \[ \frac{1}{|\det(\Lambda_{\xx,\yy})|}=\frac{||(B_{k}(\xx,\yy))_{k}||_{2}}{\PGCD_{k}(B_{k}(\xx,\yy))}\leqslant 1, \] on obtient :  \[ T_{1,\xx}\ll \sum_{\yy\in \mathcal{A}_{2,\lambda}^{c}\cap P_{2}\BB_{2}}P_{3}^{n} \ll P_{2}^{n+1-\lambda}P_{3}^{n}\ll P_{2}^{n-1}P_{3}^{n}. \]
On montre de m\^eme que \[ T_{2,\xx} \ll P_{2}^{n}P_{3}^{n-1}.\] Par ailleurs, si $ B_{j}'(\xx,\zz)=0  $ pour tout $ j $, alors $ \zz\in V_{3,\xx}^{\ast} $, et donc si $ \xx\in \mathcal{A}_{1,\lambda}(\ZZ) $ : \[ T_{3,\xx} \ll P_{2}^{n+1}P_{3}^{\lambda}\ll P_{2}^{n-1}P_{3}^{n}. \] De m\^eme on montre \[ T_{4,\xx} \ll P_{2}^{n-1}P_{3}^{n}. \] Par cons\'equent, si $ P_{2}\leqslant P_{3} $, on a \[ N_{U,\xx}(P_{2},P_{3})=N_{\xx}(P_{2},P_{3})+O(P_{2}^{n-1}P_{3}^{n})=\mathfrak{S}_{\xx}J_{\xx}P_{2}^{n}P_{3}^{n}+O([\xx|^{4}P_{2}^{n-\delta}P_{3}^{n})\] uniform\'ement pour tout $ \xx $ tel que $ |\xx|\leqslant  P_{2}^{\frac{1}{4}}P_{3}^{\frac{1}{2}} $. Par les m\^emes calculs, on obtient exactement le m\^eme r\'esultat dans le cas $ P_{3}\leqslant P_{2} $. Par cons\'equent, on trouve que :  \begin{align*} \sum_{l_{2}\leqslant P_{2}, \; l_{3}\leqslant P_{3}}h(l_{1},l_{2},l_{3}) & =\sum_{\xx\in \mathcal{A}_{1,\lambda}(\ZZ), |\xx|=l_{1}}N_{U,\xx}(P_{2},P_{3}) \\ & =c_{1}(l_{1})P_{2}^{n}P_{3}^{n}+ O\left(l_{1}^{n+4}P_{2}^{n}P_{3}^{n}\min\{P_{2},P_{3}\}^{-\delta}\right), \end{align*} uniform\'ement pour tout $ l_{1}\leqslant P_{2}^{\frac{1}{4}}P_{3}^{\frac{1}{4}} $, avec \[  c_{1}(l_{1})=\sum_{\xx\in \mathcal{A}_{1,\lambda}(\ZZ), |\xx|=l_{1}}\mathfrak{S}_{\xx}J_{\xx}. \] Donc $ h $ v\'erifie le premier point de la condition \ref{II} avec $ D=n+4 $, et $ \alpha=\frac{1}{4} $. Par sym\'etrie, on montre de m\^eme que $ h $ v\'erifie les deux autres points de la condition \ref{II}. \\

On fixe $ (\xx,\yy)\in \mathcal{A}_{1,\lambda}(\ZZ)\times \mathcal{A}_{2,\lambda}(\ZZ) $ tels que $ \max_{k}|B_{k}(\xx,\yy)|\neq 0 $. Si l'on note \begin{multline*} N_{U,\xx,\yy}(P_{3})=\card\{\zz\in \ZZ^{n+1}\cap P_{3}\BB_{3} \; | \; (\xx,\yy,\zz)\in U, \; F(\xx,\yy,\zz)=0 \},  \end{multline*} on remarque que \[  N_{U,\xx,\yy}(P_{3})=N_{\Lambda_{\xx,\yy},\BB_{3}}(P_{3})+O(T_{1,\xx,\yy})+O(T_{2,\xx,\yy})+O(T_{3,\xx,\yy}), \] avec 
\[
T_{1,\xx,\yy}=\card\{ \zz\in \mathcal{A}_{3,\lambda}(\ZZ)^{c}\cap P_{3}\BB_{3}\},
\]
\[
T_{2,\xx,\yy}=\card\{ \zz\in \ZZ^{n+1}\cap P_{3}\BB_{3}, \; B_{j}'(\xx,\zz)=0 \; \forall j\},
\]
\[
T_{3,\xx,\yy}=\card\{ \zz\in \ZZ^{n+1}\cap P_{3}\BB_{3}, \; B_{i}''(\yy,\zz)=0 \; \forall i\}. 
\]
On a alors imm\'ediatement \[ T_{1,\xx,\yy}\ll P_{3}^{n+1-\lambda}\ll P_{3}^{n-1} , \]\[ T_{2,\xx,\yy}\ll P_{3}^{\lambda}\ll P_{3}^{n-1}, \]\[ T_{3,\xx,\yy}\ll P_{3}^{\lambda}\ll P_{3}^{n-1}. \] Par cons\'equent \begin{align*}
 N_{U,\xx,\yy}(P_{3}) & =N_{\Lambda_{\xx,\yy},\BB_{3}}(P_{3})+O(P_{3}^{n-1}) \\ & = \frac{\PGCD_{k}(B_{k}(\xx,\yy))\Vol(C_{\xx,\yy})}{||(B_{k}(\xx,\yy))_{k}||_{2}}P_{3}^{n}+O(P_{3}^{n-1}). 
\end{align*}
On a ainsi : \begin{align*} \sum_{ l_{3}\leqslant P_{3}}h(l_{1},l_{2},l_{3}) & =\sum_{\substack{\xx\in \mathcal{A}_{1,\lambda}(\ZZ), |\xx|=l_{1}\\ \yy\in \mathcal{A}_{2,\lambda}(\ZZ), |\yy|=l_{2} \\ \max_{k}|B_{k}(\xx,\yy)|\neq 0}}N_{U,\xx,\yy}(P_{3}) \\ & =\tilde{c_{3}}(l_{1},l_{2})P_{3}^{n}+ O\left(l_{1}^{n}l_{2}^{n}P_{3}^{n-1}\right), \end{align*} et donc $ h $ v\'erifie le troisi\`eme point de la condition \ref{III} pour un $ \alpha $ quelconque et pour $ D=2n $, et par sym\'etrie, cette condition est enti\`erement v\'erifi\'ee. \\

On a donc montr\'e que $ h $ est une $ (n,\sigma,2n,\frac{1}{4},\delta) $-fonction, et donc en appliquant la proposition \ref{propfin}, on trouve : 

\begin{prop}\label{preconclusion}
Si $ n\geqslant 28 $, alors pour tout $ B\geqslant 1 $, on a la formule asymptotique : \[ N_{U}(B)=\frac{1}{2}n^{2}\sigma B^{n}\log(B)^{2}+O\left(B^{n}\log(B)\right). \]
\end{prop}

\section{Conclusion et interpr\'etation des constantes}

Nous pouvons finalement calculer  le cardinal \begin{multline*} \tilde{N}_{U}(B)=\card\{(\xx,\yy,\zz)\in U\cap(\ZZ^{n+1}\times \ZZ^{n+1}\times \ZZ^{n+1}) \; | \; \\ (x_{0},...,x_{n}), (y_{0},...,y_{n}), (z_{0},...,z_{n}) \; \primitifs, \;  F(\xx,\yy,\zz) =0, \; H'(\xx,\yy,\zz)\leqslant B\}. \end{multline*}

On remarque en effet que si $ N_{d,e,f}(B) $ d\'esigne
\begin{multline*}\card\{(d\xx,e\yy,f\zz)\in U\cap(d\ZZ^{n+1}\times e\ZZ^{n+1}\times f\ZZ^{n+1}) \; | \; \\  F(\xx,\yy,\zz) =0, \; H'(d\xx,e\yy,f\zz)\leqslant B\}=N_{U}(B/def) \end{multline*} et
 \begin{multline*} \tilde{N}_{k,l,m}(B)=\card\{(k\xx,l\yy,m\zz)\in U\cap(k\ZZ^{n+1}\times l\ZZ^{n+1}\times m\ZZ^{n+1}) \; | \; \\ (x_{0},...,x_{n}), (y_{0},...,y_{n}), (z_{0},...,z_{n}) \; \primitifs \;,  F(\xx,\yy,\zz) =0, \; H'(k\xx,l\yy,m\zz)\leqslant B\} \\ =\tilde{N}_{U}(B/klm ) \end{multline*} (pour $ d,e,f,k,l,m \in \NN $), alors on a 
  \[ N_{d,e,f}(B)=\sum_{d|k}\sum_{e|l}\sum_{f|m}\tilde{N}_{k,l,m}(B). \] Par inversions de M\"{o}bius successives appliqu\'ees \`a $ (d,e,f)=(1,1,1) $, on obtient : \begin{align*} \tilde{N}_{U}(B)& =\tilde{N}_{(1,1,1)}(B)  =\sum_{k\in \NN^{\ast}}\mu(k)\sum_{l\in \NN^{\ast}}\mu(l) \sum_{m\in\NN^{\ast}}\mu(m)N_{k,l,m}(B)\\ & =\sum_{k,l,m \in \NN^{\ast}}\mu(k)\mu(l)\mu(m)N_{U}(B/klm) \\ & = \frac{1}{2}\sigma\sum_{k,l,m \in \NN^{\ast}}\frac{\mu(k)\mu(l)\mu(m)}{k^{n}l^{n}m^{n}}n^{2}B^{n}\log(B)^{2} +O(B^{n}\log(B)).
 \end{align*} 
 On remarque que \[ \sum_{k,l,m \in \NN^{\ast}}\frac{\mu(k)\mu(l)\mu(m)}{k^{n}l^{n}m^{n}}=\left(\sum_{k\in \NN^{\ast}}\frac{\mu(k)}{k^{n}}\right)^{3}, \] et que \[\sum_{k\in \NN^{\ast}}\frac{\mu(k)}{k^{n}}=\prod_{p \in \mathcal{P}}\left(1-\frac{1}{p^{n}}\right) \] ($ \mathcal{P} $ d\'esignant l'ensemble des entiers premiers). 
En rappelant que l'on a $ \mathcal{N}_{U}(B)=\frac{1}{8}\tilde{N}_{U}(B^{\frac{1}{n}}) $, on a donc finalement d\'emontr\'e le r\'esultat suivant : 
\begin{prop}\label{conclusion}
Pour tout $ n\geqslant 28 $, on a : \[\mathcal{N}_{U}(B)=\frac{1}{16}\sigma'  B\log(B)^{2}+O(B\log(B)), \]
lorsque $ B \ra \infty $, o\`u l'on a not\'e $ \sigma'=\sigma\prod_{p \in \mathcal{P}}\left(1-\frac{1}{p^{n}}\right)^{3} $.
\end{prop}

Nous allons \`a pr\'esent donner une interpr\'etation des constantes introduites, et constater finalement que l'expression obtenue est bien en accord avec les formules conjectur\'ees par Peyre dans \cite{Pe}. \\

Dans tout ce qui va suivre, on notera $ \pi $ la projection \begin{multline*}\pi : \AA_{\QQ}^{3(n+1)}\setminus \left((\{\0\}\times\AA_{\QQ}^{n+1}\times \AA_{\QQ}^{n+1})\cup(\AA_{\QQ}^{n+1}\times\{\0\}\times \AA_{\QQ}^{n+1})\right. \\ \left. \cup(\AA_{\QQ}^{n+1}\times \AA_{\QQ}^{n+1}\times\{\0\})\right)\ra \PP_{\QQ}^{n}\times \PP_{\QQ}^{n}\times \PP_{\QQ}^{n} \end{multline*}

On note $ W=\pi^{-1}(V) $. Si $ (\xx,\yy,\zz)\in W $ est un point lisse avec, par exemple, $ B_{k_{1}}(\xx,\yy)\neq 0 $ pour un certain $ k_{1}\in \{0,...,n\} $, alors la forme de Leray $ \omega_{L} $ sur $ W $ est donn\'ee par \begin{multline*} \omega_{L}(\xx,\yy,\zz)=\frac{(-1)^{n+1-k_{1}}}{B_{k_{1}}(\xx,\yy)}dx_{0}\wedge...\wedge dx_{n}\wedge dy_{0}\wedge ...\wedge dy_{n} \\ \wedge dz_{0}\wedge...\wedge \widehat{dz_{k_{1}}}\wedge ...\wedge dz_{n}(\xx,\yy,\zz).  \end{multline*} Pour toute place $ \nu\in \Val(\QQ) $ la forme de Leray induit une mesure locale $ \omega_{L,\nu} $.

\subsection{\'Etude de l'int\'egrale $ J $} 

 Rappelons que l'on a \[ J =\int_{\RR}\int_{\BB_{1}\times\BB_{2}\times\BB_{3}}e(\beta F(\xx,\yy,\zz))d\xx d\yy d\zz d\beta. \]
et cette int\'egrale est absolument convergente (cf. lemme \ref{ConvergeJ}). On pose par ailleurs : \[ \sigma_{\infty}(V)=\int_{W\cap [-1,1]^{3n+3}}\omega_{L,\infty}. \]
Nous allons montrer que l'int\'egrale $ J $ co\"{i}cide avec $ \sigma_{\infty}(W) $. Il nous suffit de le v\'erifier localement i.e. montrons que pour tout ouvert $ U $ de $ \BB_{1}\times \BB_{2}\times \BB_{3} $ sur lequel, par exemple, $ B_{n}(\xx,\yy)\neq 0 $, \[ \sigma_{\infty}(U)=\int_{U\cap[-1,1]^{3n+3}} \omega_{L,\infty}=\int_{U\cap[-1,1]^{3n+3}}\frac{1}{|B_{n}(\xx,\yy)|}d\xx d\yy d\hat{\zz}, \] (avec $ d\hat{\zz}=dz_{0}...dz_{n-1} $) co\"{i}ncide avec \[ 
J_{U}=\int_{\RR}\int_{U\cap (\BB_{1}\times \BB_{2}\times \BB_{3})}e(\beta F(\xx,\yy,\zz))d\xx d\yy d\zz d\beta. \]Consid\`erons donc un tel ouvert $ U $. De la m\^eme mani\`ere que pour le lemme \ref{ConvergeJ}, on montre que $ J_{U}=\lim_{\mu\ra \infty}J_{U}(\mu) $, o\`u \[ J_{U}(\mu)=\int_{-\mu}^{\mu}\int_{U\cap(\BB_{1}\times\BB_{2}\times\BB_{3})}e(\beta F(\xx,\yy,\zz))d\xx d\yy d\zz d\beta \] pour tout $ \mu>0 $. On peut r\'e\'ecrire l'int\'egrale $ J_{U}(\mu) $ sous la forme : \begin{align*} J_{U}(\mu)&=\int_{U\cap (\BB_{1}\times\BB_{2}\times\BB_{3})}\left(\int_{-\mu}^{\mu}e(\beta F(\xx,\yy,\zz))d\beta\right) d\xx d\yy d\zz \\ & =\int_{U\cap(\BB_{1}\times\BB_{2}\times\BB_{3})}\frac{\sin(2\pi\mu F(\xx,\yy,\zz))}{\pi F(\xx,\yy,\zz)}d\xx d\yy d\zz \end{align*} 
On remplace ensuite la variable $ z_{n} $ par $ t=F(\xx,\yy,\zz) $, et on note \[ \chi(t)=\left\{ \begin{array}{l}
1 \; \; \mbox{si} \; \; z_{n}=\frac{t}{B_{n}(\xx,\yy)}-\sum_{k=0}^{n-1}\frac{B_{k}(\xx,\yy)}{B_{n}(\xx,\yy)}z_{k}\in [-1,1] \; \mbox{et} \; (\xx,\yy,\zz)\in U \\ 0 \; \; \mbox{sinon}
\end{array}\right.  \]

On a alors (si $ A=\sum_{i,j}|\alpha_{i,j,n}| $) : \[ J_{U}=\int_{-A}^{A}\frac{\sin(2\pi\mu t)}{\pi t}\int_{[-1,1]^{3n+2}}\frac{\chi(t)}{|B_{n}(\xx,\yy)|}d\xx d\yy d\hat{\zz} dt
\]
Si l'on note \[ \phi(t)=\int_{[-1,1]^{3n+2}}\frac{\chi(t)}{|B_{n}(\xx,\yy)|}d\xx d\yy d\hat{\zz}, \]
on remarque que cette fonction est \`a variations born\'ees. Par cons\'equent par application des r\'esultats d'analyse de Fourier (voir \cite[9.43]{W-W}) on a que \begin{align*} J_{U}=\phi(0)& =\int_{[-1,1]^{3n+2}}\frac{\chi(0)}{|B_{n}(\xx,\yy)|}d\xx d\yy d\hat{\zz} \\ & =\int_{U\cap[-1,1]^{3n+3}} \omega_{L,\infty}= \sigma_{\infty}(U). \end{align*} Remarquons que ces calculs constituent un \'equivalent du travail effectu\'e par Igusa dans \cite[\S IV.6]{Ig} pour le cas les int\'egrales de fonctions indicatrices. \\

Nous allons \`a pr\'esent interpr\'eter cette constante $ \sigma_{\infty} $ en termes de mesures de Tamagawa. Rappelons que (avec les notations de \cite{S2}) la mesure $ \tau_{\infty}=\omega_{\infty} $ est d\'efinie localement sur l'ouvert \[ U_{0,0,0}=\{([\xx],[\yy],[\zz])\in \PP^{n}\times \PP^{n}\times \PP^{n} \; |\; x_{0}y_{0}z_{0}\neq 0,\; B_{n}(\xx,\yy)\neq 0\} \] par \[ \omega_{\infty}=\frac{du_{1}...du_{n}dv_{1}...dv_{n}dw_{1}...dw_{n-1}}{h_{\infty}(\uu,\vv,\ww)|B_{n}(\uu,\vv)|} \]
o\`u $ \uu=(1,u_{1},...,u_{n}) $, $ \vv=(1,v_{1},...,v_{n}) $, $ \ww=(1,w_{1},...,w_{n}) $ et \[ h_{\infty}(\xx,\yy,\zz)=h_{\infty}^{1}(\xx) h_{\infty}^{2}(\yy) h_{\infty}^{3}(\zz), \] avec \[ h_{\infty}^{1}(\xx)=\max_{0\leqslant i\leqslant n}|x_{i}|^{n},\; \;  h_{\infty}^{2}(\yy)=\max_{0\leqslant j\leqslant n}|y_{j}|^{n},\; \; h_{\infty}^{3}(\zz)=\max_{0\leqslant k\leqslant n}|z_{k}|^{n}. \]
Nous allons d\'emontrer le r\'esultat suivant : 
\begin{lemma}\label{Tamagawa infini} On a 
\[ \tau_{\infty}=\frac{n^{3}}{8}\sigma_{\infty}. \]
\end{lemma}
\begin{proof}
On d\'emontre le r\'esultat localement i.e. montrons que pour tout ouvert $ U $ par exemple inclus dans $ U_{0,0,0} $ d\'efini plus haut (les autres cas se traitant de mani\`ere analogue) on a $ \tau_{\infty}(U)=\frac{n^{3}}{8}\sigma_{\infty}(\pi^{-1}(U)) $. Par d\'efinition de la mesure de Leray, pour un tel ouvert $ U $, on a \[ \sigma_{\infty}(\pi^{-1}(U))=\int_{\substack{\pi^{-1}(U)\cap \{\max_{i}|x_{i}|\leqslant 1\\\max_{j}|y_{j}|\leqslant 1, \;  \max_{k}|z_{k}|\leqslant 1\}}}\frac{d\xx d\yy d\hat{\zz}}{|B_{n}(\xx,\yy)|} \] avec $ d\hat{\zz}=dz_{0}...dz_{n-1} $. On remarque que \[ \max_{i}|x_{i}|=1\; \Leftrightarrow \; |x_{0}|\leqslant \left(\max_{i\neq 0}\frac{|x_{i}|}{|x_{0}|}\right)^{-1}. \] On applique alors les changements de variables $ x_{i}=x_{0}u_{i} $, $ y_{j}=y_{0}v_{j} $ et $ z_{k}=z_{0}w_{k} $ dans l'int\'egrale ci-dessus. On a alors que \begin{multline*} (\xx,\yy,\zz)\in [-1,1]^{3n+3} \; \Leftrightarrow \; |x_{0}|\leqslant (\max_{i} |u_{i}|)^{-1}, \;  |y_{0}|\leqslant (\max_{j} |v_{j}|)^{-1}, \\ |z_{0}|\leqslant (\max_{k} |w_{k}|)^{-1}. \;\end{multline*} On obtient alors  \begin{multline*} \sigma_{\infty}(\pi^{-1}(U))\\= \int_{U}\frac{1}{|B_{n}(\uu,\vv)|} \left(\int_{\substack{|x_{0}|^{n}\leqslant h_{\infty}^{1}(\xx) \\  |y_{0}|^{n}\leqslant h_{\infty}^{2}(\yy) \\ |z_{0}|^{n}\leqslant h_{\infty}^{3}(\zz)}} |x_{0}|^{n-1}|y_{0}|^{n-1}|z_{0}|^{n-1}dx_{0}dy_{0}dz_{0}\right)d\uu d\vv d\hat{\ww} \\ =\frac{8}{n^{3}}\int_{U}\frac{d\uu d\vv d\hat{\ww}}{h_{\infty}(\uu,\vv,\ww)|B_{n}(\uu,\vv)|}=\frac{8}{n^{3}}\int_{U}\omega_{\infty} \end{multline*}

\end{proof}

\subsection{\'Etude de la s\'erie $ \mathfrak{S} $}
Rappelons que l'on a \[  \mathfrak{S}=\sum_{q=1}^{\infty}A(q) \] en notant \[ A(q)=q^{-3n-3}\sum_{\substack{0\leqslant a< q \\ \PGCD(a,q)=1}}\sum_{\bb,\bb',\bb''\in (\ZZ/q\ZZ)^{n+1}}e\left(\frac{a}{q}F(\bb,\bb',\bb'')\right)\]
et nous avons vu d'autre part (cf. lemme \ref{convergeS}) que cette s\'erie converge absolument. 

\begin{lemma}
Si $ \PGCD(q_{1},q_{2})=1 $, alors on a \[ A(q_{1}q_{2})=A(q_{1})A(q_{2}), \] autrement dit, la fonction $ A $ est multiplicative. \end{lemma}

\begin{proof}
On remarque dans un premier temps que \begin{equation}
A(q)=q^{-3n-3}\sum_{\substack{0\leqslant a< q \\ \PGCD(a,q)=1}}\sum_{\bb,\bb'\in (\ZZ/q\ZZ)^{n+1}}\prod_{k=0}^{n}\left(\sum_{b\in \ZZ/q\ZZ}e\left(\frac{a}{q}B_{k}(\bb,\bb')b\right)\right).
\end{equation}
Or on a \[ \sum_{b\in \ZZ/q\ZZ}e\left(\frac{a}{q}B_{k}(\bb,\bb')b\right)=\left\{\begin{array}{ll} q & \mbox{si} \; B_{k}(\bb,\bb')\equiv 0\; (q) \\ 0 & \mbox{sinon}

\end{array}\right.\]
On a donc \begin{align*}
A(q)& =q^{-2n-2}\sum_{a\in (\ZZ/q\ZZ)^{\ast}}\card\{\bb,\bb'\in (\ZZ/q\ZZ)^{n+1}\; |\;B_{k}(\bb,\bb')\equiv 0, \; \forall k \} \\ & =\varphi(q)q^{-2n-2}\card\{\bb,\bb'\in (\ZZ/q\ZZ)^{n+1}\; |\;B_{k}(\bb,\bb')\equiv 0\; (q), \; \forall k \}.
\end{align*}
Or si l'on a $ q=q_{1}q_{2} $, par le th\'eor\`eme chinois : \begin{multline*}
\card\{\bb,\bb'\in (\ZZ/q\ZZ)^{n+1}\; |\;B_{k}(\bb,\bb')\equiv 0\; (q), \; \forall k \} \\ =\card\{\bb_{1},\bb_{1}'\in (\ZZ/q_{1}\ZZ)^{n+1}\; |\;B_{k}(\bb_{1},\bb_{1}')\equiv 0\; (q_{1}), \; \forall k \} \\ .\card\{\bb_{2},\bb_{2}'\in (\ZZ/q_{2}\ZZ)^{n+1}\; |\;B_{k}(\bb_{2},\bb_{2}')\equiv 0\; (q_{2}), \; \forall k \}. 
\end{multline*}
On en d\'eduit que $ A $ est bien multiplicative. 
\end{proof}
Puisque $ A $ est multiplicative et absolument convergente, on a la formule : \[ \mathfrak{S}=\sum_{q=1}^{\infty}A(q)=\prod_{p\in\mathcal{P}}\sigma_{p}  \] o\`u \[ \sigma_{p}=\sum_{k=0}^{\infty}A(p^{k}). \] Par la suite on note pour tout $ q\in \NN^{\ast} $, \begin{equation}
M(q)=\card\{(\bb,\bb',\bb'')\in (\ZZ/q\ZZ)^{3n+3} \; |\; F(\bb,\bb',\bb'')\equiv 0 \; (q) \}
\end{equation}
On peut alors interpr\'eter $ \sigma_{p} $ \`a l'aide du r\'esultat suivant :
\begin{lemma}\label{sigmap}
On a que pour tout $ N>0 $ : \[ \sum_{k=0}^{N}A(p^{k})=\frac{M(p^{N})}{p^{N(3n+2)}}, \] et par cons\'equent : \[ \sigma_{p}=\lim_{k\ra \infty}\frac{M(p^{k})}{p^{k(3n+2)}}. \]
\end{lemma}
\begin{proof}
On remarque dans un premier temps que \begin{align*} M(q)&=q^{-1}\sum_{t=0}^{q-1}\sum_{\bb,\bb',\bb''\in (\ZZ/q\ZZ)^{n+1}}e\left(\frac{t}{q}F(\bb,\bb',\bb'')\right) \\ &=q^{-1}\sum_{q_{1}|q}\sum_{\substack{0\leqslant a< q_{1} \\ \PGCD(a,q_{1})=1}}\sum_{\bb,\bb',\bb''\in (\ZZ/q\ZZ)^{n+1}}e\left(\frac{a}{q_{1}}F(\bb,\bb',\bb'')\right). \end{align*}
On a donc, si $ q=p^{N} $ : 
\begin{align*} M(p^{N}) &=p^{-N}\sum_{k=0}^{N}\sum_{\substack{0\leqslant a< p^{k} \\ \PGCD(a,p^{k})=1}}\sum_{\bb,\bb',\bb''\in (\ZZ/p^{N}\ZZ)^{n+1}}e\left(\frac{a}{p^{k}}F(\bb,\bb',\bb'')\right) \\ &=p^{-N}\sum_{k=0}^{N}\sum_{\substack{0\leqslant a< p^{k} \\ \PGCD(a,p^{k})=1}}\left(\frac{p^{N}}{p^{k}}\right)^{3n+3}\sum_{\bb,\bb',\bb''\in (\ZZ/p^{k}\ZZ)^{n+1}}e\left(\frac{a}{p^{k}}F(\bb,\bb',\bb'')\right)\\ &=p^{(3n+2)N}\sum_{k=0}^{N}\sum_{\substack{0\leqslant a< p^{k} \\ \PGCD(a,p^{k})=1}}p^{-(3n+3)k}\sum_{\bb,\bb',\bb''\in (\ZZ/p^{k}\ZZ)^{n+1}}e\left(\frac{a}{p^{k}}F(\bb,\bb',\bb'')\right)\\ &=p^{(3n+2)N}\sum_{k=0}^{N}A(p^{k}), \end{align*} d'o\`u le r\'esultat.
\end{proof}

Nous allons \`a pr\'esent \'etudier le lien entre les constantes $ \sigma_{p} $ et la mesure de Tamagawa $ \tau_{p} $ d\'efinie (avec les notations de \cite{S2}) par : \[ \tau_{p}=\left(1-\frac{1}{p}\right)^{3}\omega_{p} \] o\`u $ \omega_{p} $ est la mesure d\'efinie localement sur $ V(\QQ_{p})\cap U_{0,0,0} $ par \[ \omega_{p}=\frac{du_{1,p}...du_{n,p}dv_{1,p}...dv_{n,p}dw_{1,p}...dw_{n-1,p}}{h_{p}(\uu,\vv,\ww)|B_{n}(\uu,\vv)|_{p}} \]
o\`u $ \uu=(1,u_{1},...,u_{n}) $, $ \vv=(1,v_{1},...,v_{n}) $, $ \ww=(1,w_{1},...,w_{n}) $ et \[ h_{p}(\xx,\yy,\zz)=h_{p}^{1}(\xx) h_{p}^{2}(\yy) h_{p}^{3}(\zz), \] avec \[ h_{p}^{1}(\xx)=\max_{0\leqslant i\leqslant n}|x_{i}|_{p}^{n},\; \;  h_{p}^{2}(\yy)=\max_{0\leqslant j\leqslant n}|y_{j}|_{p}^{n},\; \; h_{p}^{3}(\zz)=\max_{0\leqslant k\leqslant n}|z_{k}|_{p}^{n}. \]

\begin{lemma}\label{sigmap0}
Soit $ p\in \mathcal{P} $, on pose : \[ a(p)=\left(1-\frac{1}{p}\right)^{3}\left(1-\frac{1}{p^{n}}\right)^{-3}. \]
On a alors \[ \int_{\substack{W(\QQ_{p})\cap \{h_{p}^{1}(\xx)\leqslant 1 \\h_{p}^{2}(\yy)\leqslant 1, \; h_{p}^{3}(\zz)\leqslant 1 \}}}\omega_{L,p}(\xx,\yy,\zz)=a(p)\omega_{p}(V(\QQ_{p})). \]
\end{lemma}
\begin{proof}
Il suffit de montrer que pour tout ouvert $ U $ de $ \AA_{\QQ_{p}}^{3n-1}\subset \PP_{\QQ_{p}}^{n}\times \PP_{\QQ_{p}}^{n}\times \PP_{\QQ_{p}}^{n-1} $ de la forme $ U_{1}\times U_{2}\times U_{3} $, tel que pour tout $ ([\xx],[\yy],[\zz])\in U $ on a (par exemple) $ x_{0}y_{0}z_{0}\neq 0 $ et $ B_{n}(\xx,\yy)\neq 0 $ (les autres cas se traitant de fa\c{c}on analogue) l'\'egalit\'e \[ \int_{\substack{\pi^{-1}(U)\cap \{h_{p}^{1}(\xx)\leqslant 1 \\h_{p}^{2}(\yy)\leqslant 1, \; h_{p}^{3}(\zz)\leqslant 1 \}}}\omega_{L,p}(\xx,\yy,\zz)=a(p)\omega_{p}(U) \] est v\'erifi\'ee. Remarquons dans un premier temps que, pour un tel ouvert $ U $, on a\[ a(p)\omega_{p}(U)=\int_{U_{1}\times U_{2}\times U_{3}}\frac{du_{1,p}...du_{n,p}dv_{1,p}...dv_{n,p}dw_{1,p}...dw_{n-1,p}}{|B_{n}(\uu,\vv)|_{p}h_{p}^{1}(\uu) h_{p}^{2}(\vv) h_{p}^{3}(\ww)}. \]
En appliquant trois fois le lemme $ 5.4.5 $ de \cite{Pe}, on obtient alors : \[ a(p)\omega_{p}(U)=\int_{\pi^{-1}(U)}\frac{d\xx d\yy d\hat{\zz}}{|B_{n}(\xx,\yy)|_{p}}=\int_{\pi^{-1}(U)}\omega_{L,p}(\xx,\yy,\zz). \]
\end{proof}
Nous allons \`a pr\'esent \'etablir le lemme suivant dont la d\'emontration est inspir\'ee de \cite[Lemme 3.2]{P-T} et de \cite[Lemme 3.4]{S2} : \begin{lemma}\label{sigmap1}
Soit \begin{multline*}W^{\ast}(r)=\{ (\xx,\yy,\zz)\in (\ZZ_{p}/p^{r})^{3n+3}, \; \xx\not\equiv \0(p), \; \yy\not\equiv \0(p), \; \\ \zz\not\equiv \0(p), \; F(\xx,\yy,\zz)\equiv 0(p^{r}) \}, \end{multline*} et on pose $ N^{\ast}(r)=\card(W^{\ast}(r)) $. Il existe alors un entier $ r_{0} $ tel que pour tout $ r\geqslant r_{0} $ : \[ \int_{\substack{\{(\xx,\yy,\zz)\in \ZZ_{p}^{3n+3}, \; \xx \not\equiv \0(p) \\ \yy\not\equiv \0(p), \; \zz\not\equiv \0(p), \; F(\xx,\yy,\zz)=0\}}}\omega_{L,p}=\frac{N^{\ast}(r)}{p^{r(3n+2)}}. \]
\end{lemma}
\begin{proof}
Soit $ (\xx,\yy,\zz)\in \ZZ_{p}^{3n+3} $. Dans tout ce qui suit, on note $ [\xx,\yy,\zz]_{r}=(\xx,\yy,\zz)\mod p^{r} $. On \'ecrit alors : \begin{align}
\int_{\substack{\{(\xx,\yy,\zz)\in \ZZ_{p}^{3n+3}, \; \xx \not\equiv \0(p) \\ \yy\not\equiv \0(p), \; \zz\not\equiv \0(p), \; F(\xx,\yy,\zz)=0\}}}\omega_{L,p} & =\sum_{\substack{(\xx,\yy,\zz)\mod p^{r} \\ \xx\not\equiv \0(p), \; \yy\not\equiv \0(p), \\ \zz\not\equiv \0(p)}}\int_{\substack{\{(\uu,\vv,\ww)\in \ZZ_{p}^{3n+3}, \\ [\uu,\vv,\ww]_{r}=(\xx,\y,\zz)   \\  F(\uu,\vv,\ww)=0\}}}\omega_{L,p}(\uu,\vv,\ww) \\ & =\sum_{(\xx,\yy,\zz)\in W^{\ast}(r)}\int_{\substack{\{(\uu,\vv,\ww)\in \ZZ_{p}^{3n+3}, \\ [\uu,\vv,\ww]_{r}=(\xx,\y,\zz)   \\  F(\uu,\vv,\ww)=0\}}}\omega_{L,p}(\uu,\vv,\ww).
\end{align}
Puisque $ V $ est lisse, il existe un $ r>0 $ assez grand tel que, pour tout $ (\xx,\yy,\zz)\in \ZZ_{p}^{3n+3} $ tel que $ \xx\not\equiv \0 (p) $, $\yy\not\equiv \0(p)$, $\zz\not\equiv \0(p) $ et $ F(\xx,\yy,\zz)=0 $  : \[ c=\inf_{i,j,k}\lbrace \nu_{p}\left(B_{k}(\xx,\yy)\right), \; \nu_{p}\left(B_{j}'(\xx,\zz)\right), \; \nu_{p}\left(B_{i}''(\yy,\zz)\right)\rbrace \] soit non nul et constant sur la classe d\'efinie par $ (\xx,\yy,\zz) $. On peut supposer que $ r>c $ et que $ c=\nu_{p}\left(B_{n}(\xx,\yy)\right) $. On consid\`ere $ (\uu,\vv,\ww)\in \ZZ_{p}^{3n+3} $ tel que $ [\uu,\vv,\ww]_{r}=(\xx,\y,\zz) $, et $ (\uu',\vv',\ww')\in \ZZ_{p}^{3n+3} $ quelconque. On a alors \begin{multline*} 
F(\uu+\uu',\vv+\vv',\ww+\ww')=F(\uu,\vv,\ww)+\sum_{k=0}^{n}B_{k}(\uu,\vv)w_{k}'+\sum_{j=0}^{n}B_{j}'(\uu,\ww)v_{j}' \\+\sum_{i=0}^{n}B_{i}''(\vv,\ww)u_{i}'+G(\uu,\vv,\ww,\uu',\vv',\ww'), 
\end{multline*}

o\`u $ G(\uu,\vv,\ww,\uu',\vv',\ww') $ est une somme de termes contenant au moins deux facteurs $ u_{i}' $, $ v_{j}' $ ou $ w_{k}' $. Ainsi, on a donc, si $ (\uu',\vv',\ww')\in (p^{r}\ZZ_{p})^{3n+3} $ : \[ F(\uu+\uu',\vv+\vv',\ww+\ww')\equiv F(\uu,\vv,\ww) (p^{r+c}). \] Par cons\'equent, l'image de $ F(\uu,\vv,\ww) $ dans $ \ZZ_{p}/p^{r+c} $ d\'epend uniquement de $ (\uu,\vv,\ww)\mod p^{r} $, on note alors $ F^{\ast}(\xx,\yy,\zz) $ cette image. \\

 Si $ F^{\ast}(\xx,\yy,\zz)\neq 0 $, alors l'int\'egrale \[ \int_{\substack{\{(\uu,\vv,\ww)\in \ZZ_{p}^{3n+3}, \\ [\uu,\vv,\ww]_{r}=(\xx,\y,\zz)   \\  F(\uu,\vv,\ww)=0\}}}\omega_{L,p}(\uu,\vv,\ww) \] est nulle, et l'ensemble \[ \{(\uu,\vv,\ww) \mod p^{r+c} \; | \; [\uu,\vv,\ww]_{r}=(\xx,\y,\zz), \; F(\uu,\vv,\ww)\equiv 0(p^{r+c}) \} \] est vide. \\

Si $ F^{\ast}(\xx,\yy,\zz)= 0 $ alors, par le lemme de Hensel, les applications coordonn\'ees $ X_{0},...,X_{n},Y_{0},...,Y_{n},Z_{0},...,Z_{n-1} $ d\'efinissent un isomorphisme de \[ \{(\uu,\vv,\ww)\in \ZZ_{p}^{3n+3} \; | \; [\uu,\vv,\ww]_{r}=(\xx,\y,\zz), \; F(\uu,\vv,\ww)=0 \} \] sur \[(\xx,\yy,\hat{\zz})+(p^{r}\ZZ_{p})^{3n+2}, \] o\`u $ \hat{\zz}=(z_{0},...,z_{n-1}) $. Par cons\'equent, on a : \begin{multline*}
\int_{\substack{\{(\uu,\vv,\ww)\in \ZZ_{p}^{3n+3}, \\ [\uu,\vv,\ww]_{r}=(\xx,\y,\zz)   \\  F(\uu,\vv,\ww)=0\}}}\omega_{L,p}(\uu,\vv,\ww)  \\ =\int_{(\xx,\yy,\hat{\zz})+(p^{r}\ZZ_{p})^{3n+2}}p^{c}du_{0,p}...du_{n,p}dv_{0,p}...dv_{n,p}dw_{0,p}...dw_{n-1,p} =p^{c-r(3n+2)}.
\end{multline*} 
On a d'autre part, puisque $ F(\uu,\vv,\ww) \mod p^{r+c} $ ne d\'epend que de $ (\xx,\yy,\zz) $ : \begin{multline*}
p^{-(r+c)(3n+2)}\card\{(\uu,\vv,\ww)\mod p^{r+c}\; | \; [\uu,\vv,\ww]_{r}=(\xx,\yy,\zz),\; \\F(\uu,\vv,\ww)\equiv 0(p^{r+c}) \} =p^{-(r+c)(3n+2)}p^{(3n+3)c}  = p^{c-r(3n+2)}.
\end{multline*}
On a donc finalement : \begin{multline*}
\int_{\substack{\{(\xx,\yy,\zz)\in \ZZ_{p}^{3n+3}, \; \xx \not\equiv \0(p) \\ \yy\not\equiv \0(p), \; \zz\not\equiv \0(p), \; F(\xx,\yy,\zz)=0\}}}\omega_{L,p}  =\sum_{\substack{(\xx,\yy,\zz)\in W^{\ast}(r) \\ F^{\ast}(\xx,\yy,\zz)=0}}p^{c-r(3n+2)} \\ = \sum_{\substack{(\xx,\yy,\zz)\in W^{\ast}(r) }}p^{-(r+c)(3n+2)}\card\{(\uu,\vv,\ww)\mod p^{r+c}\; | \; [\uu,\vv,\ww]_{r}=(\xx,\yy,\zz),\;\\ F(\uu,\vv,\ww)\equiv 0(p^{r+c}) \}  = \frac{N^{\ast}(r+c)}{p^{(r+c)(3n+2)}}.
\end{multline*}

D'o\`u le r\'esultat.

\end{proof}
Nous \'etablissons \`a pr\'esent un lemme issus de \cite[Lemme 3.3]{P-T} et \cite[Lemme 3.5]{S2}. 
\begin{lemma}\label{sigmap2}
On a que \[ \int_{\substack{\{(\xx,\yy,\zz)\in \ZZ_{p}^{3n+3}, \; \xx \not\equiv \0(p) \\ \yy\not\equiv \0(p), \; \zz\not\equiv \0(p), \; F(\xx,\yy,\zz)=0\}}}\omega_{L,p}=\left(1-\frac{1}{p^{n}}\right)^{3}\int_{\substack{(\xx,\yy,\zz)\in \ZZ_{p}^{3n+3}\\ F(\xx,\yy,\zz)=0}}\omega_{L,p}, \] et \[ \lim_{r\ra \infty}\frac{N^{\ast}(r)}{p^{r(3n+2)}}=\left(1-\frac{1}{p^{n}}\right)^{3}\sigma_{p}. \]
\end{lemma}
\begin{proof}
Pour d\'emontrer la premi\`ere \'egalit\'e, il suffit de remarquer que : 

\[ \omega_{L,p}(p\xx,\yy,\zz)=\omega_{L,p}(\xx,p\yy,\zz)= \omega_{L,p}(\xx,\yy,p\zz)=p^{-n} \omega_{L,p}(\xx,\yy,\zz), \]
\[\omega_{L,p}(p\xx,p\yy,\zz)=\omega_{L,p}(p\xx,\yy,p\zz)= \omega_{L,p}(\xx,p\yy,p\zz)=p^{-2n} \omega_{L,p}(\xx,\yy,\zz), \]
\[ \omega_{L,p}(p\xx,p\yy,p\zz)=p^{-3n} \omega_{L,p}(\xx,\yy,\zz).\]

En effet, on a alors \begin{align*}
\int_{\substack{\{(\xx,\yy,\zz)\in \ZZ_{p}^{3n+3}, \; \xx \not\equiv \0(p) \\ \yy\not\equiv \0(p), \; \zz\not\equiv \0(p), \; F(\xx,\yy,\zz)=0\}}}\omega_{L,p} &= \left(1-\frac{3}{p^{n}}+\frac{3}{p^{2n}}-\frac{1}{p^{3n}}\right)\int_{\substack{(\xx,\yy,\zz)\in \ZZ_{p}^{3n+3}\\ F(\xx,\yy,\zz)=0}}\omega_{L,p}  \\ & =\left(1-\frac{1}{p^{n}}\right)^{3}\int_{\substack{(\xx,\yy,\zz)\in \ZZ_{p}^{3n+3}\\ F(\xx,\yy,\zz)=0}}\omega_{L,p}. \end{align*}
 
Nous avons vu par ailleurs que (cf. lemme \ref{sigmap}) : \[ \sigma_{p}=\lim_{r\ra \infty}\frac{N(r)}{p^{r(3n+2)}}, \] o\`u $ N(r)=\card\{(\xx,\yy,\zz) \mod p^{r}\; | \;  F(\xx,\yy,\zz)\equiv 0(p^{r}) \} $. On consid\`ere ensuite pour $ r>0 $ fix\'e et pour des entiers $ i,j,k $ tels que $ r\geqslant i+j+k $ : \begin{multline*} \widetilde{N}(i,j,k)=\card\{(\xx,\yy,\zz) \; | \; \xx\in (p^{i}\ZZ_{p}/p^{r})^{n+1}, \; \xx\not\equiv \0 (p^{i+1}), \\ \; \yy\in (p^{j}\ZZ_{p}/p^{r})^{n+1},\;  \yy\not\equiv \0 (p^{j+1}), \; \zz\in (p^{k}\ZZ_{p}/p^{r})^{n+1}, \; \\ \zz\not\equiv \0 (p^{k+1}), \; F(\xx,\yy,\zz)\equiv 0(p^{r}) \}. \end{multline*}
On a alors \begin{multline*}  
 \widetilde{N}(i,j,k)=\card\{ (\xx \mod p^{r-i}, \yy \mod p^{r-j}, \zz \mod p^{r-k}) \; | \; \xx \not\equiv \0 (p), \; \\  \yy \not\equiv \0 (p), \; \zz \not\equiv \0 (p), \; F(\xx,\yy,\zz)\equiv \0 (p^{r-i-j-k})\},
 \end{multline*}
 
 et on en d\'eduit : 
 \begin{align*}
  \widetilde{N}(i,j,k)& =p^{(n+1)(i+j)}p^{(n+1)(j+k)}p^{(n+1)(i+k)}N^{\ast}(r-i-j-k) \\ &=p^{2(n+1)(i+j+k)}N^{\ast}(r-i-j-k).
 \end{align*}

Soit $ r_{0} $ un entier comme dans le lemme pr\'ec\'edent, et soit \[ I(r)=\{ (i,j,k) \; | \; r-r_{0}<i+j+k\leqslant r-r_{0}+3\}. \] On remarque que : \begin{multline*}
N(r)=\sum_{\substack{i,j,k\geqslant 0 \\ r-i+j+k\geqslant r_{0}}}\widetilde{N}(i,j,k) \\ + O\left( \sum_{(i,j,k)\in I(r)}\card\{(\xx,\yy,\zz)\mod p^{r} \; | \; \xx\equiv \0 (p^{i}),\; \yy\equiv \0 (p^{j}),\; \zz\equiv \0 (p^{k}) \}\right).
\end{multline*}
Or, \begin{multline*} \sum_{(i,j,k)\in I(r)}\card\{(\xx,\yy,\zz)\mod p^{r} \; | \; \xx\equiv \0 (p^{i}),\; \yy\equiv \0 (p^{j}),\; \zz\equiv \0 (p^{k}) \} \\ \ll_{r_{0}} r^{2}\max_{(i,j,k)\in I(r)}p^{(n+1)(r-i)+(n+1)(r-j)+(n+1)(r-k)  }\\ \ll_{r_{0}} r^{2}p^{(3n+2)r}\max_{(i,j,k)\in I(r)}p^{r-2(i+j+k)} \\ \ll_{p,r_{0}}r^{2}p^{(3n+2)r}p^{-r}. 
\end{multline*}
On a donc \[ N(r)=\sum_{\substack{i,j,k\geqslant 0 \\ r-i+j+k\geqslant r_{0}}}p^{2(n+1)(i+j+k)} +O(r^{2}p^{(3n+1)r}). \] Puisque la somme est restreinte aux $ (i,j,k) $ tels que $ r_{0}\leqslant r-i-j-k $, on a alors par le lemme pr\'ec\'edent : \[ \frac{N^{\ast}(r-i-j-k)}{p^{(r-i-j-k)(3n+2)}}=\frac{N^{\ast}(r)}{p^{r(3n+2)}}. \] On a donc \[ N^{\ast}(r-i-j-k)=N^{\ast}(r)p^{-(3n+2)(i+j+k)}, \] et donc \[N(r)=\left( \sum_{r_{0}\leqslant r-i-j-k}p^{-(i+j+k)n}\right)N^{\ast}(r)+O(r^{2}p^{(3n+1)r}). \] On obtient donc finalement  \[ \sigma_{p}=\lim_{r\ra \infty} p^{-(3n+2)r}N(r)=\left(1-\frac{1}{p^{n}}\right)^{-3}\lim_{r\ra \infty}p^{-(3n+2)r}N^{\ast}(r). \]

\end{proof}
D'apr\`es ce lemme, on a donc : \[ \sigma_{p}=\int_{\substack{(\xx,\yy,\zz)\in \ZZ_{p}^{3n+3}\\ F(\xx,\yy,\zz)=0}}\omega_{L,p}. \] 
En utilisant le lemme \ref{sigmap0} on a ainsi : \begin{equation}\label{taup}\tau_{p}(V(\QQ_{p}))=\left(1-\frac{1}{p^{n}}\right)^{3}\sigma_{p}.
\end{equation}

\subsection{Conclusion}

Rappelons que la formule asymptotique conjectur\'ee par Peyre dans \cite{Pe} pour le nombre $ \mathcal{N}_{U}(B) $ de points de hauteur born\'ee  par $ B $ sur l'ouvert $ U $ de Zariski de la vari\'et\'e $ V $ (pour la hauteur associ\'ee au fibr\'e anticanonique $ \omega_{V}^{-1} $) est : \begin{equation}\alpha(V)\beta(V)\tau_{H}(V)B\log(B)^{\rg(\Pic(V))-1} \end{equation} 
o\`u \[ \alpha(V)= \frac{1}{(\rg(\Pic(X))-1)!}\int_{\Lambda_{\eff}^{1}(V)^{\vee}}e^{-\langle \omega_{V}^{-1},y\rangle}dy, \] \[ \Lambda_{\eff}^{1}(V)^{\vee}=\{y\in \Pic(V)\otimes \RR^{\vee}\; |\; \forall x\in \Lambda_{\eff}^{1}(V), \langle x,y\rangle\geqslant 0 \} \] et \[ \beta(V)=\card(H^{1}(\QQ,\Pic(\overline{V}))), \] \[ \tau_{H}(V)=\prod_{\nu\in \Val(\QQ)}\tau_{\nu}(V(\QQ_{\nu})). \]
Or, dans le cas pr\'esent on a \[ \Pic(V)=\ZZ.\OO_{V}(1,0,0)\oplus\ZZ.\OO_{V}(0,1,0)\oplus\ZZ.\OO_{V}(0,0,1)\simeq \ZZ^{3}, \;\; \rg(\Pic(V))=3, \] \[ \omega_{V}^{-1}=\OO_{V}(n,n,n), \] \[\Lambda_{\eff}^{1}(V)=\ZZ.\OO_{V}(1,0,0)\oplus\ZZ.\OO_{V}(0,1,0)\oplus\ZZ.\OO_{V}(0,0,1)\simeq (\RR^{+})^{3}. \] On a par cons\'equent : \begin{align*}
\alpha(V)=\frac{1}{2}\int_{[0,+\infty[^{3}}e^{-nt_{1}-nt_{2}-nt_{3}}dt_{1}dt_{2}dt_{3}=\frac{1}{2n^{3}}. 
\end{align*} 
D'autre part $ \Pic(\overline{V})\simeq \ZZ^{3} $, et le groupe de Galois $ \Gal(\overline{\QQ}/\QQ) $ agit trivialement sur $ \Pic(\overline{V}) $, on a donc \[ \beta(V)=1. \] Par ailleurs, d'apr\`es ce qui a \'et\'e vu dans les sections pr\'ec\'edentes, on a \[ \prod_{p\in \mathcal{P}}\tau_{p}(V(\QQ_{p}))=\mathfrak{S}\prod_{p\in \mathcal{P}}\left(1-\frac{1}{p^{n}}\right)^{3} \] et \[ \tau_{\infty}(V(\RR))=\frac{n^{3}}{8}J. \]
Ainsi on a ici \begin{align*} \alpha(V)\beta(V)\tau_{H}(V)B\log(B)^{\rg(\Pic(V))-1}& =\frac{1}{16}\mathfrak{S}J\prod_{p\in \mathcal{P}}\left(1-\frac{1}{p^{n}}\right)^{3}B\log(B)^{2}\\ &=\frac{1}{16}\sigma'B\log(B)^{2}, \end{align*} et on retrouve bien la formule de la proposition \ref{conclusion}.

\end{document}